\documentclass[12pt]{amsart}
\usepackage{amsmath, amsthm, amssymb, hyperref, color}
\usepackage{MnSymbol} 
\usepackage{graphicx}
\usepackage[margin=1.1in]{geometry}
\usepackage{verbatim}
\usepackage{chemfig, chemnum}
\usepackage{tikz}
\usetikzlibrary{arrows.meta}
\usepackage[linesnumbered,lined,commentsnumbered]{algorithm2e}
\usepackage{caption}
\usepackage{subcaption}
\usepackage{float}
\usepackage[capitalize]{cleveref}
\usepackage{tikz-cd}
\usepackage{enumitem}
 \makeatletter
\newcommand{\longdashrightarrow}[1][2.5pt]{%
  \settowidth{\@tempdima}{$\longrightarrow$}\longrightarrow
  \makebox[-\@tempdima]{\hskip-0.5ex\color{white}\rule[0.5ex]{#1}{1pt}}
    \phantom{\longrightarrow}
  \makebox[-\@tempdima]{\hskip-2.8ex\color{white}\rule[0.5ex]{#1}{1pt}}
  \phantom{\longrightarrow}
}
\newcommand{\xdashrightarrow}[2][]{\ext@arrow 0359\rightarrowfill@@{#1}{#2}}
\makeatother
\newcommand{\arxiv}[1]{\href{https://arxiv.org/abs/#1}{\textup{\texttt{arXiv:#1}}}}
\def\ttau{\tilde \tau}

\def\ttt{\tilde t}
\def\td{\tilde d}
\def\tI{{\tilde \I}}
\def\tC{{\tilde C}}
\def\PGL{{\rm PGL}}

\def\RR{{\mathcal{R}}}

\def\R{{\mathbb R}}
\def\C{{\mathbb C}}
\def\Z{{\mathbb Z}}
\def\F{{\mathbb F}}
\def\AA{{\mathbb A}}

\def\Mat{{\rm Mat}}
\def\tree{{\rm tree}}
\def\Sym{{\rm Sym}}
\def\detp{{\det{}\!'}}
\def\Proj{{\operatorname{Proj}}}
\def\dlog{{\operatorname{dlog}}}
\def\Spec{\operatorname{Spec}}
\def\Hom{\operatorname{Hom}}

\def\Trop{{\operatorname{trop}}}

\def\D{{\mathcal{D}}}
\def\PT{{\rm PT}}
\def\sp{\operatorname{span}}

\def\I{{\mathcal{I}}}
\def\lk{{\rm lk}}

\def\a{{\mathbf{a}}}
\def\y{{\mathbf{y}}}

\def\Fl{{\rm Fl}}
\def\Gr{{\rm Gr}}
\def\hGr{\widehat{\Gr}}

\def\CC{{\mathcal{C}}}
\def\M{{\mathcal{M}}}
\def\bM{\overline{\M}}
\def\tM{{\widetilde{\M}}}
\def\GL{{\rm GL}}

\def\Vol{{\rm Vol}}
\def\val{{\rm val}}

\def\pt{{\rm pt}}
\def\Res{{\rm Res}}
\renewcommand{\P}{{\mathbb P}}

\def\F{{\mathbb F}}

\def\oU{{U}}

\def\usigma{{\underline{\sigma}}}
\def\s{{\underline{s}}}

\def\tK{{\Lambda}}
\def\lat{{\Lambda}}
\def\tlambda{\tilde \lambda}
\def\hLambda{{\hat \Lambda}}
\def\tmu{\tilde \mu}
\def\SL{{\rm SL}}
\def\SO{{\rm SO}}
\def\Spin{{\rm Spin}}
\def\J{{\mathcal{J}}}
\def\tJ{{\tilde \J}}
\def\trop{{\rm trop}}
\def\diag{{\rm diag}}

\newtheorem{conjecture}{Conjecture}
\newtheorem{theorem}[conjecture]{Theorem}
\newtheorem{lemm}[conjecture]{Lemma}
\newtheorem{proposition}[conjecture]{Proposition}
\newtheorem{corollary}[conjecture]{Corollary}
\newtheorem{definition}[conjecture]{Definition}
\newtheorem{exercise}[conjecture]{Exercise}
\newtheorem{problem}[conjecture]{Problem}
\newtheorem{assumption}[conjecture]{Assumption}
\newtheorem{remark}[conjecture]{Remark}

\newtheorem{example}[conjecture]{Example}

\numberwithin{conjecture}{section}
\numberwithin{equation}{section}

\title{Moduli spaces in positive geometry}
\author{Thomas Lam}
\address{Department of Mathematics, University of Michigan \\ 2074 East Hall, 530 Church Street \\ Ann Arbor, MI 48109-1043, USA}

\begin{document}
\begin{abstract}
These are lecture notes for five lectures given at MPI Leipzig in May 2024.  We study the moduli space $\M_{0,n}$ of $n$ distinct points on $\P^1$ as a positive geometry and a binary geometry.  We develop mathematical formalism to study Cachazo-He-Yuan's scattering equations and the associated scalar and Yang-Mills amplitudes.  We discuss open superstring amplitudes and relations to tropical geometry.
\end{abstract}
\maketitle
\section*{Introduction}

The new field of positive geometry is developing at the interface of combinatorial algebraic geometry in mathematics and scattering amplitudes in physics.  The purpose of these notes is to give an introduction to some of the ideas in this area.  Our main target audience is a mathematics graduate student with some experience in geometric combinatorics or algebraic geometry.  At the urging of the organizers of the workshop ``Combinatorial algebraic geometry from physics", I have leant towards an algebro-geometric approach, though I have kept the language as concrete as possible.  The text should be understandable to a motivated student who has taken one graduate class in algebraic geometry.  No physics background is assumed.

\medskip

When I started writing these notes, it became clear very quickly that I could not give a comprehensive survey of the area, such is the breadth of the ideas involved.  I narrowed the scope as follows.

On the mathematical side, we focus on the space $\M_{0,n}$.  This is the configuration space of $n$ distinct points on the projective line $\P^1$.  The space $\M_{0,n}$ is a central example in combinatorial geometry, as it is simultaneously:
\begin{enumerate}
\item a quotient of an open subset of the Grassmannian, or of a matrix space,
\item a hyperplane arrangement complement,
\item (naturally) an open subset of a toric variety,
\item a very affine variety,
\item a positive geometry,
\item a binary geometry,
\item a moduli space,
\item a cluster configuration space.
\end{enumerate}


On the physics side, we consider three classes of amplitudes, all with tight connections to combinatorial algebraic geometry: 
\begin{enumerate}
\item
the planar scalar $\phi^3$-amplitude, 
\item
the planar gluon-amplitude for Yang-Mills theory, and 
\item 
the open superstring-amplitude.
\end{enumerate}
In all cases, we focus on the kinematic aspects of the amplitudes, and only consider \emph{tree-level} amplitudes for massless scattering.  Our aim is to give \emph{precise} definitions of various functions on kinematic space, with \emph{heuristic} explanations of the relations to quantum field theory and string theory.  Since $\M_{0,n}$ is the main mathematical object we consider, our approach to amplitudes leans heavily on the Cachazo-He-Yuan \emph{scattering equations} \cite{CHYKLT,CHYmassless,CHYarbitrary}.  For gluon amplitudes, we borrow also from twistor string theory \cite{Wit,RSV}.  We regret the absence of loop amplitudes, of supersymmetry, and of momentum-twistors, all of which are the bread and butter of modern amplitudes theory.  

A recurring theme in the notes is to ask whether amplitudes can be generalized beyond $\M_{0,n}$ to more general settings, such as that of binary geometries, positive geometries, and very affine varieties.  In \cite{Lamfuture}, we develop these amplitudes for hyperplane arrangement complements.

\medskip
\noindent
{\bf Overview.}  We now give a more detailed survey of the contents of this work.  In \cref{sec:M0n}, we define and study the space $\M_{0,n}$.  We use crucially a collection of cross-ratios $u_{ij}$ called \emph{dihedral coordinates} \cite{KN,Bro}.  These cross-ratios are labeled by the diagonals, but not sides, of a $n$-gon.  They form a basis of the character lattice of the intrinsic torus of $\M_{0,n}$ (\cref{prop:uijchar}), and define an affine variety $\tM_{0,n}$ (\cref{sec:MD}) that is a partial compactification of $\M_{0,n}$.  The strata of $\tM_{0,n}$ are labeled by subdivisions of the $n$-gon, or equivalently, by faces of the associahedron polytope.  Thus, $\tM_{0,n}$ is an affine patch in which we can see the combinatorics and geometry of the boundary stratification of $\bM_{0,n}$.

The dihedral coordinates $u_{ij}$ depend on a choice of a connected component $(\M_{0,n})_{>0}$ of the real points $\M_{0,n}(\R)$.  We call $(\M_{0,n})_{>0}$ the positive part of $\M_{0,n}$.  The analytic closure $(\M_{0,n})_{\geq 0}$ is a positive geometry (\cref{def:posgeom}) inside $\bM_{0,n}$, and we discuss properties of the canonical form $\Omega_{0,n} = \Omega((\M_{0,n})_{\geq 0})$, a holomorphic top-form on $\M_{0,n}$.

\medskip

In \cref{sec:binary}, we formalize ideas from \cite{AHLT} and define \emph{binary geometries}.  The dihedral coordinates $u_{ij}$ of $\M_{0,n}$ satisfy \emph{binary} $u$-equations.  For example, with $n = 5$, we have five dihedral coordinates $u_{13},u_{14},u_{24},u_{25},u_{35}$, satisfying the five relations
\begin{align}\label{eq:ueq5}
\begin{split}
u_{13} + u_{24}u_{25} &=1 \\
u_{24} + u_{13}u_{35} &=1 \\
u_{35} + u_{14}u_{24} &=1 \\
u_{14} + u_{25}u_{35} &=1 \\
u_{25} + u_{13}u_{14} &=1 
\end{split}
\end{align}
A key feature of dihedral coordinates is that if $u_{ij} = 0$ then we have $u_{kl} = 1$ for all diagonals $(k,l)$ crossing $(i,j)$.  This is the ``binary" nature of the $u$-equations.  In \cref{def:binary} we define binary geometries, a class of affine varieties with distinguished coordinates $u_i$ satisfying equations of a form similar to \eqref{eq:ueq5}.  We classify (\cref{thm:onedim}) one-dimensional binary geometries and show (\cref{thm:pseudo}) that the face poset of a binary geometry is a simplicial complex that is a pseudomanifold.  We suggest a number of potential directions for further exploration of binary geometries.

\medskip

\cref{sec:SE} is the core of the lecture series.  We start with an informal discussion of scattering amplitudes in quantum field theory, leading to a precise definition of the \emph{kinematic space} $K_n$ for $n$-particle scattering (\cref{def:Kn}).  The link between kinematics and the geometry of $\M_{0,n}$ is made possible by a natural isomorphism between the integer points $K_n(\Z)$ and the character lattice $\Lambda(\M_{0,n})$ (\cref{prop:Kniso}).  The celebrated \emph{scattering equations} of Cachazo-He-Yuan were a revolution in the construction of tree-level amplitudes, and we interpret them as a \emph{scattering correspondence} (see \eqref{eq:scatcorr})
\begin{equation}\label{eq:corr}
\I \subset K_n \times \M_{0,n}.
\end{equation}
This correspondence is cut out by the critical points of the scattering potential, also called a likelihood function in algebraic statistics, or master function in the theory of hyperplane arrangements.

At the heart of our philosophy is that the correspondence \eqref{eq:corr} is \emph{canonically} associated to the very affine variety $\M_{0,n}$.  The CHY scalar amplitudes are obtained by summing a rational function over the pre-images of a point $\s \in K_n$ under the map $\I \to K_n$.  We make CHY's definition more conceptual in two ways.  First, we define in \cref{def:ampl} an amplitude $A(\Omega|\Omega')$ that depends on two rational top-forms $\Omega,\Omega'$, which we expect to be taken to be canonical forms of positive geometries.  Second, we define a scattering form on $K_n$ as a push-pull of the canonical form of $\M_{0,n}$.  This gives a conceptual approach to the scattering form of \cite{ABHY}.

\medskip
In \cref{sec:4dim}, we specialize to four-dimensional space-time and consider amplitudes for gluons (massless spin one particles).  We introduce spinor coordinates on kinematic space.  Mathematically, these coordinates come from the exceptional isomorphism $\Spin(4) \simeq \SL(2) \times \SL(2)$ involving the spin group $\Spin(4) \to \SO(4)$ covering the complexification of the four-dimensional Lorentz group.  Spinor coordinates capture in a beautiful way the special features of four-dimensional space-time and also encode the polarization vectors that appear in gluon scattering.

Restricting \eqref{eq:corr} to the subvariety of four-dimensional kinematics $K_n^4 \subset K_n$, the correspondence breaks up into $(n-3)$ irreducible components $J_1,\ldots,J_{n-3}$, called ``sectors" (\cref{prop:Jd}).  These sectors were discovered in Witten's twistor string theory \cite{Wit} and further studied by Roiban-Spradlin-Volovich \cite{RSV} and Cachazo-He-Yuan \cite{CHYthree}.  The sector decomposition arises from the interpretation of \eqref{eq:corr} as an incidence subvariety $\tI \subset \M_{0,n}(\P^3,n-2)$, where $ \M_{0,n}(\P^3,n-2)$ denotes the moduli space of rational $n$-pointed maps to $\P^3$ of degree $n-2$, and $\tI$ is the subvariety of maps with image in a quadric.  We explain a construction of \cite{ABCT}, which reinterprets the scattering correspondence of the $d$-th sector as a subvariety of the Grassmannian $\Gr(d+1,n)$.  This is one of the starting points of the Grassmannian formulae for scattering amplitudes.

\medskip
In \cref{sec:string}, we study the tree-level $n$-point open superstring amplitude as an integral over $(\M_{0,n})_{>0}$, explaining how field theory amplitudes can be obtained as limits of string amplitudes as strings turn into particles.  We consider a general class of stringy integrals \cite{AHLstringy}, also called Euler integrals in \cite{MHMT}, which take the form 
$$
I(\tau,c) = \int_{U_{>0}} \phi_{\tau,c} \, \Omega = \int_{\R_{>0}^d}  \frac{y_1^{\tau_1}\cdots y_n^{\tau_n}}{p_1^{c_1} \cdots p_r^{c_r}} \frac{dy_1}{y_1} \wedge \cdots \wedge \frac{dy_n}{y_n},
$$
where $p_i=p_i(\y)$ are polynomials in the variables $y_1,\ldots,g_n$.
Associated to such an integral function is a (very affine) subvariety $U$ of a toric variety, and the integral is taken over a positive part $U_{>0} \subset U(\R)$.  The integrand $\phi_{\tau,c}$ is the natural potential function of $U$, and the condition that the integral $I(\tau,c)$ converges cuts out a class of bounded characters of the character lattice $\Lambda(U)$.  In the case of the string amplitude itself, we have $U = \M_{0,n}$ and we recover the dihedral coordinates of \cref{sec:M0n} in this way.  

We explain that the field theory limit of a string integral is naturally related to tropicalization.  Indeed, the (positive) tropicalization $\trop(\phi_{\tau,c})$ of the potential function controls the field theory limit of $I(\tau,c)$.  In this way, the Feynman diagrams of $\phi^3$-theory (which are cubic trees) naturally appear, and we have the following satisfying ``commutative diagram":
\[ \begin{tikzcd}
\M_{0,n} \arrow[swap]{d}{{\rm tropicalization}} \arrow[<->]{r}{} & \mbox{string amplitude} \int_{(\M_{0,n})_{>0}} \phi_X \Omega \arrow{d}{{\rm field} \;  {\rm theory} \; {\rm limit}} \\%
\mbox{cubic trees}  \arrow[<->]{r}{} & \mbox{Feynman amplitude} \sum_{{\rm trees} \; T} \prod_{e \in E(T)} \frac{1}{X_e} 
\end{tikzcd}
\]

\medskip
In \cref{sec:veryaffine}, we review basic definitions for very affine varieties, and we speculate on possible definitions of amplitudes in that setting.

\medskip
\noindent
{\bf Exercises and Problems.} There are exercises at the end of each lecture, which we hope will be useful for a student reader.  ``Problems" are more open-ended: some of these are conjectures or open problems, and some are statements from the physics literature that I do not know proofs of in the mathematics literature.

\medskip

\noindent
{\bf Acknowledgements.} 
We acknowledge support from the National Science Foundation under DMS-1953852 and DMS-2348799, and from the Simons Foundation.

We thank the organizers, Dmitrii Pavlov, Bernd Sturmfels, and Simon Telen, of the workshop ``Combinatorial algebraic geometry from physics" for the invitation to give these lectures.  I thank the participants of the workshop for their encouragement and enthusiasm.  I thank Daniele Agostini, Igor Makhlin, Dmitrii Pavlov, Mark Spradlin, Emanuele Ventura for comments and corrections on earlier versions of this work.

My understanding of this subject owes a lot to my many collaborators.  The content of these lectures rely especially on my joint work with Nima Arkani-Hamed and Song He.

\tableofcontents

\section{Positive geometry of $\M_{0,n}$}\label{sec:M0n}
Let $[n]:=\{1,2,\ldots,n\}$.
\subsection{Matrices, Moduli space, and the Grassmannian}\label{sec:hyper}
We represent points $\sigma \in \P^1 = \C \cup \{\infty\}$ by non-zero vectors $[x_0:x_1]$, with $[x_0:x_1] \sim [\alpha x_0: \alpha x_1]$ for a nonzero scalar $\alpha$.  The general linear group $\GL(2)$ of invertible $2 \times 2$ matrices acts on $\C^2$ and thus on $\P^1$.  This action factors through the projective linear group $\PGL(2)$ which is the quotient group of $\GL(2)$ by the scalar matrices.

\begin{definition}
For $n \geq 3$, the space $\M_{0,n}$ is the moduli space of $n$ distinct points $\sigma_1,\sigma_2,\ldots,\sigma_n$ on $\P^1$ considered up to the simultaneous action of $\PGL(2)$.
\end{definition}

We may represent a point $\usigma \in \M_{0,n}$ as a $2 \times n$ matrix (with no zero columns)
$$
\begin{bmatrix}
x_{11} & x_{12} & \cdots & x_{1n} \\
x_{21} & x_{22} & \cdots & x_{2n}
\end{bmatrix}
$$
where the $i$-th column represents the point $\sigma_i \in \P^1$.  

Let $\Mat_{2,n}$ denote the space of $2 \times n$ matrices.  Let $\Mat_{2,n}^\circ \subset \Mat_{2,n}$ denote the subset where all $2 \times 2$ minors are nonzero.
The group $\GL(2)$ acts on $\Mat_{2,n}^\circ$ on the left as row operations.  The group $T \cong (\C^\times)^n$ acts on $\Mat_{2,n}^\circ$ (on the right) by scaling the columns.  Since one group acts on the left and the other on the right, the action of the two groups commute, and we obtain an action of the product group $\GL(2) \times T$.  Let $Z \subset \GL(2)$ denote the subgroup of scalar matrices, and let $D = \{(t,t,\ldots,t)\} \subset T$ denote the subgroup where all entries are equal.  Then the action of $Z$ and $D$ on $\Mat_{2,n}^\circ$ are the same, so the action of $\GL(2) \times T$ factors through the quotient group where $Z$ and $D$ are identified.  Equivalently, the action factors through the action of the group
$$
\PGL(2) \times T \cong \GL(2) \times T'
$$
where $T' = T/D \cong (\C^\times)^{n-1}$.  We leave the proof of the following result as an exercise for the reader.

\begin{lemm}
The action of $\PGL(2) \times T \cong \GL(2) \times T'$ on $\Mat_{2,n}^\circ$ is free.
\end{lemm}

We have the following commutative diagram of quotient maps:

\begin{equation}\label{eq:M0ncomm} \begin{tikzcd}
&\Mat_{2,n}^\circ \arrow{rd}{T} \arrow[swap]{ld}{\GL(2)} & \\%
\Gr^\circ(2,n) \arrow[swap]{rd}{T'} && ((\P^1)^n)^\circ \arrow{ld}{\PGL(2)}\\
&\M_{0,n}&
\end{tikzcd}
\end{equation}
Here, $\Gr^\circ(2,n) \subset \Gr(2,n)$ is the locus inside the Grassmannian of $2$-planes where all Pl\"ucker coordinates are non-vanishing, and $((\P^1)^n)^\circ$ denotes the space of $n$ distinct points in $\P^1$, obtained by removing the ``diagonals" from $(\P^1)^n$.
Since $\M_{0,n}$ is the quotient of the smooth variety $\Mat_{2,n}^\circ$ by the free action of the algebraic group $\PGL(2) \times T$, we obtain the following.
\begin{proposition}\label{prop:M0n}
$\M_{0,n}$ is a smooth complex algebraic variety of dimension $n-3$.
\end{proposition}
%

The group $\PGL(2)$ acts simply-transitively on the space of 3 distinct ordered points in $\P^1$ (\cref{ex:PGL}).  Thus, the action of $\PGL(2)$ allows us to place $(\sigma_1,\sigma_2,\sigma_n)$ at $(0,1,\infty)$ and $\sigma_3,\ldots,\sigma_{n-1}$ are coordinates on $\M_{0,n}$.  In particular, $\M_{0,3}$ is a single point and $\M_{0,4} \cong \P^1 \setminus \{0,1,\infty\}$.

In general, $\M_{0,n}$ can be identified with the hyperplane arrangement complement in $\C^{n-3}$ for the $\binom{n}{2} - n$ hyperplanes
\begin{align*}
\{\sigma_i = 0 \mid i = 3,4,\ldots,n-1\} &\bigcup \{\sigma_i =1  \mid i = 3,4,\ldots,n-1\} \\ 
&\bigcup \{\sigma_i -\sigma_j =0 \mid 3 \leq i < j \leq  n-1\}.
\end{align*}

\subsection{Cross ratios}
For $i,j \in [n]$, we write $(ij)$ to denote the $2 \times 2$ minor of a $2 \times n$ matrix, using the columns indexed $i$ and $j$.  For two distinct points $\sigma,\sigma' \in \P^1$, the expression $\sigma - \sigma'$ is simply the difference of two complex numbers; if one of $\sigma$ or $\sigma'$ is $\infty$ then it is taken to be $\pm \infty$.  We also use the notation $\sigma_{ab}:=(\sigma_a - \sigma_b)$.

Using the action of $\PGL(2) \times T$ we can put a $2 \times n$ matrix representing a point $\usigma \in \M_{0,n}$ into the form
$$
\begin{bmatrix} 1 & 1 & \cdots & 1 \\
\sigma_1 & \sigma_2 & \cdots & \sigma_n
\end{bmatrix}.
$$
In this form, we have $(ij) = \sigma_j - \sigma_i$.

\begin{definition}
For distinct $i,j,k,l \in [n]$, define the \emph{cross ratio}
$$
[ij|kl] := \frac{(ik)(jl)}{(il)(jk)} = \frac{(\sigma_i - \sigma_k)(\sigma_j-\sigma_l)}{(\sigma_i - \sigma_l)(\sigma_j-\sigma_k)}
$$
\end{definition}

\begin{lemm}\label{lem:cr}
The cross ratio $[ij|kl]$ is a regular function on $\M_{0,n}$, satisfying the identities
\begin{align*}
&[ij|kl] = 1-[ik|jl], \qquad [ij|lk]=[ij|kl]^{-1} = [ji|kl], \\
& [ij|kl]=[kl|ij]=[ji|lk]=[lk|ji].
\end{align*}
\end{lemm}

When $n = 4$, any of the cross-ratios $[ij|kl]$ with $(i,j,k,l)$ a permutation of $1,2,3,4$, defines an isomorphism
$
\M_{0,4} \cong \P^1 \setminus \{0,1,\infty\}.
$
More generally, for any $n$, we have the following.

\begin{proposition}\label{prop:CR}
The cross-ratios define an embedding
\begin{equation}\label{eq:iota}
\iota: \M_{0,n} \hookrightarrow (\P^1 \setminus \{0,1,\infty\})^{\binom{[n]}{4}}.
\end{equation}
\end{proposition}
\begin{proof}
Using the $\PGL(2)$ action, we place $(\sigma_1,\sigma_2,\sigma_n)$ at $(0,1,\infty)$.  The cross ratio $[1n|j2] = \sigma_j$ then determines $\sigma_j$.  This shows that \eqref{eq:iota} is injective.  Let $W \subset (\P^1 \setminus \{0,1,\infty\})^{\binom{[n]}{4}}$ be the closed subvariety cut out by the ideal generated by the relations in \cref{lem:cr}.  We define a morphism $\pi: W \to \M_{0,n}$ by sending $(\sigma_1,\sigma_2,\sigma_n)$ to $(0,1,\infty)$ and $\sigma_j$ to $[1n|j2]$.  The relations of \cref{lem:cr} can then be used to show that $\pi$ is inverse to $\iota$.
\end{proof}

\subsection{Dihedral coordinates}
Define the \emph{dihedral coordinates}
\begin{equation}\label{eq:defuij}
u_{ij} := [i ,i+1|j+1,j] = \frac{(\sigma_{i}-\sigma_{j+1})(\sigma_{i+1}-\sigma_j)}{(\sigma_{i}-\sigma_{j})(\sigma_{i+1}-\sigma_{j+1})}
\end{equation}
for $i,j \in [n]$ such that $i-j \notin \{-1,0,1\} \mod n$, that is $i,j$ are not equal or adjacent modulo $n$.  We have that $u_{ij} = u_{ji}$, and thus the cross-ratios $u_{ij}$ are labeled by the diagonals, or chords, (but not the sides) of a $n$-gon with vertices $1,2,\ldots,n$.    Let $\diag_n$ denote the set of diagonals of an $n$-gon.  There are $|\diag_n| = \binom{n}{2}-n$ dihedral coordinates.

It is easy to check (\cref{ex:Rij}) that the $u_{ij}$ satisfy the following relations:
\begin{equation}\label{eq:Rij}
R_{ij}:= u_{ij} + \prod_{(k,l) \text{ crosses } (i,j)} u_{kl} - 1 = 0.
\end{equation}
A key feature of dihedral coordinates is that if $u_{ij} = 0$ then we have $u_{kl} = 1$ for all $(k,l)$ crossing $(i,j)$.  This follows from the relations $R_{kl} = 0$.

For example, for $n = 4$, we have two dihedral coordinates $u_{13}, u_{24}$ satisfying the relation 
$$
u_{13}+u_{24} = 1.
$$
For $n = 5$, the relations $R_{ij}$ are given in \eqref{eq:ueq5}. 

Let $T_U = \Spec(\C[u_{ij},u_{ij}^{-1}])$ be the $(\binom{n}{2}-n)$-dimensional torus with coordinates $u_{ij}$.  Define the closed subvariety $\oU \subset T_U$ cut out by the relations \eqref{eq:Rij}.

\begin{theorem}[\cite{Bro,AHLcluster}]\label{thm:U}
The dihedral coordinates give a closed embedding $\M_{0,n} \hookrightarrow T_U$ with image equal to $\oU$.  The ideal of $\oU$ in $T_U$ is equal to the ideal $I = (R_{ij})$ generated by the relations \eqref{eq:Rij}. 
\end{theorem}

In particular, $\M_{0,n}$ is isomorphic (as a scheme) to $ \Spec(\C[u_{ij},u_{ij}^{-1}]/I)$.  In contrast to \cref{prop:CR}, \cref{thm:U} identifies $\M_{0,n}$ naturally with a closed subvariety of a \emph{torus}.  

For two subsets $A,C \subset [n]$, let $u_{A,C} = \prod_{a\in A, c\in C} u_{ac}$.  Let $(A,B,C,D)$ be a decomposition of $[n]$ into four non-empty cyclic intervals, in cyclic order, and define
$$
R_{A,B,C,D} := u_{A,C} + u_{B,D} - 1.
$$
For example, for $(A,B,C,D) = (\{2,3\},\{4\},\{5,6,7\},\{1\})$, we have $R_{A,B,C,D} = u_{25}u_{26}u_{27}u_{35}u_{36}u_{37} + u_{14} - 1$.
Brown \cite{Bro} includes the generators $R_{A,B,C,D}$ in his definition of the ideal.  It follows from \cite{AHLcluster} that these additional generators are not necessary; we give a direct proof.

\begin{proposition}
The elements $R_{A,B,C,D}$ belong to the ideal $I = (R_{ij})$ of $\C[u_{ij}]$.
\end{proposition}
\begin{proof}
We prove the result using the following identity.  Suppose $|A| \geq 2$ and $a$ is the maximal element of $A$ in cyclic order, and set $A' := A \setminus a$.  Then we have (\cref{ex:Rij})
\begin{equation}\label{eq:ABCD}
R_{A,B,C,D} - R_{\{a\},B,C,A' \cup D} - u_{a,C} R_{A',B \cup a, C, D} + u_{B,D} R_{A',\{a\},B,C \cup D} = 0.
\end{equation}
We prove that $R_{A,B,C,D} \in I$ by reverse induction on $d = \max(|A|,|B|,|C|,|D|)$.  The base case is $d = n-3$, in which case after cyclic rotation we have $(|A|,|B|,|C|,|D|) = (1,1,1,n-3)$ and $R_{A,B,C,D} = R_{a,c}$.  Suppose the claim has been been proven for all $d' > d$, and suppose that $|D| = d$.  Suppose that $|A| \geq 2$.  Then by \eqref{eq:ABCD}, we have
$$
R_{A,B,C,D} = R_{\{a\},B,C,A' \cup D} + u_{a,C} R_{A',B \cup a, C, D} - u_{B,D} R_{A',a,B,C \cup D}.
$$
By induction, $- u_{B,D} R_{A',a,B,C \cup D}$ and $R_{\{a\},B,C,A' \cup D} $ belong to $I$.  Repeatedly applying this relation, we reduce to proving that $R_{A,B,C,D} \in I$ for $|A| = 1$.  Now if $|A|= 1$ and $|C| \geq 2$, we use the mirror reflection of \eqref{eq:ABCD} to reduce to the case $|A| = |C| = 1$, which belongs to $I$ by definition.  This completes the inductive step of the argument, and the proof.
\end{proof}


The embedding $\M_{0,n} \hookrightarrow T_U$ is in fact intrinsic to $\M_{0,n}$.   The group of units $ \C[\M_{0,n}]^\times/\C^\times$ of the coordinate ring of $\M_{0,n}$, modulo scalars, is a free abelian group $\lat = \lat(\M_{0,n}) := \C[\M_{0,n}]^\times/\C^\times$ (see \cref{sec:intrinsic}).  The \emph{intrinsic torus} of $\M_{0,n}$ is the torus $T$ with character lattice $\lat$.

\begin{proposition}\label{prop:uijchar}
The intrinsic torus of $\M_{0,n}$ is canonically isomorphic to the torus $T_U$.   Equivalently, the lattice $\lat = \C[\M_{0,n}]^\times/\C^\times$ can be identified with the set of Laurent monomials $\{\prod u_{ij}^{X_{ij}} \mid X_{ij} \in \Z\}$ in the dihedral coordinates.
\end{proposition}
\begin{proof}
One way to obtain this result is to view $\M_{0,n}$ as the complement of $\binom{n}{2} - n$ hyperplanes in $\C^{n-3}$, as explained in \cref{sec:hyper}.  Let $f_1(x) = 0, f_2(x) = 0, \ldots$ be the linear equations cutting out these hyperplanes, and note that the normal vectors to these hyperplanes span $\C^{n-3}$.  Then the coordinate ring $\C[\M_{0,n}]$ is obtained from the polynomial ring $\C[x_1,\ldots,x_{n-3}]$ by adjoining the inverses $f_1^{-1}, f_2^{-1},\ldots$.  It follows that $\C[\M_{0,n}]^\times/\C^\times$ is isomorphic to the lattice of Laurent monomials in the $f_i$.  Finally, we have an invertible monomial transformation between the $\binom{n}{2} - n$ dihedral coordinates $u_{ij}$ and the $\binom{n}{2} - n$ linear functions $f_1,f_2,\ldots$.  This can be obtained directly, or from the equality \eqref{eq:KBpotential}.
\end{proof}

We shall call
\begin{equation}\label{eq:L}
\lat = \left\{\prod u_{ij}^{X_{ij}} \mid X_{ij} \in \Z \right\} = \C[\M_{0,n}]^\times/\C^\times
\end{equation}
the \emph{character lattice} of $\M_{0,n}$.  The dual lattice $\Lambda^\vee$ is the \emph{cocharacter lattice} of $\M_{0,n}$.  While the torus $T_U$ is intrinsic to $\M_{0,n}$, the basis of characters $u_{ij}$ is not, but depends on a choice of connected component of $\M_{0,n}(\R)$ (\cref{ex:intrinsic}).

\begin{example}
Suppose $n = 4$.  Then $\M_{0,4}$ is the subvariety of $(\C^\times)^2$ cut out by the single relation $u_{13} + u_{24} = 1$.  The projection of $\M_{0,4}$ to the first factor gives an isomorphism $\M_{0,4} \cong \C \setminus \{0,1\}$, and we have $\C[\M_{0,4}] \cong \C[u,u^{-1},(1-u)^{-1}]$, where $u = u_{13}$.  The character lattice is $\lat = \{u_{13}^{X_{13}} u_{24}^{X_{24}}\} = \{u^a (1-u)^b\} \cong \Z^2$.
\end{example}

\subsection{Affine closure and compactification}\label{sec:MD}
The torus $T_U$ is an open subset of the affine space $\AA_U = \C^{\binom{n}{2} -n} = \Spec(\C[u_{ij}])$.  We define a partial compactification $\tM_{0,n}$ of $\M_{0,n}$ as the closure of $\M_{0,n} \hookrightarrow T_U \hookrightarrow \AA_U$ in affine space.  The variety $\tM_{0,n}$ has coordinate ring $\C[u_{ij}]/I$, where $I = (R_{ij})$.

The partial compactification $\tM_{0,n}$ has a natural stratification with strata indexed by subdivisions of the $n$-gon, or in other words, by the faces of the associahedron.  For a point $\usigma \in \tM_{0,n}$, define 
$$
M(\usigma) := \{(i,j)  \in \diag_n \mid u_{ij}(\usigma) \neq 0\}.
$$
We may think of $M(\usigma)$ as analogous to the matroid of a matrix.  It follows from the relations $R_{ij}$ that if the diagonals $(i,j)$ and $(k,l)$ cross then at least one of $(i,j)$ and $(k,l)$ belong to $M(\usigma)$.  Thus the diagonals not belonging to $M(\usigma)$ form a subdivision of the $n$-gon by non-intersecting diagonals.  

Let $\Delta = \Delta_{0,n}$ be the simplicial complex (\cref{sec:binary}) with vertex set equal to the set of diagonals $(ij)$ of an $n$-gon, and faces equal to subdivisions $\D$ of the $n$-gon, considered as a collection of diagonals.  We have a disjoint union 
$$
\tM_{0,n} = \bigcup_{\D \in \Delta} \M_D, \qquad \M_D = \{\usigma \in \tM_{0,n}\mid \diag_n \setminus M(\usigma) = \D\}.
$$
For example, $\M_{\emptyset} = \M_{0,n}$.  Let $\tM_D$ denote the closure of $\M_D$ in $\tM_{0,n}$.

\begin{proposition}\label{prop:break}
Let $\D$ be a subdivision of the $n$-gon into polygons of sizes $n_1,n_2,\ldots,n_s$.  Then the stratum $\M_D$ is isomorphic to $\M_{0,n_1} \times \cdots \M_{0,n_s}$ and $\tM_D$ is isomorphic to $ \tM_{0,n_1} \times \cdots \times \tM_{0,n_s}$.  It is irreducible and codimension $|\D|$ in $\tM_{0,n}$.
\end{proposition}
\begin{proof}
We prove the isomorphism $\tM_D \cong \tM_{0,n_1} \times \cdots \times \tM_{0,n_s}$ for the case $s = 2$.  The general case follows by induction and the other isomorphism is similar.
The diagonal $(ij)$ divides the $n$-gon into a polygon $P_1$ with vertices $i,i+1,\ldots,j$ and a polygon $P_2$ with vertices $j,j+1,\ldots,i$, where all vertex labels are taken cyclically modulo $n$.  When $u_{ij} = 0$, we have $u_{kl} = 1$ for all $(kl)$ crossing $(ij)$.  The relation $R_{kl} = 0$ becomes $1 + 0 - 1 = 0$ while the relation $R_{ij} = 0$ becomes $0 + 1 - 1 = 0$, and are automatically satisfied.  

Now let $(ab)$ be a diagonal that does not cross $(ij)$.  We assume that $(ab)$ is a diagonal of $P_1$.  The diagonals $(cd)$ that cross $(ab)$ are of two types: diagonals of $P_1$ that cross $(ab)$, or diagonals that cross both $(ab)$ and $(ij)$.  It follows that with $u_{ij} = 0$ and $u_{kl} = 1$ for all $(kl)$ crossing $(ij)$, the relation $R_{ab} = 0$ becomes
$$
R_{ab} = u_{ab} + \prod_{(cd)} u_{cd} - 1= 0
$$
where now $(cd)$ varies (only) over diagonals of $P_1$ that cross $(ab)$.  These are the $u$-equations for the polygon $P_1$.  Similarly, we obtain the $u$-equations for the polygon $P_2$.  More formally, let $I' = I + (u_{ij})$.  Then we have proven that
$$
\C[u_{ab}]/I' \cong (\C[u_{a_1b_1}]/I_1) \times (\C[u_{a_2b_2}]/I_2)
$$
where $I_1$ and $I_2$ are the $u$-relations for the polygons $P_1$ and $P_2$ respectively.  The last statement follows immediately from dimension counting $\dim(\M_{0,n_i}) = n_i-3$.
\end{proof}
%

The affine variety $\tM_{0,n}$ only partially compactifies $\M_{0,n}$.  We can use it to study the compactification $\bM_{0,n}$, the moduli space of stable rational marked curves, as follows.
\begin{definition}\label{def:bM0n}
The moduli space $\bM_{0,n}$ of stable rational curves with $n$ marked points is the (Zariski-)closure of $\M_{0,n}$ under the cross-ratio embedding $\iota$ of \eqref{eq:iota}
$$
\bM_{0,n} = \overline{\M_{0,n}} \subset (\P^1 )^{\binom{[n]}{4}}.
$$
\end{definition}

By definition, the cross-ratios $[ij|kl]$ are still well-defined on $\bM_{0,n}$, but now they take values in $\P^1 = \C \cup \{\infty\}$ rather than $\P^1 \setminus \{0,1,\infty\}$.
Since the dihedral coordinates are some of the cross-ratios, it is not hard to see that $\tM_{0,n}$ is an open subset of $\bM_{0,n}$.  Indeed, $\tM_{0,n}$ is the locus in $\bM_{0,n}$ where $u_{ij} \neq \infty$ for all diagonals $(i,j)$.  The action of $S_n$ on $\M_{0,n}$ by permuting points extends to to an action as automorphisms on $\bM_{0,n}$.  However, the action of $\alpha \in S_n$ sends $\tM_{0,n}$ to another subvariety $  \tM_{0,n}^\alpha \subset \bM_{0,n}$.  The subvariety $\tM_{0,n}^\alpha$ depends only on the dihedral ordering on $1,2,\ldots,n$ induced by $\alpha$, and often we will view $\alpha$ as a dihedral ordering.  A dihedral ordering on $1,2,\ldots,n$ is a permutation of the $n$ points, where two permutations are considered equivalent if they are related by either cyclic rotation or the reversal $w_1w_2 \cdots w_n \sim w_n w_{n-1} \cdots w_1$.  Note that the cyclic rotation $(12\cdots n)\in S_n$ sends $\tM_{0,n}$ to itself by sending $u_{ij} \mapsto u_{i+1,j+1}$, and similarly $(n,n-1,\ldots,1) \in S_n$ maps $u_{ij} \mapsto u_{n+1-j,n+1-i}$. 

By \cref{ex:Ucover}, every point $\usigma \in \bM_{0,n}$ belongs to $\tM_{0,n}^\alpha$ for some $\alpha$.  We view $$\{\tM_{0,n}^\alpha \mid \alpha \text{ a dihedral ordering}\}$$ as open charts covering $\bM_{0,n}$.

\begin{theorem}
The moduli space $\bM_{0,n}$ is a smooth, projective algebraic variety and the boundary $\bM_{0,n} \setminus \M_{0,n}$ is a simple normal crossings divisor.
\end{theorem}

\begin{proof}[Sketch]
Since the $\tM_{0,n}^\alpha$ cover $\bM_{0,n}$, smoothness of $\bM_{0,n}$ follows from smoothness of $\tM_{0,n}$.  This would in principle follow from the Jacobian criterion and the relations $R_{ij}$ \eqref{eq:Rij}.  However, I do not know a direct proof along these lines.

Brown \cite{Bro} proves this result by induction, considering inclusion maps of the form $\tM_{0,n} \to \tM_{0,n_1} \times \tM_{0,n_2}$ where $n_1+n_2 = n+4$.  

The statement also follows from the result in \cref{sec:string}, where $\tM_{0,n}$ is identified with an open subset of a smooth toric variety.  This toric variety is the toric variety associated to a particular associahedron, and the smoothness follows from results in the theory of cluster algebras; see \cite{AHLcluster} for further details.
\end{proof}

The moduli space $\bM_{0,n}$ has a natural stratification whose strata correspond to knowing which cross-ratios are equal to $0$, $1$, or $\infty$.  These strata are indexed by trees with $n$ leaves labeled $1,2,\ldots,n$ (\cref{ex:trees}).

\subsection{Positive coordinates}\label{sec:poscoord}
Let us consider the real points $\M_{0,n}(\R)$.  This is an open manifold of dimension $(n-3)$.  There is an action of the symmetric group $S_n$ on $\M_{0,n}$, acting by permuting the $n$ points $\sigma_1,\ldots,\sigma_n$.  Viewing points of $\M_{0,n}$ as $2 \times n$ matrices, the symmetric group $S_n$ acts by permuting the $n$ columns.

\begin{proposition}
The symmetric group $S_n$ acts transitively on the connected components of $\M_{0,n}(\R)$.  The manifold $\M_{0,n}(\R)$ has $(n-1)!/2$ connected components, in bijection with dihedral orderings of $1,2,\ldots,n$.
\end{proposition}
\begin{proof}
Using the $\PGL(2)$ action, we place $(\sigma_1,\sigma_2,\sigma_n)$ at $(0,1,\infty)$.  Define the positive part $(\M_{0,n})_{> 0} \subset \M_{0,n}(\R)$ by \begin{equation}\label{eq:M0npos}
(\M_{0,n})_{> 0} = \{(\sigma_3,\ldots,\sigma_{n-1}) \mid 0 < 1 < \sigma_3 < \sigma_4 < \cdots < \sigma_{n-1} < \infty\},
\end{equation}
the space of $n$ points in $\P^1(\R)$ arranged in order.  This is clearly a connected component of $\M_{0,n}(\R)$.  By a direct calculation, $(\M_{0,n})_{> 0}$ is the locus in $ \M_{0,n}(\R)$ where all dihedral coordinates $u_{ij}$ take positive values.  The symmetric group $S_n$ sends cross-ratios to cross-ratios.  The subgroup of $S_n$ that sends dihedral coordinates to dihedral coordinates is exactly the dihedral subgroup $D_{2n} \subset S_n$ (generated by the cyclic rotation $(12\cdots n)$ and the reflection $(n(n-1) \cdots 21)$).  It follows that the connected components of $\M_{0,n}(\R)$ are in bijection with the cosets $S_n/D_{2n}$, and the second statement follows.
\end{proof}

We denote the connected components of $\M_{0,n}(\R)$ by $\M_{0,n}(\alpha)$, where $\alpha$ is a dihedral ordering of $1,2,\ldots,n$.  

The positive part $(\M_{0,n})_{> 0}$ \eqref{eq:M0npos} of $\M_{0,n}$ is diffeomorphic to an open ball.  We realize this using explicit parametrizations:
$$
\begin{bmatrix}
1 & 1 & 1 & 1 & 1 & \cdots & 0 \\
0 & 1 & 1+ x_1 & 1+x_1 + x_2 &1+x_1 + x_2+ x_3 & \cdots & 1 
\end{bmatrix} 
$$
$$
\begin{bmatrix}
1 & 1 & 1 & 1 & 1 & \cdots & 0 \\
0 & 1 & 1+ y_1 & 1+y_1 + y_1y_2 &1+y_1 + y_1y_2+ y_1y_2y_3 & \cdots & 1 
\end{bmatrix} 
$$
where $(x_1,\ldots,x_{n-3}) \in \R_{>0}^{n-3}$ or $(y_1,\ldots,y_{n-3}) \in \R_{>0}^{n-3}$.  The rational coordinates $(x_1,\ldots,x_{n-3})$ and $(y_1,\ldots,y_{n-3})$ are called \emph{positive coordinates}, and are related by the monomial transformation
$$
x_i = y_1 y_2 \cdots y_i.
$$  
Each coordinate system consists of regular functions on $\M_{0,n}$ that generate the field $\C(\M_{0,n})$ of rational functions.  Furthermore, the maps 
$$
(x_1,\ldots,x_{n-3}): \M_{0,n} \to (\C^\times)^{n-3}, \qquad (y_1,\ldots,y_{n-3}): \M_{0,n} \to (\C^\times)^{n-3}
$$
map $(\M_{0,n})_{>0}$ diffeomorphically onto $\R_{>0}^{n-3}$.  In these coordinate systems, we have the following explicit formulae.  For $1 \leq i < j < n$, the $2 \times 2$ minor $(ij)$ is given by
$$
(ij) = x_{i-1} + x_i + \cdots + x_{j-2} = y_1y_2 \cdots y_{i-1} (1+  y_i + y_i y_{i+1} + \cdots + y_i y_{i+1} \cdots y_{j-2}),
$$
where by convention $x_0 = 1$, and for $j = n$, we have $(in) = 1$.  The dihedral coordinates are related to the positive coordinates by
$$
u_{ij} 
= \begin{cases} 
\frac{y_i(1+  y_{i+1}+ y_{i+1} y_{i+2} + \cdots + y_{i+1} \cdots y_{n-3})}{(1+  y_i + y_i y_{i+1} + \cdots + y_i y_{i+1} \cdots y_{n-3})} & \mbox{if $j = n-1$} \\
\frac{(1+  y_1+ y_1 y_2 + \cdots + y_1 y_2 \cdots y_{i-2})}{(1+  y_1+ y_1 y_2 + \cdots + y_1 y_2 \cdots y_{i-1})}& \mbox{if $j = n$} \\
\frac{(1+  y_i + y_i y_{i+1} + \cdots + y_i y_{i+1} \cdots y_{j-1})(1+  y_{i+1}+ y_{i+1} y_{i+2} + \cdots + y_{i+1} \cdots y_{j-2})}{(1+  y_i + y_i y_{i+1} + \cdots + y_i y_{i+1} \cdots y_{j-2})(1+  y_{i+1}+ y_{i+1} y_{i+2} + \cdots + y_{i+1} \cdots y_{j-1})} & \mbox{else}.
\end{cases}
$$
We obtain, by inspection, 
\begin{proposition}
The dihedral coordinates take values in $(0,1)$ on $(\M_{0,n})_{>0}$.
\end{proposition}

For example, for $n = 4$ we have
$$
u_{13} = \frac{y_1}{1+y_1}, \qquad u_{24} = \frac{1}{1+y_1},
$$
and for $n = 5$ we have
\begin{align}\label{eq:u5y}
\begin{split}
u_{13} &= \frac{1+y_1+y_1y_2}{(1+y_1)(1+y_2)}, \qquad u_{14} = \frac{y_1(1+y_2)}{1+y_1+y_1y_2}, \qquad u_{24} = \frac{y_2}{1+y_2},\\
 \qquad u_{25} &= \frac{1}{1+y_1}, \qquad u_{35} = \frac{1+y_1}{1+y_1+y_1y_2}.
 \end{split}
\end{align}

The next result follows immediately from these formulae.
\begin{proposition}\label{prop:inverty}
The rational function field $\C(\M_{0,n})$ of $\M_{0,n}$ is $\C(y_1,\ldots,y_{n-3})$.  The character lattice $\lat(\M_{0,n})$ is the lattice of Laurent monomials in the $n-3$ coordinates $y_1,y_2,\ldots,y_{n-3}$ and the $\binom{n-2}{2}$ polynomials 
$$
p_{ij} = (1+  y_i + y_i y_{i+1} + \cdots + y_i y_{i+1} \cdots y_{j-2})
$$
where $1 \leq i <  j-1 < j \leq n-1$.
\end{proposition}

Note that $\binom{n-2}{2} + (n-3) = \binom{n}{2} - n$ is the number of diagonals of a $n$-gon.

\subsection{Canonical form on $\M_{0,n}$}
Let $(\M_{0,n})_{\geq 0} \subset \tM_{0,n} \subset \bM_{0,n}$ be the analytic closure of the positive component $(\M_{0,n})_{>0}$ in $\tM_{0,n}$.  This is a compact semialgebraic set inside $\tM_{0,n}(\R)$.  It is cut out by the inequalities $u_{ij} \geq 0$.  In fact, the coordinates $u_{ij}$ take values in $[0,1]$ on $(\M_{0,n})_{\geq 0}$, and identify $(\M_{0,n})_{\geq 0}$ with a closed subspace of the unit cube $[0,1]^{\binom{n}{2}-n}$.  The face structure of $(\M_{0,n})_{\geq 0}$ is the same as that of the \emph{associahedron}: for each subdivision $\D$ of the $n$-gon, the (open) stratum $(\M_{0,n})_{\geq 0} \cap \M_\D$ is an open ball of codimension equal to the number of diagonals $|\D|$.  

\begin{example}
Let $n = 5$.  The two-dimensional space $(\M_{0,5})_{\geq 0}$ has five facets, given by $u_{ij} = 0$ for $(ij)$ a diagonal of the pentagon.  Consider one of the facets, which by cyclic symmetry we can take to be $u_{13} = 0$.  The five $u$-equations then reduce to
$$
u_{24} = 1, \qquad u_{25} = 1, \qquad u_{35} + u_{14} = 1.
$$
So on this facet $u_{13},u_{24},u_{25}$ are constant, and we have a one-dimensional space sitting inside two-dimensional $u_{14},u_{35}$ space, which is isomorphic to $(\M_{0,4})_{\geq 0}$, an interval.  It is easy to see that $(\M_{0,5})_{\geq 0}$ is combinatorially a pentagon.
\end{example}

\begin{definition}
The \emph{canonical form} of $(\M_{0,n})_{\geq 0}$, or the \emph{Parke-Taylor form} is the top-degree differential form
$$
\Omega = \Omega_{0,n}:= \bigwedge_{i=1}^{n-3} \dlog x_i = \bigwedge_{i=1}^{n-3} \frac{dx_i}{x_i}
=   \bigwedge_{i=1}^{n-3} \dlog y_i = \bigwedge_{i=1}^{n-3} \frac{dy_i}{y_i}
$$
\end{definition}
Since $y_i$ is a regular, non-vanishing function on $\M_{0,n}$, this is a regular top-form on $\M_{0,n}$.  It has poles in the partial compactification $\tM_{0,n}$, and in general it has both poles and zeroes in $\bM_{0,n}$.  By \cref{ex:formuij}, this form can also be written as 
\begin{equation}\label{eq:formuij}
\Omega = \bigwedge_{(ij)} \dlog \frac{u_{ij}}{1-u_{i,j}}\end{equation}
where $(ij)$ ranges over a triangulation $\D$ of the $n$-gon with no interior triangles (that is, triangles none of whose sides are sides of the $n$-gon).  The ``no interior triangle" condition is necessary.  A recent result of Silversmith \cite{Sil} states that the field extension $\C(u_{ij} \mid (ij) \in \D) \subset \C(\M_{0,n})$ has degree $2^a$ where $a$ is the number of interior triangles in the triangulation $\D$.

The Parke-Taylor form $\Omega$ encodes a remarkable amount of physics.  By taking different kinds of integrals of this form, one obtains scalar $\phi^3$-amplitudes, string theory amplitudes, gauge theory amplitudes, and so on.

For each diagonal $(ij) \in \diag_n$, let $\tM_{ij} \subset \tM_{0,n}$ be the divisor cut out by the equation $u_{ij} = 0$.  Then by \cref{prop:break}, $\tM_{ij}$ can be identified with $\tM_{0,n_1} \times \tM_{0,n_2}$, where $n_1 = j - i + 1$ and $n_2 = n+i - j +1$.  For the notion of residue in the following result, we refer the reader to \cite[Section 2]{LamPosGeom}.

\begin{lemm}\label{lem:M0npos}
The canonical form $\Omega$ has a simple pole along $\tM_{ij}$ and we have
$$
\Res_{\tM_{ij}} \Omega = \Omega_{0,n_1} \wedge \Omega_{0,n_2}
$$
under the isomorphism $\tM_{ij} \cong \tM_{0,n_1} \times \tM_{0,n_2}$.
\end{lemm}
\begin{proof}
Pick a triangulation $\D$ with no interior triangles that uses the diagonal $(ij)$.  For all other diagonals $(kl) \in \D$ the function $u_{kl}$ is a regular function on $\tM_{ij}$ that is neither identically 0 nor identically equal to 1.  Thus $\dlog \frac{u_{kl}}{1-u_{kl}}$ has neither poles nor zeroes on $\tM_{ij}$.  However, $\dlog  \frac{u_{ij}}{1-u_{i,j}} = \frac{1}{u_{ij}(1-u_{ij})} du_{ij}$ has a simple pole along $u_{ij} = 0$.  We have
$$
\Res_{z = 0} \frac{1}{z(1-z)} dz = 1
$$
and it follows that 
\begin{equation*}
\Res_{u_{ij} = 0} \Omega = \bigwedge_{(kl) \in \D \setminus (ij)}  \dlog \frac{u_{kl}}{1-u_{kl}} = \Omega_{0,n_1} \wedge \Omega_{0,n_2}.  \qedhere
\end{equation*}
\end{proof}

For example, the canonical form of $\M_{0,4}$ is $\frac{du_{13}}{u_{13}(1-u_{13})}$ which has simple poles with residues $1$ and $-1$ at the two boundaries $u_{13} = 0$ and $u_{24}=0$ ($u_{13} = 1$).

\subsection{Positive geometry}
\cref{lem:M0npos} is close to saying that $(\M_{0,n})_{\geq 0}$ is a positive geometry with canonical form $\Omega$.  Informally, a positive geometry $X_{\geq 0}$ is a semialgebraic space equipped with a top-degree rational form $\Omega(X_{\geq 0})$ whose polar structure reflects the face structure of $X_{\geq 0}$.  

\begin{definition}\label{def:posgeom}
A \emph{positive geometry} is a $d$-dimensional (oriented) semialgebraic space $X_{\geq 0}$ in the real points of a projective $d$-dimensional complex algebraic variety $X$ such that there exists a unique top-degree rational form $\Omega(X_{\geq 0})$ on $X$ with the following properties: 
\begin{enumerate}
\item either $d = 0$, and $X = X_{\geq 0}$ is a point and $\Omega(X_{\geq 0}) = \pm 1$ (depending on orientation), or
\item every facet $(F, F_{\geq 0})$ of $X_{\geq 0}$ is a $(d-1)$-dimensional positive geometry, and $\Omega(X_{\geq 0})$ has simple pole along $F$ satisfying
$$
\Res_{F} \Omega(X_{\geq 0}) = \Omega(F_{\geq 0})
$$
and these are the only poles of $\Omega(X_{\geq 0})$.
\end{enumerate}
\end{definition}
For a more careful discussion of the definition of a positive geometry we refer the reader to \cite{ABL,LamPosGeom}.  

Note that even though $(\M_{0,n})_{\geq 0}$ is combinatorially a polytope (the associahedron), it is not actually a polytope.  Instead, it can be identified with the nonnegative part of a toric variety (see \cref{sec:string}) and it is also a \emph{wondertope} in the sense of Brauner, Eur, Pratt, Vlad \cite{BEPV}.  

It follows from induction on $n$ and \cref{lem:M0npos} that $(\bM_{0,n}, (\M_{0,n})_{\geq 0})$ is a positive geometry with canonical form $\Omega$, modulo the statement that $\Omega$ has no poles on $\bM_{0,n}$ except the simple poles along the divisors $\tM_{ij}$.  We leave this as \cref{ex:Omegazero}.

\begin{theorem}[\cite{AHLcluster}]\label{thm:M0npos}
$(\bM_{0,n}, (\M_{0,n})_{\geq 0})$ is a positive geometry with canonical form $\Omega$.
\end{theorem}

Another proof of \cref{thm:M0npos} is given in \cite{BEPV}.  

\subsection{Exercises and Problems}

\begin{exercise}\label{ex:PGL}\
\begin{enumerate}
\item
Show that $\PGL(k)$ acts simply-transitively on the set of $k+1$ distinct ordered points in $\P^{k-1}$.
\item
Verify the dimension of $\M_{0,n}$ stated in \cref{prop:M0n}.
\end{enumerate}
\end{exercise}

\begin{exercise}\label{ex:Rij} \
\begin{enumerate}[label=(\alph*)]
\item
Prove \cref{lem:cr}.
\item
Prove the relations \eqref{eq:Rij}.
\item
Let $(A,B,C,D)$ be a decomposition of $[n]$ into four non-empty cyclic intervals, in cyclic order.  Show that the identity
$$
u_{A,C} + u_{B,D} = 1
$$
holds by writing both terms as cross-ratios.
\item
Show that
$$
\frac{1-u_{ij}}{u_{ij}} \frac{1-u_{i+1,j+1}}{u_{i+1,j+1}} = (1-u_{i,j+1})(1-u_{i+1,j}).
$$
\item
Prove \eqref{eq:ABCD}.
\end{enumerate}
\end{exercise}

\begin{exercise} \label{ex:intrinsic}
Call a character $\gamma \in \Lambda$ \emph{bounded} if it takes bounded values on $(\M_{0,n})_{>0}$.  Show that $\gamma$ is bounded if and only if it is a monomial $\gamma = \prod_{ij} u^{a_{ij}}_{ij}$ in the $u_{ij}$, where $a_{ij} \geq 0$.
%
%
\end{exercise}

\begin{exercise}\label{ex:Ucover}\
\begin{enumerate}[label=(\alph*)]
\item
If you know another definition of $\bM_{0,n}$, show that it agrees with \cref{def:bM0n}.
\item
Let $\usigma \in \bM_{0,n}$.  Show that there exists a dihedral ordering $\tau \in S_n$ such that $u_{\tau(i),\tau(j)}(\usigma) \neq \infty$ for all diagonals $(i,j)$.
\end{enumerate}
\end{exercise} 

\begin{exercise}\label{ex:trees} \
\begin{enumerate}[label=(\alph*)]
\item
Prove that the closure $\tM_D$ of the stratum $\M_\D$ in $\tM_{0,n}$ is the union $\bigcup_{\D'} \M_{\D'}$ over all subdivisions $\D'$ that refine $\D$.
\item
Find a bijection between subdivisions of the $n$-gon and planar $n$-leaf trees.  These are trees, embedded in the plane, with $n$ leaves labeled $1,2,\ldots,n$ in order when traversing the outside of the tree.  Thus the strata of $\tM_{0,n}$ are labeled by planar $n$-leaf trees.  
\item
Show that the strata of $\bM_{0,n}$ are labeled by all $n$-leaf trees with leaves labeled $1,2,\ldots,n$.  Here, the trees are treated as abstract graphs with labeled leaves.
\end{enumerate}
\end{exercise}

\begin{exercise}\label{ex:signpatternM0n}
The functions $u_{ij}$ take nonzero real values on $\M_{0,n}(\R)$.  Prove that the signs of $u_{ij}$ determine the connected component that a point lies in.  What are the possible signs for $\{u_{ij}\}$?  This can be considered a moduli space analogue of a matroid.
\end{exercise}

\begin{problem}\label{ex:formuij}
Prove \eqref{eq:formuij}.  Is there a formula for $\Omega$ using $\{u_{ij} \mid (ij) \in \D\}$ where $\D$ is a triangulation that does have interior triangles?
\end{problem}

\subsubsection{Canonical form of $\Gr(2,n)_{\geq 0}$}
The \emph{positive Grassmannian} $\Gr(2,n)_{> 0}$ is the locus in $\Gr(2,n)$ represented by $2 \times n$ matrices all of whose $2 \times 2$ minors satisfy $(ij) > 0$.  Its analytic closure in $\Gr(2,n)$ is the (totally) nonnegative Grassmannian $\Gr(2,n)_{\geq 0}$.  Even though $(\M_{0,n})_{>0}$ is a rather straightforward quotient of $\Gr(2,n)_{>0}$ by a positive torus $T_{>0} \cong \R_{>0}^{n-1}$, the combinatorics of $(\M_{0,n})_{\geq 0}$ and $\Gr(2,n)_{\geq 0}$ is remarkably quite different \cite{LamCDM,PosTP}!  Like $(\M_{0,n})_{\geq 0}$, the totally nonnegative Grassmannian $\Gr(2,n)_{\geq 0}$ is a regular CW-complex homeomorphic to a closed ball, but it is not a polytopal complex; see \cite{GKL1,GKL3}.

We investigate the canonical form $\Omega(\Gr(2,n)_{\geq 0})$ of the positive Grassmannian.  The map $\Gr^\circ(2,n) \to \M_{0,n}$ from \eqref{eq:M0ncomm} is a locally-trivial fibration with fiber the $(n-1)$-dimensional torus $T'$.  In fact, this family can be trivialized and we have an isomorphism $\Gr^\circ(2,n) = \M_{0,n} \times T'$.  The torus $T'$ has a canonical form $\Omega_{T'} = \bigwedge_{i=1}^{n-1} \dlog t_i$, where $t_i$ are a basis of characters of $T'$.    Define a top-form on $\Gr(2,n)$ by 
$$
\Omega(\Gr(2,n)_{\geq 0}) := \Omega_{0,n} \wedge \Omega_{T'}
$$
using the isomorphism $\Gr^\circ(2,n) \cong \M_{0,n} \times T'$, and noting that $(n-3) + (n-1) = 2n-4 = \dim \Gr(2,n)$. 

\begin{exercise}\label{ex:forms} 
\noindent
(a) Show that $\Omega(\Gr(2,n)_{\geq 0})$ has no zeroes on $\Gr(2,n)$, and simple poles only, exactly along the $n$ positroid divisors $(i,i+1) = 0$ (see \cite{LamCDM, KLS} for more on positroid varieties).

\noindent
(b)
Now let $\pi: \Mat_{2,n} \longdashrightarrow \Gr(2,n)$ be the natural quotient (rational) map.  Then $\Theta := \pi^*\Omega(\Gr(2,n)_{\geq 0})$ is a $2n-4$ form on the $2n$-dimensional affine space $\Mat_{2,n}$.  Show that (up to a constant) we have
$$
\Theta = \frac{1}{(12)(23) \cdots (n1)} \det(C_1 C_2 d^{n-2}C_1)\wedge \det(C_1 C_2 d^{n-2} C_2),
$$
where $C_1,C_2$ denote the two rows of a $2 \times n$ matrix, the notation $d^{n-2}C_1$ denotes a $(n-2) \times n$ matrix of 1-forms, all of whose rows are $dC_1$.  For example, with $n = 4$, we would have $\det(C_1 C_2 d^{n-2}C_1)=$
$$
\det \begin{bmatrix} C_{11} & C_{12} & C_{13} & C_{14} \\
C_{21} & C_{22} & C_{23} & C_{24} \\ 
dC_{11} & dC_{12} & dC_{13} & dC_{14} \\ 
dC_{11} & dC_{12} & dC_{13} & dC_{14} 
 \end{bmatrix} = C_{11} C_{22} dC_{13} \wedge dC_{14} - C_{11} C_{22} dC_{14} \wedge dC_{13} \pm \cdots
$$
\end{exercise}

\subsubsection{Canonical class of $\bM_{0,n}$ and zeroes of $\Omega_{0,n}$}

We recall some statements concerning the intersection theory of $\bM_{0,n}$ \cite{Keel}.
The dimension of the first Chow group $A^1(\bM_{0,n})$ (or the cohomology group $H^2(\bM_{0,n})$) is equal to $2^{n-1} - \binom{n}{2} - 1$.  For example, for $n = 4$, we get $\dim H^2(\P^1) = 1$.  A spanning set is indexed by the classes of the $2^{n-1}$ boundary divisors $D_S = D_{\bar S}$ where $S \subset 2^{[n]}$ is a subset of size $|S| \in [2,n-2]$, and $\bar S:= [n] \setminus S$.  In the notation of \cref{ex:trees}, $D_S$ is the boundary divisor corresponding to a tree with two interior vertices, one connected to the leaves in $S$, and the other to the leaves in $\bar S$.  For any four distinct $i,j,k,l \in [n]$, we have the relation
$$
\sum_{\substack{i,j \in S\\ k,l \notin S}} [D_S] = \sum_{\substack{i,k \in S\\ j,l \notin S}} [D_S] = \sum_{\substack{i,l \in S\\ j,k \notin S}} [D_S] \qquad \mbox{in $A^1(\bM_{0,n})$.}
$$
The canonical class of $\bM_{0,n}$ is given by \cite{Pan}
\begin{equation}\label{eq:KM}
K_{\bM_{0,n}} = \sum_{s=2}^{\lfloor n/2 \rfloor} \left(\frac{s (n-s)}{n-1} - 2\right) \sum_{|S| = s} [D_S].
\end{equation}
For example, for $n=4$, we get $K_{\bM_{0,4}} = -\frac{2}{3} ([0] +[1] +[\infty]) = -2[\pt]$, as expected.  For $n=5$, we get
$$
K_{\bM_{0,5}} = -\frac{1}{2} \sum_{i,j} [D_{ij}].
$$

\begin{exercise}\label{ex:Omegazero} 
 Compute the order of vanishing of $\Omega_{0,n}$ on every divisor $D_{S} \subset \bM_{0,n}$.  In particular, show that the only poles of $\Omega_{0,n}$ are the divisors $D_{S}$ such that $D_{S} \cap \tM_{0,n}$ is a divisor in $\tM_{0,n}$.  Hint: use the action of the symmetric group $S_n$ permuting the $n$ points.  For $w \in S_n$, write $w^*(\Omega_{0,n}) = f_w \Omega_{0,n}$, where $f_w$ is a rational function on $\M_{0,n}$ depending on $w$.  Now compute the poles of $f_w$ on $\tM_{0,n}$.
\end{exercise}

\begin{exercise} By definition, the class $$[\Omega_{0,n}] = [\mbox{divisor of zeroes}] - [\mbox{divisor of poles}]$$ of the rational top-form $\Omega_{0,n}$ in $\bM_{0,n}$ is equal to $K_{\bM_{0,n}}$.  Verify the formula \eqref{eq:KM} by computing $[\Omega_{0,n}]$.  
Hint: $w \in S_n$ acts on $A^1(\bM_{0,n})$ sending $D_S$ to $D_{w(S)}$. The canonical class is invariant under the automorphisms, and thus invariant under the action of $S_n$.  Thus we can compute $K_{\bM_{0,n}}$ by averaging $w \cdot [\Omega_{0,n}]$ over all permutations $w \in S_n$.
\end{exercise}

%
%
%
%
%
%

\section{Binary geometry}\label{sec:binary}
\def\tU{{\tilde U}}
We formalize some of the properties of the relations \eqref{eq:Rij} in terms of binary geometries \cite{AHLT,AHLcluster}.  

\subsection{Binary geometries and $u$-equations}
A simplicial complex $\Delta$  on $[n]$ is a collection of subsets $\Delta \subseteq 2^{[n]}$ satisfying:
\begin{enumerate}
\item
If $F \in \Delta$ and $F' \subset F$ then $F' \in \Delta$.
\item
$\Delta$ is non-empty.
\end{enumerate}
The elements $F \in \Delta$ of a simplicial complex are called the \emph{faces} of $\Delta$.  The dimension $\dim \Delta$ is the size of the largest face in $\Delta$, minus one.  Faces of $\Delta$ that are maximal under inclusion are called \emph{facets}.  We say that $\Delta$ is \emph{pure} if all facets have the same dimension.

We say that $\Delta$ is a \emph{flag complex} if each $\{i\}$ is a face of $\Delta$ for each $i$, and for $r >2$, we have that $F = \{a_1,\ldots,a_r\} \in \Delta$ whenever $\{a_i,a_j\} \in \Delta$ for all $ 1\leq i < j \leq r$.  We write $i \sim j$ if $\{i,j\}$ is a face of a flag complex $\Delta$, and say that $i$ and $j$ are \emph{compatible}.  Otherwise, $i$ and $j$ are incompatible.

For $S \subset [n]$, let $H_S \subset \C^{[n]}$ denote the subspace where $x_s = 0$ for $s \in S$.

\begin{definition}\label{def:binary}
Let $\Delta$ be a flag simplicial complex on $[n]$.  A \emph{binary geometry} $\tU$ for $\Delta$ is an affine algebraic variety $\tU \subset \C^n$ of dimension $\dim \tU = \dim \Delta + 1$ cut out by $n$ equations of the form
\begin{equation}\label{eq:u}
R_i := u_i + \prod_{j \mid j \not \sim i} u_j^{a_{ij}} -1 = 0, \qquad i = 1,2,\ldots,n
\end{equation}
for integers $a_{ij} > 0$, satisfying the following property: for a subset $S \subset [n]$, the subvariety $\tU_S:= \tU \cap H_S $ is non-empty exactly when $S \in \Delta$, and in this case $\tU_S$ is irreducible, and pure of codimension $|S|$ in $\tU$.
\end{definition}

Note that when $u_i = 0$, the equation $R_j = 0$ gives $u_j = 1$ for any $j \not \sim i$.  Suppose that $\tU$ is a binary geometry for $\Delta$, and $x \in \tU$.  Let 
$$
F_x := \{i \in [n] \mid u_i(x) = 0\}.
$$
By the assumptions of \cref{def:binary}, we have $F_x \in \Delta$.  We have a stratification of $\tU$ by the closed subvarieties $\tU_F$, for $F \in \Delta$.  Let 
$$
\oU_F = \{ x \in \tU \mid F_x = F\}
$$ 
denote the locally-closed pieces in the stratification.  Thus $\tU = \bigsqcup_{F \in \Delta} \oU_F$, and each $\oU_F$ is itself non-empty, irreducible, and codimension $|F|$ in $\tU$.

We denote $\oU:= \oU_\emptyset$.  Thus $\oU$ is the Zariski-open subset of $\tU$ obtained by intersecting $\tU$ with the torus $T = (\C^\times)^n \subset \C^n$.  It is an irreducible, closed subvariety of the torus $T$, and thus by definition a very affine variety (see \cref{sec:veryaffine}).  Indeed, each $\oU_F$ is a very affine variety.

\begin{problem}
Which flag simplicial complexes have binary geometries?  When is a binary geometry smooth, and when are $\tU_F$ and $\oU_F$ smooth?  When do the equations \eqref{eq:u} cut $\tU$ out scheme-theoretically?
\end{problem}

\begin{example}\label{ex:triangle}
Let $\Delta = \{ \emptyset, 1,2,3,12,23,13\}$ be dual to the face complex of a triangle.  (This simplicial complex is not flag, but it still makes sense to consider the $u$-equations.)  The equations for a binary geometry $\tU$ for $\Delta$ take the form 
$$
u_1 + 1 = 1, \qquad u_2 +1 = 1, \qquad u_3 + 1 = 1,
$$
which has a unique solution $(u_1,u_2,u_3) = (0,0,0)$.  Since this has the wrong dimension, there is no binary geometry for $\Delta$.
\end{example}

\begin{example}\label{ex:square}
Let $\Delta = \{ \emptyset, 1,2,3,4,12,23,14,23\}$ be dual to the face complex of a square.  The equations for a binary geometry $\tU$ for $\Delta$ take the form 
$$
u_1 + u^a_3 = 1, \qquad u_2 +u^b_4 = 1, \qquad u_3 + u_1^c = 1, \qquad u_4 +u_2^d = 1,
$$
for positive integers $a,b,c,d$.  If we choose $a =b=c=d=1$, we obtain a binary geometry $U \cong \C^2$.  For all other choices of $a,b,c,d$, we have $\dim U < 2$ and we do not have a binary geometry.
\end{example}

\begin{example}
Let $\Delta = \Delta_{0,n}$ be the simplicial complex with vertex set equal to the set of diagonals $(ij)$ of an $n$-gon, and faces equal to subdivisions $\D$ of the $n$-gon, considered as a collection of diagonals.  Then $\tM_{0,n}$ is a binary geometry for $\Delta_{0,n}$.
\end{example}

Other examples of binary geometries are given in \cite{AHLcluster} and \cite{HLRZ}.

%
%

\begin{proposition}\label{prop:Ulink}
Let $\tU$ be a binary geometry for $\Delta$ and $F \in \Delta$.  The variety $\tU_F$ is a binary geometry for the link
$$
\lk_\Delta(F) = \{G \in \Delta \mid F \cap G = \emptyset \text{ and } F \cup G \in \Delta\}.
$$
In particular, $\dim \lk_\Delta(F) + \dim F = \dim \Delta - 1$.
\end{proposition}
\begin{proof}
The link $\lk_\Delta(F)$ is a simplicial complex on the set $S = \{j \notin F \mid F \cup \{j\} \in \Delta\}$.  On $\tU_F$ we have $u_i = 0$ for $i \in F$ and by \eqref{eq:u} we have $u_k = 1$ for $k \notin S$.  Thus $\tU_F$ is naturally embedded in the affine space $\C^S \subset \C^{[n]}$ given by $\{u_i = 0 \mid i \in F\} \cap \{u_k = 1 \mid k \notin S \cup F\}$.  The equations $R_i$ for $i \in F$ and $R_k$ for $k \notin S$ are automatically satisfied.  The equations $R_j$ for $j \in S$ become 
$$
R_j|_{\C^S} = u_j + \prod_{l \in S \mid j \not \sim l}  u_{l}^{a_{jl}} - 1.
$$
We have that $(\tU_F)_G = \tU_{F \cup G}$.  It follows that $\tU_F$ is a binary geometry for $\lk_{\Delta(F)}$.
\end{proof}

\begin{corollary}\label{cor:pure}
Suppose that a binary geometry for $\Delta$ exists.  Then $\Delta$ is pure.
\end{corollary}
\begin{proof}
If $F, F' \in \Delta$ are maximal then $\lk_\Delta(F) = \{\emptyset\} = \lk_\Delta(F')$.  It follows from \cref{prop:Ulink} that $|F| = |F'|$.
\end{proof}

\begin{remark}
The $u$-equations \eqref{eq:u} that we consider are sometimes called \emph{perfect $u$-equations}.  For examples of some more general $u$-equations see \cite[Section 11]{AHLstringy} and \cite{HLRZ}.  More recent appearances of $u$-equations include \cite{ACDFS,AFSPT}.

\end{remark}

\subsection{One-dimensional binary geometries}
\begin{theorem}\label{thm:onedim}
Suppose that $\tU$ is a one-dimensional binary geometry for $\Delta$.  Then $\Delta$ consists of two isolated vertices and $U \cong \C$.
\end{theorem}
\begin{proof}
By \cref{cor:pure}, $\Delta = \{\emptyset,1,2,\ldots,n\}$ must be $0$-dimensional, consisting of $n$ isolated vertices.  If $n = 1$, we see directly that no binary geometry for $\Delta$ exists.  If $n=2$, we have the two relations 
$$
u_1 + u^a_2 = 1, \qquad u_2 + u^b_1 = 1
$$
and we have binary geometry if and only if $a = b = 1$, giving $U \cong \C$.  

Now suppose that $n \geq 3$.  We will show that no binary geometry exists.  The relations are $R_i = u_i + \prod_{j\neq i} u_j^{a_{ij}} -1 = 0$, where $a_{ij} >0$.  One solution to these equations is the point $p = (u_1,u_2,u_3,\ldots,u_n) = (0,1,1,\ldots,1)$.  We shall show that this is an isolated point of $\tU$, contradicting the requirement that $\tU$ is one-dimensional and irreducible.

Consider the Jacobian matrix $(\frac{\partial R_i}{\partial u_j})_{i,j=1}^n$ at the point $p$.  It has the form
$$ J = 
\begin{bmatrix} 1 & a_{12} & a_{13} & a_{14} &  \cdots & a_{1n} \\
\beta_2 & 1 & 0 & 0 & \cdots & 0 \\
\beta_3 & 0 & 1 & 0 &\cdots & 0 \\
\vdots & \vdots & \vdots & \ddots 
\end{bmatrix}
$$
where $\beta_i = \delta_{a_{i1},1}$ is the Kronecker delta.  If the Jacobian is nonsingular at $p$ then $p$ must be an isolated point of $\tU$.  We have $\det J = 1 - a_{12}\beta_2 - a_{13}\beta_3 - \cdots - a_{1n}\beta_n$.  Since each $a_{ij}$ is a positive integer and $\beta_i \in \{0,1\}$, we have $\det J = 0$ exactly when one of the terms $a_{1j}\beta_j$ is equal to $1$, and the rest of these terms vanish.  

We now rule out the case $\det J = 0$.  By relabeling, we may assume that $j = 2$.  Thus $a_{12} = a_{21} = 1$ and $a_{i1} \geq 2$ for $i = 3,\ldots,n$.  We search for a formal power series solution $(u_1(t),\ldots,u_n(t)) \in \C[[t]]^n$ to the equations $R_1,\ldots,R_n$, where $(u_1(0),\ldots,u_n(0)) = p$.  (Assuming $\tU$ is a curve, the parameter $t$ may be taken to be a local analytic parameter on the normalization of $\tU$, in the neighborhood of a preimage of $p \in U$.)
We may assume that not all the linear terms of $u_1(t),\ldots,u_n(t)$ vanish, and considering the linear coefficients $[t]R_1, [t] R_2,\ldots$, we deduce that
\begin{align*}
&u_1(t) = \alpha t + \beta t^2 + \cdots, \qquad u_2(t) = 1 - \alpha t + \beta' t^2 + \cdots,\\
&u_3(t),u_4(t),\ldots, u_n(t) \in 1 + O(t^2),
\end{align*}
where $\alpha \neq 0$.  (This also follows from the Jacobian calculation.)
For $i \geq 3$, we have, writing $[t^e]f(t)$ for the coefficient of $t^e$ in $f(t) \in \C[[t]]$,
$$
[t^e] R_i = \begin{cases} [t^e] u_i(t) & \mbox{ if $0 < e < a_{i1}$,} \\
[t^e] u_i(t) + \alpha^{a_{i1}} & \mbox{ if $e = a_{i1}$.}
\end{cases}
$$
and thus the equation $R_i = 0$ gives
$$
u_i(t) = 1 - \alpha^{a_{i1}} t^{a_{i1}} + O(  t^{a_{i1}+1}), \qquad i = 3,\ldots,n.
$$
Now let $a = \min_{i \geq 3}(a_{i1}) \geq 2$.  Then 
\begin{align*}
[t^{a}]R_1 &= [t^{a}] u_1(t) + [t^{a}] u_2(t) + \sum_{i \geq 3 \mid a_{i1} = a} a_{1i} (-\alpha^{a}) = 0 \\
[t^{a}] R_2 &= [t^{a}] u_1(t) + [t^{a}] u_2(t) = 0,
\end{align*}
and since $a_{1i} > 0$ for all $i$, it follows that $\alpha^a = 0$, contradicting our choice that $\alpha \neq 0$.  Thus there are no non-trivial power series solutions near $p$ to the equations $R_1=0,\ldots,R_n=0$, and it follows that $p$ is an isolated point.
\end{proof}

A pure $d$-dimensional simplicial complex $\Delta$ is a \emph{pseudomanifold} if any $(d-1)$-dimensional face $F \in \Delta$ is contained in exactly two facets of $\Delta$.
\begin{theorem}\label{thm:pseudo}
Suppose that $\tU$ is a binary geometry for $\Delta$.  Then $\Delta$ is a pseudomanifold.
\end{theorem}
\begin{proof}
Apply \cref{prop:Ulink} to a codimension-one face $F$ of $\Delta$, to deduce that $\tU_F$ is a one-dimensional binary geometry for $\lk_\Delta(\F)$.  Then it follows from \cref{thm:onedim} that $\lk_\Delta(F)$ consists of two vertices, or equivalently, $F$ is contained in two facets of $\Delta$.
\end{proof}

\subsection{Two-dimensional binary geometries}

\begin{proposition}\label{prop:prodbin}
Suppose $\tU$ and $\tU'$ are binary geometries for flag complexes $\Delta$ and $\Delta'$ on disjoint vertex sets.  Then $\tU \times \tU'$ is a binary geometry for $\Delta \times \Delta' = \{F \sqcup F' \mid F \in \Delta, F' \in \Delta'\}$.
\end{proposition}
In the product binary geometry, any vertex of $\Delta$ is compatible with any vertex of $\Delta'$.  By \cref{prop:prodbin} and \cref{thm:onedim}, there is a unique two-dimensional binary geometry that is a product of two one-dimensional binary geometries.  This is \cref{ex:square}.

Now let $\Delta$ be a $1$-dimensional flag simplicial complex with vertex set $[n]$ that is a pseudomanifold.  Then $\Delta$ can be viewed as a graph, and as a graph it is a disjoint union of cycles.  If we further assume that this graph is connected, then $\Delta$ is the face complex of the $n$-gon.  
If $n = 4$, we have \cref{ex:square}.  For $n =5$, we have $\tM_{0,5}$ with $u$-equations \eqref{eq:ueq5}.  We now give examples of binary geometries for the hexagon and the octagon.

\begin{proposition}[\cite{AHLT,AHLcluster}]\label{prop:B2}
The following $u$-equations give a binary geometry for $\Delta$ the face complex of the hexagon, with vertices labeled $u_1,v_1,u_2,v_2,u_3,v_3$ in cyclic order:
$$
u_i + u_{i+1} v_{i+1} u_{i+2} = 1, \qquad v_i + v_{i+1} u_{i+2}^2 v_{i+2} = 1,
$$
where the indexing of $u_i$ and $v_i$ is taken modulo $3$.
\end{proposition}
\begin{proposition}[\cite{AHLT,AHLcluster}]\label{prop:G2}
The following $u$-equations give a binary geometry for $\Delta$ the face complex of the octagon, with vertices labeled $u_1,v_1,u_2,v_2,u_3,v_3,u_4,v_4$ in cyclic order:
$$
u_i + u_{i+1} v_{i+1} u_{i+2}^2 v_{i+2} u_{i+3} = 1, \qquad v_i+ v_{i+1} u_{i+2}^3 v_{i+2}^2 u_{i+3}^3 v_{i+3} = 1,
$$
where the indexing of $u_i$ and $v_i$ is taken modulo $4$.
\end{proposition}

\begin{conjecture}
Let $\Delta$ be the face complex of a $n$-gon.  Then a binary geometry for $\Delta$ exists if and only if $n \in \{4,5,6,8\}$.
\end{conjecture}

\subsection{Positive binary geometry}

Let $\tU$ be a binary geometry.  The \emph{nonnegative part} $U_{\geq 0}$ (resp. \emph{positive part} $U_{>0}$) is the locus inside $\tU$ where $u_i \geq 0$ (resp. $u_i >0$) for all $i \in [n]$.  It follows from \eqref{eq:u} that $u_i(x) \in [0,1]$ for $x \in U_{\geq 0}$.  Thus 
$$(u_1,\ldots,u_n): U_{\geq 0} \hookrightarrow [0,1]^n$$ 
embeds $U_{\geq 0}$ into a hypercube.  We define $U_{F, \geq 0}:= U_F \cap U_{\geq 0}$ and $U_{F, >0}:= \oU_F \cap U_{>0}$. 

We call the binary geometry $\tU$ \emph{polytopal} if $U_{F,>0}$ is an open ball, $U_{F, \geq 0}$ is a closed ball, and $U_{F, \geq 0} = \bigcup_{G \supseteq F} U_{G, >0}$.  In this case $U_{\geq 0}$, together with the stratification $U_{F, \geq 0}$ is a regular CW complex.  Thus $\Delta$ is the dual of the face poset of this regular CW complex.

\begin{problem}
Find natural conditions for a binary geometry to be polytopal.
\end{problem}

\def\bU{{\bar U}}

Viewing $\tU \subset \C^n$ as contained in the closure $\bU \subset \P^n$, a projective variety, we may apply the \cref{def:posgeom} of a positive geometry.

\begin{definition}
Suppose that $\tU$ is $d$-dimensional and thus $\dim \Delta = d-1$.  Let $\Omega_U$ be a rational form $d$-form on $\tU$.  We say that $\Omega_U$ is a canonical form for $\tU$ if either
\begin{enumerate}
\item $d = 0$ and $U = \pt$ and $\Omega_U = \pm 1$, or 
\item
$d > 0$, and $\Omega_U$ has only poles along $U_{\{i\}}$, each of these poles is simple, and for all $i \in [n]$ we have $\Res_{u_i = 0} \Omega_U = \Omega_{U_{\{i\}}}$, where $\Omega_{U_{\{i\}}}$ is a canonical form for $U_{\{i\}}$ (see \cref{prop:Ulink}).  Furthermore, we require that $\Omega_U$ has no poles at the hyperplane at infinity $\P^n \setminus \C^n$.
\end{enumerate}
In this case we call $\tU$ a positive binary geometry.
\end{definition}

One way to construct a potential canonical form $\Omega_U$ is by positively parametrizing $U$.  Namely, let $y_1,\ldots,y_d$ be regular functions on $\oU$ which restrict to an isomorphism:
$$
(y_1,\ldots,y_d): U_{>0} \to \R_{>0}^d.
$$ 
Then $\Omega_U:= \bigwedge \frac{dy_i}{y_i}$ is a candidate canonical form.  Indeed, in many cases we have the following desirable situation: the functions $y_1,\ldots,y_d$ generate the rational function field $\C(\oU)$, and furthermore, each function $u_i$ is a rational function 
$$
u_i = \frac{a_i(y)}{b_i(y)}
$$
where $a_i(y)$ and $b_i(y)$ are polynomials with positive coefficients.  This is the case for the positive parametrization of \cref{sec:poscoord}, and we will revisit this situation in \cref{sec:string}.

\begin{remark}
According to the ``pushforward heuristic" of \cite[Heuristic 4.1]{ABL}, we more generally expect to be able to construct a canonical form if we are given a dominant rational map $(\C^\times)^d \longdashrightarrow U$, which restricts to a diffeomorphism $\R_{>0}^d \to U_{>0}$, even if it is not necessarily a birational map.
\end{remark}

\subsection{Exercises and Problems}

\begin{exercise}\
\begin{enumerate}[label=(\alph*)]
\item
Prove \cref{prop:B2} and \cref{prop:G2}.
\item
The binary geometries of \cref{prop:B2} and \cref{prop:G2} are positive binary geometries.  Find canonical forms for them.
\end{enumerate}
\end{exercise}

\begin{problem}\
\begin{enumerate}[label=(\alph*)]
\item
Show that there are no monomial relations among the generators $u_i$ of a binary geometry.  That is, we have $\prod_i u_i^{a_i} = 1$ if and only if $a_i = 0$ for all $i$.
\item
Is it true that the intrinsic torus of $\oU$ (\cref{sec:intrinsic}) has character lattice equal to the lattice $\lat = \{ \prod_i u_i^{a_i} \mid a_i \in \Z\}$ of Laurent monomials in the $u_i$?
\end{enumerate}
\end{problem}

We remark that binary geometries are defined over the integers since the equations \eqref{eq:u} are integral.  

\begin{problem}
Find the Euler characteristic of $\oU$ for binary geometries $\tU$.  When does $\oU$ have polynomial point count?  See \cref{thm:pointcount} and \cref{thm:Huh} for context.
\end{problem}

\begin{problem}\label{prob:Pell2}
Study the $u$-equations
\begin{align*}
&u_1 + u_6u_8 = 1, &&u_4 + u_3u_5=1,&&u_8+u_1u_2 = 1, &&u_3+u_4u_7 = 1 \\
&u_5 + u_2u_4u_6 = 1, &&u_2+u_5u_7u_8=1, &&u_6 + u_1u_5u_7=1, &&u_7+u_2u_3u_6=1
\end{align*}
for the simplicial complex $\Delta$ with maximal facets
$$\{123,124,135,147,157,236,246,358,368,468,478,578\}.$$
Can you prove that this is a binary geometry?  What can you say about $\tU$ and $\oU$?  How is this space related to $\M_{0,6}$?
This problem is revisited in \cref{sec:Pells}.
\end{problem}

\section{Scattering equations}
\label{sec:SE}
\subsection{Scattering amplitudes}
In particle physics, we are interested in predicting the outcomes of elementary particle scattering experiments.  Let $p_1,p_2,\ldots,p_n \in \R^D$ be the 
space-time momenta of $n$-particles in $D$-dimensional space-time.  The scattering amplitude is a function $A_n(p_1,p_2,\ldots,p_n)$ of the momentum vectors, and for the amplitude to be non-vanishing, the momenta $p_1,\ldots,p_n$ must satisfy momentum conservation:
\begin{equation}\label{eq:momcon}
p_1+ \cdots + p_n = 0.
\end{equation}
The function $|A_n(p_1,p_2,\ldots,p_n)|^2$ is the probability density of measuring an event with the specified momenta.

To determine the amplitude, we choose a quantum field theory (QFT), which is a choice of particles and a choice of possible interactions between these particles.  The scattering amplitude $A_n(p_1,p_2,\ldots,p_n)$ is then determined by specifying the type of each of the particles $1,2,\ldots,n$, together with information about the internal symmetries (``polarization vectors", ``quantum numbers") of each particle.  Roughly speaking, the type of a particle is a representation of some Lie group, and the additional internal information is a vector in this representation. 

In perturbative QFT, the scattering amplitude has an expansion
\begin{equation}\label{eq:Feyn}
A_n = \sum_{L =0}^\infty A_n^{(L)}, \qquad A_n^{(L)} = \sum_{\text{$L$-loop Feynman diagrams } \Gamma} \int_{\R^{L \times D}} \omega_\Gamma
\end{equation}
as a sum over the loop order $L$, where $A_n^{(L)}$ is a sum over graphs with $n$ leaves labeled $1,2,\ldots,n$, called $L$-loop Feynman diagrams.  An $L$-loop Feynman diagram is a graph with first Betti number equal to $L$.  The form $\omega_\Gamma$ is a rational differential form that depends on the momenta $p_1,\ldots,p_n$.  

In these lectures we will focus exclusively on the tree amplitude $A_n^{(0)}$, and often omit the exponent ${}^{(0)}$ in our notation.  The function $A_n^{(0)}$ is a rational function obtained as the sum over Feynman diagrams that are trees.

\subsection{Kinematic space}\label{sec:Kn}
Space-time $\R^D$ is equipped with a Lorentzian metric, a nondegenerate symmetric bilinear form with signature $(1,D-1)$.  We assume that physics, and therefore the scattering amplitude, is invariant under the orthogonal group $O_{1,D-1}(\R)$ preserving the Lorentz metric.  For example, this group includes usual rotations of space, and other linear transformations of space-time studied by Lorentz.

We will mathematically simplify the situation by complexifying the set up, simultaneously obscuring the physics.  Namely, we will consider space-time momenta as vectors $p_1,\ldots,p_n$ in a complex vector space $\C^D$ equipped with a symmetric bilinear form, and the amplitude $A_n(p_1,\ldots,p_n)$ is then a complex-valued function invariant under the (complex) orthogonal group $O(D)$, the complexification of $O_{1,D-1}(\R)$.  Let $P$ denote the $n \times D$ matrix with row vectors equal to $p_1,\ldots,p_n$.  By the first fundamental theorem of invariant theory, the ring of invariant (polynomial) functions $\C[p_1,\ldots,p_n]^{O(D)}$ is generated by the Mandelstam variables
$$
s_{ij} = (p_i+p_j)^2 = 2 p_i \cdot p_j = s_{ji}
$$
where $u \cdot v$ denotes the symmetric bilinear form on $\C^D$, and we use the standard physics convention that $u^2 := u \cdot u$.  The Mandelstam variables are also (up to a multiple of 2) the matrix entries of the $n \times n$ symmetric matrix $M := P P^T$.   The $n \times n$ matrix $M$ has rank at most $D$.  The second fundamental theorem of invariant theory states that the ideal of relations between the $s_{ij}$ is generated by the $(D+1) \times (D+1)$ minors of $M$.  We refer the reader to \cite{Pro} for more on the fundamental theorems of invariant theory.

The tree amplitude $A_n^{(0)} = A_n^{{\rm tree}}$ is thus a rational function of the $s_{ij}$.  Taking the dot product of \eqref{eq:momcon} with $p_i$, we obtain the relations 
\begin{equation}\label{eq:sijrel}
\sum_{j \neq i} s_{ij} = 0, \qquad \text{for each } i=1,2,\ldots,n.
\end{equation}
In fact, \eqref{eq:sijrel} is equivalent to \eqref{eq:momcon} in the following sense: any symmetric $n \times n$ matrix $M$ with rank $\leq D$ and row (and thus column) sums equal to 0 is of the form $M = P P^T$ where the rows of $P$ satisfy \eqref{eq:momcon}.  To see this, note that while $M$ may not be diagonalizable, it has an Autonne-Yakagi factorization $M = UDU^T$, where $U$ is unitary and $D$ is diagonal.  Since $M[1,1,\ldots,1]^T = 0$, we have $DU^T [1,1,\ldots,1]^T  = 0$.  We have $D = QQ^T$ for a $n \times D$ matrix $Q$, and thus $M = PP^T$ for $P = UQ$.    It is easy to see that $Q$ can be chosen so that $P^T [1,1,\ldots,1]^T = Q^T U^T [1,1,\ldots,1]^T = 0$, that is, $P$ satisfies momentum conservation. 

We will furthermore restrict ourselves to scattering of massless particles, that is, the momentum vectors satisfy $p_i^2 = 0$, or equivalently, $s_{ii} = 0$.  

\begin{definition}\label{def:Kn}
The (massless) kinematic space $K_n$ is the $\binom{n}{2} - n$ dimensional complex vector space of symmetric $n \times n$ matrices with diagonal entries equal to $0$, and row and column sums equal to $0$.  Equivalently, the dual $K_n^*$ of kinematic space is the vector space spanned by matrix entry functionals $s_{ij} = s_{ji}$ satisfying $s_{ii} = 0$ and the $n$ equations \eqref{eq:sijrel}.
\end{definition}

We will often view a point in $K_n$ as a collection $\s=(s_{ij})$ of complex numbers satisfying $s_{ii}=0$ and \eqref{eq:sijrel}.

Note that in our definition of kinematic space we have not fixed a dimension $D$ of space-time.  We will view scattering amplitudes in general space-time dimension as complex analytic functions on $K_n$.  Tree amplitudes $A_n^\tree$ are then rational functions on $K_n$.  In \cref{sec:4dim}, we return to discuss the special kinematics when we work in ($D=4$)-dimensional space-time of our real world.

\subsection{Scattering potential on $\M_{0,n}$}
Define the \emph{planar kinematic functions} $X_{ij} \in K_n^*$ by the formula
\begin{equation}\label{eq:Xij}
X_{ij} = \sum_{i \leq a < b \leq j-1} s_{ab},
\end{equation}
where $[i,j-1] = \{i,i+1,\ldots,j-1\}$ is a cyclic interval and the index of summation is considered cyclically.  For example, $X_{25} = s_{23}+s_{24}+s_{34}$.  We have that $X_{ij}$ vanishes when $(ij)$ is not a diagonal of the $n$-gon.  The $s_{ij}$ can be expressed in terms of $X_{ij}$ via the equation
\begin{equation}\label{eq:sij}
s_{ij} = X_{i,j+1}+X_{i+1,j} -X_{i,j} - X_{i+1,j+1}.
\end{equation}
Since $K_n^*$ has dimension $\binom{n}{2} - n$, we have the following.
\begin{proposition}
The planar kinematic functions $\{X_{ij} \mid (ij) \in \diag_n\}$ form a basis of $K_n^*$.
\end{proposition}

Furthermore, we may define dual lattices $K_n(\Z) \subset K_n$ and $K_n^*(\Z) \subset K^*_n$ as follows:
\begin{align*}
K_n^*(\Z) &:= \sp_\Z\{s_{ij}\} = \sp_\Z \{ X_{ij}\} \\
K_n(\Z)&:= \{(\s) \in K_n \mid f(\s) \in \Z \text{ for all } f \in K_n^*(\Z)\} = \{(\s) \in K_n \mid s_{ij} \in \Z\}.
\end{align*}

The relation between $X_{ij}$ and $s_{ij}$ can be viewed in terms of the cross-ratios $u_{ij}$ and minors $(\sigma_i-\sigma_j)$ on $\M_{0,n}$.  Namely, for any point in kinematic space $K_n(\Z)$, we have a regular function on $\M_{0,n}$:
\begin{equation}\label{eq:KBpotential}
\phi_X := \prod_{(ij) } u_{ij}^{X_{ij} }= \prod_{i < j} (\sigma_i - \sigma_j)^{s_{ij}},
\end{equation}
the equality following from \eqref{eq:sij} and \eqref{eq:defuij}.  Thus, we have:
\begin{proposition}\label{prop:Kniso}
The integral kinematic space $K_n(\Z)$ is naturally isomorphic to the character lattice $\lat(\M_{0,n})$ of \eqref{eq:L}.  Integral dual kinematic space $K^*_n(\Z)$ is naturally isomorphic to the cocharacter lattice $\lat^\vee(\M_{0,n})$.
\end{proposition}

We also have natural vector space isomorphisms $K_n \cong \Lambda \otimes_\Z \C$ and $K^*_n \cong \Lambda^\vee \otimes_\Z \C$.


\begin{corollary}
Any cross-ratio, or more generally any Laurent monomial in cross-ratios, can be uniquely written as a Laurent monomial in the dihedral coordinates $u_{ij}$, times a sign.
\end{corollary}
\begin{proof}
Any Laurent monomial in cross-ratios is of the form $\pm \prod_{i < j} (\sigma_i - \sigma_j)^{s_{ij}}$ for integers $s_{ij}$ that define a point in $K_n(\Z)$.  The sign $\pm$ arises since $(\sigma_i -\sigma_j) = - (\sigma_j - \sigma_i)$.
\end{proof}

\begin{definition}\label{def:KB}
For arbitrary $X_{ij} \in \C$, or $X \in K_n$, the \emph{Koba-Nielsen potential}, or \emph{scattering potential} is the (multi-valued) function \eqref{eq:KBpotential}.  
\end{definition}

This function is called a likelihood function \cite{ST}, or master function \cite{Var} in other contexts.  While $\phi_X$ is multi-valued, its (log-)critical point set in $\M_{0,n}$, the solutions to the equation
$$
\dlog \phi_X = \sum_{(ij) \in \diag_n }X_{ij} \dlog u_{ij} =  \sum_{(ij) \in \diag_n } \frac{X_{ij}}{u_{ij}} du_{ij} = 0,
$$
is a well-defined algebraic set.

\subsection{Cachazo-He-Yuan scattering equations}
The traditional approach to (perturbative) scattering amplitudes is via the Feynman diagrams of \eqref{eq:Feyn}.  Around ten years ago, Cachazo, He, and Yuan \cite{CHYarbitrary} introduced a new approach to the computation of tree-level scattering amplitudes by solving \emph{scattering equations}.  This approach is inspired by Witten's twistor string theory \cite{Wit} which leads to explicit formulae for tree-level Yang-Mills amplitudes \cite{RSV}.

The scattering equations are the following $n$ equations
\begin{equation}\label{eq:CHY}
Q_i:= \sum_{j \neq i} \frac{s_{ij}}{\sigma_i-\sigma_j} = 0, \qquad \text{for } i = 1,2,\ldots,n.
\end{equation}
Here, $s_{ij}$ denote the Mandelstam coordinates of a point in kinematic space $K_n$, and $\sigma_1,\ldots,\sigma_n$ are $n$ distinct points on $\P^1$.  By convention, if $\sigma_i = \infty$ then the term $\frac{s_{ij}}{\sigma_i-\sigma_j}$  is declared to be 0.

\begin{proposition}
The $n$ scattering equations \eqref{eq:CHY} are equivalent to the critical point equation $\dlog \phi_X = 0$.
\end{proposition}
\begin{proof}
Follows from \eqref{eq:KBpotential} and that $Q_i$ is up to sign equal to $\frac{\partial}{\partial \sigma_i} \log \phi_X$.
\end{proof}

\begin{proposition}\label{prop:CHYPGL2}
The scattering equations are invariant under the action of $\PGL(2)$ on $\{\sigma_1,\ldots,\sigma_n\}$.
\end{proposition}
\begin{proof}
\cref{ex:Mobius}.
\end{proof}

Thus we may view the $\sigma_i$ in \eqref{eq:CHY} as defining a point $\usigma \in \M_{0,n}$.  We view the scattering equations in the following three ways:
\begin{enumerate}
\item
If a point $\s = (s_{ij}) \in K_n$ has been fixed, we may solve these equations to find solutions $\usigma \in \M_{0,n}$.  This is the original perspective of \cite{CHYarbitrary}.
\item
If a point $\usigma \in \M_{0,n}$ has been fixed, we may solve these equations to find solutions $\s \in K_n$.  There are usually infinitely many such solutions, but restricting to $(n-3)$-dimensional subspaces of $K_n$ we will find that there is generically a unique solution, defining a \emph{scattering (rational) map} from $\M_{0,n}$ to a subspace of $K_n$.  This is the perspective of \cite{ABHY,AHLstringy}.
\item
The scattering equations define a scattering correspondence variety
\begin{equation}\label{eq:scatcorr}
\I = \{(\s,\usigma) \mid \eqref{eq:CHY} \text{ holds }\} \subset K_n \times \M_{0,n}.
\end{equation}
This will be useful in \cref{sec:4dim}.  Note that the equations \eqref{eq:CHY} are homogeneous in the variables $s_{ij}$, so the incidence variety is also naturally defined in $ \P(K_n) \times \M_{0,n}$.  This correspondence has appeared in likelihood geometry \cite{HS} but I have not seen it explicitly in the amplitudes literature.
\end{enumerate}

If a point $\s \in K_n$ has been fixed, then CHY postulate that the (tree-)amplitude rational function $A_n \in \C(K_n)$ is obtained as a sum of some function over the solutions of the scattering equations.  The choice of function is determined by the choice of QFT.  While the CHY scattering amplitudes can be considered for a variety of theories, we focus on scalar particles in this section.  Before we define the CHY amplitude, we first discuss the solutions to the scattering equations.

\subsection{Solving the scattering equations}
Let us suppose that $\s \in K_n$ has been fixed, or equivalently, $p_1,\ldots,p_n$ have been fixed.  We consider the solutions to \eqref{eq:CHY} in $\M_{0,n}$.

\begin{example}\label{ex:n4SE}
Suppose $n = 4$.  Then $\dim K_n = 4-2 = 2$, and using \eqref{eq:sijrel}, we have
$$
s_{12} = s, \qquad s_{13} = -s-t, \qquad s_{14} = t, \qquad s_{23} = t, \qquad s_{24} = -s-t, \qquad s_{34} = s.
$$
where $s = X_{13}$ and $t= X_{24}$.  We use the $\PGL(2)$ action to fix $\sigma_1 = 0, \sigma_2 = 1, \sigma_4 = \infty$.  The scattering equation for $i= 4$ is automatically satisfied, and the other equations become
$$
\frac{s}{-1} + \frac{-s-t}{-\sigma} = 0, \qquad \frac{s}{1} + \frac{t}{1-\sigma}=0, \qquad \frac{-s-t}{\sigma} + \frac{t}{\sigma-1} = 0,
$$
where $\sigma = \sigma_3$.  These three equations are equivalent, with solution $\sigma = \frac{s+t}{s}$.
\end{example}

\begin{theorem}\label{thm:n!}
For generic $\s \in K_n$, the scattering equations have $(n-3)!$ solutions.
\end{theorem}

\cref{thm:n!} follows from a theorem of Orlik and Terao \cite{OT} (in the setting of hyperplane arrangement complements), or its generalization by Huh (\cref{thm:Huh}) to very affine varieties.  These results state that the number of solutions to the scattering equations is (generically) equal to the absolute value of the Euler characteristic of $\M_{0,n}$\footnote{For a discussion of scattering equations when $\s$ become non-generic, see \cite{KKT}.}.  The Euler characteristic of $\M_{0,n}$ is easy to compute by counting points over $\F_q$.  After gauge-fixing $\sigma_1 = 0 ,\sigma_2 = 1,\sigma_n = \infty$, we find that there are $(q-2)$ choices for $\sigma_3$, and then $(q-3)$ choices for $\sigma_4$, and so on.  This gives 
$$
\# \M_{0,n}(\F_q) = (q-2)(q-3) \cdots (q-n+2).
$$
According to the next result, a consequence of the Weil conjectures, we have $\chi(\M_{0,n}(\C)) = (-1)^{n-3}(n-3)!$.

\begin{theorem}\label{thm:pointcount}
Suppose that we have an algebraic variety $X$ defined over $\Z$.  That is, $X$ is cut out by polynomial equations with integer coefficients.  Further suppose that $X_{\F_q}$ is smooth and has polynomial point count $f(q)$ for all prime powers $q =p^n$ of some prime $p$.  Then the Euler characteristic of the complex algebraic variety $X(\C)$ is given by $\chi(X(\C)) = f(1).$
\end{theorem}


Let us also sketch the original argument of \cite{CHYKLT} for \cref{thm:n!}, which is based on \emph{soft-limits}; see also \cite{ABFKST}.  We proceed by induction on $n$.  We gauge-fix $\sigma_1 = \infty, \sigma_2 = 0, \sigma_3 = 1$, and set $s_{nb}(\varepsilon) = \varepsilon s_{nb}$.  Substituting this into \eqref{eq:CHY}, we obtain a collection of $\varepsilon$-dependent equations.  Setting $\varepsilon = 0$, the last equation disappears, and we are left with $n -1$ equations that we can view as a set of CHY scattering equations for $n-1$ particles.  (Note that we have to adjust $s_{1j}$ to assure that momentum conservation is satisfied.).  We have $(n-1)$ equations, not depending on $\sigma_n$, and by induction there are $(n-4)!$ solutions to this system.  

Now, let $|\varepsilon|$ be nonzero but small.  Then the last equation is of the form
$$
\varepsilon \sum_{j=2}^{n-1} \frac{s_{nj}}{\sigma_n-\sigma_b} = 0.
$$
If $\sigma_1,\ldots,\sigma_{n-1}$ have been determined, then clearing denominators, this is a polynomial equation in $\sigma_n$ of degree $n-3$, and has $n-3$ solutions.  We now argue that in a small neighborhood of each of the $(n-4)!$ solutions $\usigma' = (\sigma_1,\ldots,\sigma_{n-1})$ there are $(n-3)$ solutions for $\usigma = ( \sigma_1,\ldots,\sigma_{n})$ that satisfy all the scattering equations.

\subsection{$\phi^3$-amplitude from scattering equations}\label{sec:CHYscalar}
Define the $n \times n$ matrix $\Phi$ by
$$
\Phi_{ab} := \partial_{\sigma_a} Q_b,
$$
with entries that are functions on $K_n \times \M_{0,n}$.  Since $Q_b = \partial_{\sigma_b} \log \phi_X$, the matrix $\Phi$ is a Hessian matrix.  
It follows from \cref{ex:CHYrel} that the matrix $\Phi$ has rank at most $(n-3)$.
Let $\Phi^{abc}_{pqr}$ be the $(n-3) \times (n-3)$ matrix obtained by deleting rows $a,b,c$ and columns $p,q,r$ from $\Phi$.

\begin{lemm}\label{lem:CHYred}
The \emph{reduced determinant}
$$
\det{}\!' \Phi := \frac{\det \Phi^{abc}_{pqr}}{\sigma_{ab}\sigma_{bc}\sigma_{ca}\sigma_{pq}\sigma_{qr}\sigma_{rp}}
$$
does not depends on the choice of $a,b,c,p,q,r$.
\end{lemm}
\begin{proof}\cref{ex:CHYrel}.
\end{proof}

\begin{definition}
The CHY scalar amplitude is the function on kinematic space $K_n$ given by
\begin{equation}\label{eq:CHYdef}
A_n := \sum_{{\rm solns}} \frac{1}{\detp \Phi} \frac{1}{(\sigma_{12} \sigma_{23} \cdots \sigma_{n1})^2}
\end{equation}
summed over the solutions to the scattering equations.
\end{definition}

\begin{example}\label{ex:CHY4}
Choose $\{a,b,c\} = \{1,2,4\} = \{p,q,r\}$ and use $\PGL(2)$ to place $(\sigma_1,\sigma_2,\sigma_4)$ at $(0,1,\infty)$.  Then $\Phi^{124}_{124}$ is a $1 \times 1$ matrix with entry 
$$
\frac{s+t}{\sigma^2} - \frac{t}{(\sigma-1)^2}.
$$
There are as many factors involving $\sigma_4 = \infty$ in the numerator as in the denominator of \eqref{eq:CHYdef}.  Substituting $\sigma = \frac{s+t}{s}$ from \cref{ex:n4SE}, we obtain
$$
A_4 = \frac{1}{\frac{s^2}{s+t} - \frac{s^2}{t}} \times \frac{1}{\frac{t}{s}^2} = - \frac{s+t}{st} = - \left(\frac{1}{s} + \frac{1}{t}\right) = - \left(\frac{1}{X_{13}} + \frac{1}{X_{24}}\right).
$$
\end{example}
Note that \cref{ex:CHY4} is deceptive: for $n > 4$, the solutions to the scattering equations are not rational functions in $X_{ij}$.  For $n=5$, it requires solving a quadratic equation.

The scalar amplitude defined by Cachazo-He-Yuan scattering equations recovers the scalar $\phi^3$-amplitude defined as a sum over Feynman diagrams.

\begin{theorem}[\cite{DG}]\label{thm:phi3}
We have, up to a sign,
\begin{equation}\label{eq:DG}
A_n = A_n^{\phi^3} = \sum_{T} \prod_{e \in E(T)} \frac{1}{X_{ij}}
\end{equation}
where the summation is over planar trees $T$ (\cref{ex:trees}) with $n$ leaves, interior vertices of degree three, and the product is over interior edges of $e$ which separate leaves $i,i+1,\ldots,j-1$ from $j,j+1,\ldots,i-1$.
\end{theorem}

The proof of \cref{thm:phi3} in \cite{DG} is by induction, verifying that both sides of satisfy a recursion.

The cubic planar trees appearing in \cref{thm:phi3} are the ``Feynman diagrams" \eqref{eq:Feyn} of a QFT called scalar $\phi^3$-theory, and $A_n = A_n^{(0)}$ is called the \emph{planar $\phi^3$-amplitude} (at tree level).  The choice of $\phi^3$-theory corresponds to the choice of degree 3 interior vertices.  If we imagine inflowing momenta $p_i$ at leaf $i$, and assume that momentum is conserved at every interior vertex, then 
$$
\mbox{momentum along edge }e = \pm (p_i + p_{i+1} + \cdots + p_{j-1}) = \mp (p_j + p_{j+1} + \cdots + p_{i-1}),
$$
and $X_{ij}$ is the dot product of this momentum vector with itself.  The appearance of squared momenta, called ``propagators", in the denominator is a general feature of the rational functions associated to Feynman diagrams.

$$
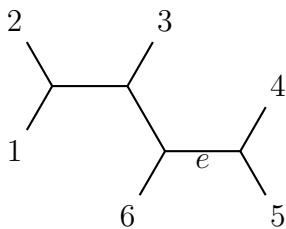
\begin{tikzpicture}
\coordinate (A) at (0,0);
\coordinate (B) at (1,0);
\coordinate (C) at (1.5,-0.866);
\coordinate (D) at (2.5,-0.866);
\draw[thick] (A) --(B)--(C)--(D);
\node (L1) at (-0.5,-0.866) {$1$};
\node (L2) at (-0.5,+0.866) {$2$};
\node (L3) at (1.5,+0.866) {$3$};
\node (L4) at (3,0) {$4$};
\node (L5) at (3,-2*0.866) {$5$};
\node (L6) at (1,-2*0.866) {$6$};
\draw[thick] (L1)--(A)--(L2);
\draw[thick] (L3)--(B);
\draw[thick] (L4)--(D)--(L5);
\draw[thick] (L6)--(C);
\node[align=left] at (2,-1) {$e$};
\node[text width=7cm] at (7,0) {The momentum traveling along the edge $e$ is equal to $\pm (p_4 + p_5) = \mp (p_1+p_2+p_3+p_6)$, with square given by $X_e = X_{46} = s_{45}$.};
\end{tikzpicture}
$$

\begin{remark}\label{rem:partial}
More generally, for a permutation $\alpha \in S_n$, we define a Parke-Taylor factor 
\begin{equation}\label{eq:PT}
\PT(\alpha) := \frac{1}{\sigma_{\alpha(1) \alpha(2)} \cdots \sigma_{\alpha(n)\alpha(1)}}.
\end{equation}
If we replace the factor $\frac{1}{(\sigma_{12} \sigma_{23} \cdots \sigma_{n1})^2}$ in \eqref{eq:CHYdef} by $\PT(\alpha)\PT(\beta)$ we obtain \emph{partial amplitudes}, which are sums over another (usually smaller) collection of trees; see \cref{ex:partial}.
\end{remark}

\def\g{{\mathfrak{g}}}
\def\Lie{{\rm Lie}}
\begin{remark}
Let $Y$ be a smooth complex algebraic variety and $f: Y \to \C$ be a function on $Y$.  Suppose that a Lie group $G$ acts freely on $Y$ and $f$ is a $G$-invariant function.  That is, for $g \in G$, we have the equality $f = f \circ g^*$, where $g^*: Y \to Y$ is the automorphism induced by the action of $g$.  Taking differentials, we get
\begin{equation}\label{eq:dfg}
df = df \circ dg^*.
\end{equation}
For a vector $v \in \g$ in the Lie algebra $\g = \Lie\, G$, we have a corresponding vector field $X_v$ on $Y$.  Then \eqref{eq:dfg} says that the $1$-form $df$ is annihilated by the vector field $X_v$.  Indeed, $df$ is annihilated by the $d = \dim G$ dimensional space of vector fields
\begin{equation}\label{eq:Xv}
\{X_v \mid v \in \Lie \, G\}.
\end{equation}
The quotient map $Y \to Y/G$ induces a map of tangent bundles $TY \to T(Y/G)$, with kernel equal to the subbundle with sections given by \eqref{eq:Xv}.  The $1$-form $df$ naturally descends to a section of the cotangent bundle $T^*(Y/G)$.

Taking $Y = (\P^1)^n \setminus \{{\rm diagonals}\}$, and $f = \phi_X$, and $G = \PGL(2)$, this explains the three-dimensional kernel of $\Phi$.
The reduced Jacobian $\detp \Phi$ of \cref{lem:CHYred} is a Jacobian associated to $f$, viewed as a function on $\M_{0,n} = Y/G$.
\end{remark}

\subsection{Delta functions} \label{sec:delta}
As we have explained in \cref{prop:Kniso}, the scattering potential, and thus the scattering correspondence \eqref{eq:scatcorr} is intrinsically associated to $\M_{0,n}$.  On the other hand, the reduced determinant of \cref{lem:CHYred} is coordinate dependent: it depends on choosing some of the $\sigma_i$ as coordinates.  We develop formalism to reformulate the definition \eqref{eq:CHYdef} in a more conceptual way.

We first introduce the formalism of ``holomorphic delta functions" that are common in the physics literature.  Recall that the delta function $\delta(x)$ is a distribution on the real line satisfying
$$
\int_{-\infty}^{\infty} h(x) \delta(x) dx = h(0).
$$
We will use this formula to motivate the definition of the \emph{holomorphic delta function}, as follows.
Let $X$ be an $n$-dimensional complex algebraic variety, and suppose that $f_1,f_2,\ldots,f_n$ are meromorphic functions defined on $X$.  (If any of the $f_i$ have singularities on $X$, by convention these singularities are removed form $X$.) 

\begin{definition}
Let $\omega$ be a holomorphic $n$-form on $X$.  Assume that the equations $f_1= f_2 = \cdots = f_n = 0$ have non-degenerate isolated solutions $\{p_1,\ldots,p_r\}$ on $X$.
The integral $\int \delta^n(f_1,f_2,\ldots,f_n) \omega$ is defined as:
$$
\int \delta^n(f_1,f_2,\ldots,f_n) \omega = \sum_{p \in \{p_1,\ldots,p_r\}} h(p) \det\left(\frac{\partial f_j}{\partial g_i}\right)^{-1}(p)
$$
where for each solution $p$, we let $g_1,\ldots,g_n$ be local coordinates such that in the neighborhood of $p$, we have 
$$
\omega(x) = h(x) dg_1 \wedge dg_2 \wedge \cdots \wedge dg_n.
$$
\end{definition}

If $f'_1= f'_2= \cdots = f'_n = 0$ defines the same solution set as $f_1= f_2 = \cdots = f_n = 0$, then 
\begin{equation}\label{eq:deltaJ}
\int \delta^n(f'_1,\ldots,f'_n) \omega = \int \delta^n(f_1,f_2,\ldots,f_n) \det\left(\frac{\partial f'_i}{\partial f_j}\right)^{-1} \omega 
\end{equation}
So we have the formula $\delta^n(f'_1,\ldots,f'_n) = \delta^n(f_1,f_2,\ldots,f_n) \det\left(\frac{\partial f'_i}{\partial f_j}\right)^{-1}$.  In particular, $\delta^2(f,g) = - \delta^2(g,f)$ and $\delta(\alpha f) = \delta(f)/\alpha$ for a scalar $\alpha$.  We also use the notation $\delta^n(f_1,f_2,\ldots,f_n)  = \delta(f_1)\delta(f_2) \cdots \delta(f_n)$.

\def\CHY{{\rm CHY}}
Now, define
\begin{equation}\label{eq:deltaCHY}
\delta^{\CHY}:=  \frac{\sigma_{ab} \sigma_{bc} \sigma_{ca} }{\sigma_{12} \sigma_{23} \cdots \sigma_{n1}} \prod_{i \notin \{a,b,c\}} \delta(Q_i),
\end{equation}
%
which (similar to \cref{lem:CHYred}) does not depend on the choice of $\{a,b,c\} \subset [n]$.  The following reformulation of the CHY amplitude is essentially the one given in \cite{BBBDF}.
\begin{proposition}\label{prop:BBB} The definition \eqref{eq:CHYdef} can also be written as
$$
A_n = \int \delta^{\CHY}\Omega_{0,n}
$$
where the three points $\sigma_a,\sigma_b,\sigma_c$ have been gauge-fixed, for example to $\{0,1,\infty\}$.
\end{proposition}

\subsection{General formalism for scalar amplitudes}
We suggest a formalism that produces scalar amplitudes just from the potential $\phi_X$, which is canonically associated to the variety $\M_{0,n}$.  There is some similarity with the approach of Mizera \cite{Miz} who worked in twisted (co)homology.

Let $X$ be an $n$-dimensional complex algebraic variety.  Let $\phi$ be a multi-valued function, which we call the potential function, and $\omega$ be a rational $n$-form on $X$.  Let $g_1,\ldots,g_n$ be local coordinates such that we have 
$$
\omega(x) = h(x) dg_1 \wedge dg_2 \wedge \cdots \wedge dg_n.
$$
We define the degree $n$ delta function $\delta^n(\phi; \omega)$ by
\begin{equation}\label{eq:deltan}
\delta^n(\phi; \omega) := h(x) \prod_{i=1}^n \delta\left( \frac{\partial \log \phi}{\partial g_i}\right).
\end{equation}
If $g'_1,\ldots,g'_n$ is another set of local coordinates, then $\omega = h'(x) dg'_1 \wedge dg'_2 \wedge \cdots \wedge dg'_n$ where 
$$
h'(x) = h(x) \det\left(\frac{\partial g_i}{\partial g'_j}\right).
$$
It thus follows from \eqref{eq:deltaJ} that $\delta^n(\phi; \omega)$ is well-defined.  

\begin{lemm}
Let $\phi$ be a potential with isolated non-degenerate critical points, and let 
$$
\omega = h(x) dg_1 \wedge dg_2 \wedge \cdots \wedge dg_n, \qquad \eta = r(x) dq_1 \wedge dq_2 \wedge \cdots \wedge dq_n
$$ 
be two $n$-forms.  We have
$$
\int \delta^n(\phi; \omega) \eta = \sum_p h(p) r(p) \det \left(\frac{\partial^2 \log \phi}{\partial g_i \partial q_j}\right)^{-1} = \int \delta^n(\phi; \eta) \omega.
$$
where the summation is over solutions $p$ to $\dlog \phi = 0$.
\end{lemm}

\begin{proposition}
We have $\delta^{n-3}(\phi_X;\Omega_{0,n}) = \delta^{\CHY}$.  Thus the CHY scalar amplitude is given by
$$
A_n = \int \delta^{n-3}(\phi_X;\Omega_{0,n})  \Omega_{0,n}.
$$
\end{proposition}
\begin{proof}
Fixing $\{\sigma_1,\sigma_2,\sigma_n\} = \{0,1,\infty\}$ the potential $\phi = \phi_X$ and form $\Omega_{0,n}$ become
$$
\phi = \prod_{j=3}^{n-1} (-\sigma_j)^{s_{1j}} (1-\sigma_j)^{s_{2j}} \prod_{3 \leq i <j \leq n-1} \sigma_{ij}^{s_{ij}},\;\;\;\Omega = \frac{1}{\sigma_{23} \sigma_{34} \cdots \sigma_{n-2,n-1}} d\sigma_3 \wedge \cdots \wedge d\sigma_{n-1}.
$$
For $i =3,4,\ldots,n-1$, we have that $\partial_{\sigma_i} \log \phi  = Q_i$, where the latter has been specialized to $\{\sigma_1,\sigma_2,\sigma_n\} = \{0,1,\infty\}$.  It follows that 
$$
 \delta^{n-3}(\phi;\Omega) = \frac{1}{\sigma_{23} \sigma_{34} \cdots \sigma_{n-2,n-1}}  \delta(Q_3)\delta(Q_4) \cdots \delta(Q_{n-1}).
$$
This equals $\delta^\CHY$, after noting that in \eqref{eq:deltaCHY}, two factors in both the numerator and denominator involve $\sigma_n = \infty$, and cancel out.
\end{proof}

\begin{example}
Let $n=4$ and write $\Omega_{0,4} = \frac{du}{u(1-u)}$ where $u = u_{13}$.  We have
$$
\phi = u^s (1-u)^t, \qquad \frac{\partial \log \phi}{\partial u} = \frac{s}{u} - \frac{t}{1-u}, \qquad \delta^1(\phi;\Omega) = \delta( \frac{s}{u} - \frac{t}{1-u}) \frac{1}{u(1-u)},
$$
where $s = X_{13}$ and $t = X_{24}$.  We calculate (noting that $\delta(af) = \delta(f)/a$)
\begin{align*}
\int \delta^1(\phi;\Omega) \Omega &= \int  \delta( \frac{s}{u} - \frac{t}{1-u}) \frac{du}{u^2(1-u)^2} \\
&= -\frac{1}{s+t} \int \delta(u - \frac{s}{s+t}) \frac{du}{u(1-u)} \\
& = - \frac{1}{s+t} \frac{1}{(s/(s+t))(t/(s+t))} = - (\frac{1}{s} + \frac{1}{t}) = -A_4^{\phi^3}(s,t).
\end{align*}
\end{example}

The above formalism suggests the following very general definition.
\begin{definition}\label{def:ampl}
Let $U$ be an $n$-dimensional very affine variety with character lattice $\lat$, and let $\phi$ be the intrinsic potential on $U$.  For two top-forms $\Omega,\Omega'$ on $U$, we define the (partial) amplitude on $\lat_\R$,
$$
A(\Omega|\Omega'):= \int \delta^n(\phi; \Omega) \Omega' = \int \delta^n(\phi; \Omega') \Omega =: A(\Omega'|\Omega).
$$
\end{definition}
We expect to obtain top-forms $\Omega,\Omega'$ from the general theory of positive geometries.  See \cref{sec:intrinsic} for definitions of $\lat$ and $\phi$, and for further discussion.

\subsection{From scattering correspondence to scattering form}
\def\sign{{\rm sign}}
\def\uc{{\underline{c}}}

Suppose we have a dominant rational map $f = (f_1,f_2,\ldots,f_n): X \longdashrightarrow Y$ between complex algebraic varieties of the same dimension $n$, and suppose $\omega$ is a rational $r$-form on $X$.  Then we define the \emph{pushforward} $f_*\omega$ (also called the \emph{trace}) as follows.   Suppose that the map $f$ has (generically) degree $d$.  Let $W \subset Y$ be an (analytic) open set such that $f^{-1}(W) = V_1 \sqcup V_2 \sqcup \cdots V_d$ is a disjoint union of open sets $V_1,\ldots,V_d \subset X$, where $f|_{V_i} : V_i \to W$ is biholomorphic.  We then define
$$
f_* \omega := \sum_{i=1}^d ((f|_{V_i})^{-1})^* \omega|_{V_i},
$$
where the form on the right hand side, defined as a rational form on $W$, is extended to a rational form on $Y$ by analytic continuation.

Let $y_1,y_2,\ldots,y_n$ be local coordinates on $Y$.  In these coordinates, we may write the rational map $f$ as $f = (f_1,f_2,\ldots,f_n)$.  If $f$ has degree $d$, the $n$-equations $f_1(x)-y_1= f_2(x)-y_2=\cdots=f_n(x)-y_n=0$ will generically have $d$ isolated solutions.  Suppose that $\omega$ has degree $n$, that is, $\omega$ is a top form.  Then the integral $\int \delta^n(f_1(x)-y_1,f_2(x)-y_2,\ldots, f_n(x)-y_n) \omega$ will be a well-defined rational function on $Y$.

\begin{proposition}
We have 
$$
f_* \omega = \left(\int  \delta^n( f_1(x)-y_1,f_2(x)-y_2,\ldots, f_n(x)-y_n) \omega\right) dy_1 \wedge \cdots \wedge dy_n.
$$
\end{proposition}

\begin{example}
Let $f : \C \to \C$ be given by $x \mapsto y = x^2$, and consider the $1$-form $\omega = \frac{dx}{x-a}$.  The pushforward is 
$$
f_*\omega = \omega|_{x = \sqrt{y}} + \omega|_{x = -\sqrt{y}} = \frac{d\sqrt{y}}{\sqrt{y}-a} +  \frac{-d\sqrt{y}}{-\sqrt{y}-a} = \frac{dy}{y-a^2}.
$$
On the other hand, using $\frac{\partial (x^2-y)}{\partial x} = 2x$, the delta function integral can be computed as
$$
\int \delta(x^2-y) \frac{dx}{x-a} =  \left( \frac{1}{2x(x-a)} \right)_{x \mapsto \sqrt{y}} + \left( \frac{1}{2x(x-a)} \right)_{x \mapsto -\sqrt{y}} = \frac{1}{y-a^2}.
$$
\end{example}

Recall the scattering correspondence $\I \subset K_n \times \M_{0,n}$ \eqref{eq:scatcorr}.  Consider the following diagram:
\[ \begin{tikzcd}
&\I \arrow{rd}{q} \arrow[swap]{ld}{p} & \\%
\M_{0,n} && K_n
\end{tikzcd}
\]

\begin{definition}\label{def:scatform}
The scattering form on $K_n$ is the $(n-3)$-form
$$
\Psi_n := q_* p^* \Omega_{0,n}.
$$
\end{definition}
The pullback $p^*\Omega_{0,n}$ is a rational $(n-3)$-form on $\I$.  By \cref{thm:n!}, the map $q: \I \to K_n$ is a degree $(n-3)!$ map between two varieties of dimension $\binom{n}{2} - n$, and thus the pushforward $q_*p^* \Omega_{0,n}$ is well-defined.

We now give an explicit formula for $\Psi_n$.  

\begin{theorem}\label{thm:scatform}
We have
\begin{equation}\label{eq:scatform}
\Psi_n = \sum_{\D} \sign(\D) \bigwedge_{(ij) \in\D} \dlog X_{ij}
\end{equation}
summed over triangulations $\D$ of the $n$-gon, or equivalently over maximal simplices in $\Delta_{0,n}$.  Here, the signs $\sign(\D) \in \{+,-\}$ are chosen so that we obtain an orientation of the simplicial complex $\Delta$.
\end{theorem}

There is an overall sign-ambiguity in \cref{thm:scatform}, matching the sign-ambiguity of $\Omega_{0,n}$ which depends on an ordering of the positive coordinates, or equivalently, an orientation of the manifold $(\M_{0,n})_{>0}$.

We spell out details of $\sign(\D)$.  The choice $\sign(\D)$ depends on the ordering of the elements $(ij) \in \D$ that appear in $\bigwedge_{(ij) \in\D} \dlog X_{ij}$, so that the form $\sign(\D) \bigwedge_{(ij) \in\D} \dlog X_{ij}$ depends only on $\D$, and not the ordering.  If we fix an ordering $(i_1j_1),(i_2,j_2),\ldots,(i_{n-3},j_{n-3})$ of the diagonals in $\D$, there is a unique other triangulation $\D'$ obtained by flipping the diagonal $(i_1,j_1)$, swapping it for $(i'_1,j'_1)$.  We ask that 
\begin{equation}
\label{eq:signflip}
\sign(\D) \dlog X_{i_1,j_1} \wedge \cdots \wedge \dlog X_{i_{n-3},j_{n-3}} = - \sign(\D') \dlog X_{i'_1,j'_1} \wedge \cdots \wedge \dlog X_{i_{n-3},j_{n-3}}
\end{equation}
for any two $\D,\D'$ related in this way.  This is equivalent to picking an orientation (that is, a generator for $H_{n-3}(|\Delta|,\Z)$) on the geometric realization $|\Delta|$ of $\Delta$, which is an orientable manifold (that is, $H_{n-3}(|\Delta|,\Z) \cong \Z$).  Indeed, $|\Delta|$ is a sphere, the boundary sphere of the dual of the associahedron polytope, which is orientable.

\begin{example}\label{ex:Psi4}
Let $n = 4$ and as usual we gauge-fix $(\sigma_1,\sigma_2,\sigma_4)$.  The canonical form is $\frac{d\sigma}{\sigma - 1} = \frac{d(\sigma-1)}{\sigma-1}$, and the scattering equations give $\sigma = \frac{X_{13} + X_{24}}{X_{13}}$.  Substituting, we get
$$
\Psi_4 = \dlog( \sigma - 1) = \dlog (X_{24}/X_{13}) =  \dlog X_{24} - \dlog X_{13}.
$$
\end{example}

The proof of \cref{thm:scatform} is given in \cref{sec:scatproof}.

\begin{remark}
The combinatorial formula in \cref{thm:scatform} is used as a definition of the scattering form in \cite{ABHY}.  Our definition of $\Psi_n$ as a push-pull gives the scattering form a conceptual explanation which can be extended to the general setting, for example of a very affine variety as in \cref{def:genscatform}.  See also \cite{FM} for a similar approach.
\end{remark}


Let us consider the subspace $H(\uc) \subset K_n$ given by the linear equations
\begin{equation}\label{eq:constant}
X_{ij} + X_{i+1,j+1} - X_{i,j+1} - X_{i+1,j} = - s_{ij} = c_{ij}
\end{equation}
where $c_{ij}$ is a constant, where $1 \leq i < j-1 < j \leq n-1$.  A quick calculation shows that these are $\binom{n-2}{2}$ linearly independent equations, so that $\dim H(\uc) = n-3$.

\begin{lemm}
Let $\D$ be a triangulation of the $n$-gon.  Then $\{X_{ij} \mid (ij) \in \D\}$ is a basis for $H(\uc)$.
\end{lemm}
 
\begin{proposition}\label{prop:pullbackAn}
Let $\iota_{\uc}: H(\uc) \hookrightarrow K_n$ denote the inclusion.
The pullback $d^{n-3}X := \iota_{\uc}^* \sign(\D) \bigwedge_{(ij) \in\D} \dlog X_{ij}$ does not depend on $\D$.
The pullback $\iota_{\uc}^* \Psi_n$ is, up to sign, equal to the planar $\phi^3$-amplitude
$$
\iota_{\uc}^* \Psi_n = \pm A^{\phi^3}_n \, d^{n-3}X.
$$
\end{proposition}
\begin{proof}
\cref{ex:pullbackAn}.
\end{proof}

The choice of the affine subspaces $H(\uc) \subset K_n$ may appear mysterious.  We explain them in \cref{sec:scatmap2}.  We remark that pullbacks of $\Psi_n$ to other subspaces of $K_n$ are also of physical interest, and are called \emph{partial amplitudes} (\cref{rem:partial}).  

\begin{example}
Let $n =4$.  Then the scattering form is (up to sign) the $1$-form 
$$
\Psi_n = \frac{dX_{13}}{X_{13}} - \frac{dX_{24}}{X_{24}}
$$
on the two-dimensional space $K_4$.  Pulling back to the one-dimensional affine subspace $X_{13} + X_{24} = c$, we have $-dX_{13} = d X_{24}$, so 
$$
\iota_{\uc}^* \Psi_n = (\frac{1}{X_{13}} + \frac{1}{X_{24}}) dX_{13},
$$
so the coefficient of $dX_{13}$ is exactly $A_4^{\phi^3}$.
\end{example}

\begin{example}
Let $n = 5$.  We may use the two linear functions $X_{13}, X_{14}$ as coordinates on $H(\uc)$, and then
$$
X_{35} = -X_{14} + c_{14} + c_{24}, \qquad X_{25} = - X_{13} + c_{13} +c_{14}, \qquad X_{24} = X_{14}-X_{13} + c_{13}
$$ 
on $H(\uc)$.  One checks that the sign works out so that all $ \iota_{\uc}^* \sign(\D) \bigwedge_{(ij) \in\D} \dlog X_{ij}$ is always $\pm dX_{13} dX_{14}$ with the same sign.
\end{example}

\begin{remark}
Recall from \cref{thm:pseudo} that if $\tU$ is binary geometry for $\Delta$, then $\Delta$ is a pure pseudomanifold.  In particular, it makes sense to ask that $\Delta$ is orientable, or oriented, using the same requirement as \eqref{eq:signflip}.  Now suppose that $\tU$ is a binary geometry for a pure, orientable, pseudomanifold $\Delta$.  We then have a combinatorial scattering form $\Psi_{\oU}$ on $\Lambda(\oU)_\R$, analogous to the formula in \cref{thm:scatform}, given by
$$
\Psi_{\oU} := \sum_{F \in \Delta} \sign(F) \bigwedge_{i \in F} \dlog X_i
$$
where the sum is over the maximal simplices $F$ of $\Delta$, and the differential forms $\sign(F) \bigwedge_{i \in F} \dlog X_i$ are chosen to give an orientation of $\Delta$.
\end{remark}

\subsection{Scattering map} \label{sec:scatmap}
\def\Ass{{\rm Ass}}
Restricting the scattering correspondence projection $\I \to K_n$ to the subspace $H(\uc)$, we obtain a rational map \cite{ABHY,AHLstringy}
\begin{equation}
\label{eq:scatPhi}
\Phi = \Phi(\uc): \M_{0,n} \longdashrightarrow H(\uc), \qquad (\usigma) \mapsto (X_{ab})
\end{equation}
of degree $(n-3)!$, between two varieties of dimension $(n-3)$.  The map $\Phi(\uc)$ can be described in coordinates explicitly.  First, use $\PGL(2)$ to take $\sigma_n$ to $\infty$, as usual.  Then the scattering equation $Q_n=0$ disappears, and summing $Q_1,\ldots, Q_k$, we obtain
$$
s_{k,k+1} = - \sum_{\substack{1 \leq i \leq k \\ k+1 \leq j \leq n-1 \\ (i,j) \neq (k,k+1)}} \sigma_{k,k+1} \frac{s_{ij}}{\sigma_{ij}}.
$$
Substituting this into \eqref{eq:Xij}, we can express each $X_{ab}$ as a sum over $s_{ij}$ where $(ij)$ is a diagonal:
\begin{equation}\label{eq:scatmap}
X_{ab} = - \sum_{\substack{1 \leq i < a\\a < j < b}} \sigma_{a,j} \frac{s_{ij}}{\sigma_{ij}} - \sum_{\substack{a \leq i <b \\ b \leq j <n}} \sigma_{i,b-1}  \frac{s_{ij}}{\sigma_{ij}} - \sum_{\substack{1 \leq i < a\\ b \leq j < n}} \sigma_{a,b-1}  \frac{s_{ij}}{\sigma_{ij}}.
\end{equation}
Note that this equality only holds if we assume the scattering equations $Q_i= 0$.  
For example, for $n = 5$, we may use the two linear functions $X_{13}, X_{14}$ as coordinates on $H(\uc)$, and then gauge-fixing $(\sigma_1,\sigma_2,\sigma_n)$ to $(0,1,\infty)$ as usual, the scattering map is
\begin{equation}\label{eq:n5scat}
X_{13} = \frac{c_{13}}{\sigma_3} + \frac{c_{14}}{\sigma_4}, \qquad X_{14} = \frac{\sigma_3 c_{14}}{\sigma_4} + \frac{(\sigma_3 -1)c_{24}}{(\sigma_4-1)}.
\end{equation}

\begin{theorem}\label{thm:pushforward}
We have $\Phi(\uc)_* \Omega_{0,n} = \iota_{\uc}^* \Psi_n$.  
\end{theorem}
\begin{proof}
The map $\Phi(\uc)$ is the restriction of the map $q$ to (the inverse image of) the subspace $H(\uc)$.  The definition of the pushforward of a form commutes with such pullbacks.  
\end{proof}

\begin{example}
For $n =4$, the scattering map is $X = \frac{c}{\sigma}$, where $X = X_{13}$, $\sigma = \sigma_3$ and $c = c_{13}$.  The canonical form is 
$$
\Omega_{0,4} = \dlog y_1 = \frac{d\sigma}{\sigma - 1},
$$
where $y_1 = \sigma - 1$. Thus substituting $\sigma = \frac{c}{X}$, we get
$$
\Phi(\uc)_* \Omega_{0,4} = -c \frac{dX}{X^2(c/X - 1)} = -c \frac{dX}{X(c - X)} = -\left(\frac{1}{X} + \frac{1}{c-X} \right)dX.
$$
This agrees (up to sign) with $A_4^{\phi^3} dX$, after using $X_{13} + X_{24} = c$.
\end{example}

\begin{remark}
There are two other proofs of \cref{thm:pushforward}.  For the first proof, combining with \cref{prop:pullbackAn} and \cref{thm:phi3}, the statement is equivalent to 
$$
A^{\phi^3}_n = \int \Omega((\M_{0,n})_{\geq 0})  \prod_{(ij) \in \D} \delta(X_{ij} - \Phi(\uc)_{ij}) 
$$
where $\D$ is a triangulation of the $n$-gon and $\Phi(\uc)_{ij}$ denotes the $X_{ij}$-coordinate of the point $\Phi(\uc) \in H(\uc)$.  Using \cref{prop:BBB}, this reduces to showing the identity of delta functions
$$
  \prod_{(ij) \in \D} \delta(X_{ij} - \Phi(\uc)_{ij}) = \frac{\sigma_{ab} \sigma_{bc} \sigma_{ca}}{\sigma_{12} \sigma_{23} \cdots \sigma_{n1}}\prod_{i \notin \{a,b,c\}} \delta(Q_i),
$$
which is a direct Jacobian calculation.

For the second proof, as explained in \cite[Section 7]{AHLstringy}, the map $\Phi:\M_{0,n} \to H(\uc)$ can be interpreted as the algebraic moment map of the toric variety of the associahedron $\Ass_{n-3}$ (see \cref{sec:string}).  The result then follows from the pushforward formula of \cite{ABL}, together with the identification of $ \iota_{\uc}^* \Psi_n$ as the canonical form $\Omega(\Ass_{n-3})$ of the $(n-3)$-dimensional associahedron $\Ass_{n-3}$, which is the intersection of the affine subspace $H(\uc)$ with the positive orthant $\{X_{ij} \geq 0 \mid (ij) \in \diag_n\} \subset K_n$.
\end{remark}


\begin{remark}\label{rem:Ass}
Another geometric statement \cite{ABHY,AHLstringy} is that when $c_{ij} >0$ are positive then $(\M_{0,n})_{\geq 0}$ is mapped diffeomorphically onto an associahedron polytope, and $\Phi_* \Omega((\M_{0,n})_{\geq 0})$ is the canonical form of the associahedron.  For example, with $n=5$, as $\usigma$ vary within $(\M_{0,5})_{> 0}$, we have $(\sigma_3,\sigma_4)$ in the region $1 < \sigma_3 < \sigma_4 < \infty$.  The image of this region under the map \eqref{eq:n5scat} in $H(\uc)$ is a pentagon, the two-dimensional associahedron.
\end{remark}

\subsection{Proof of \cref{thm:scatform}}\label{sec:scatproof}
Our proof is a variant of the proof of \cite[Theorem 7.12]{ABL}, which states that the pushforward of the canonical form of a projective toric variety $X_P$ under the algebraic moment map is equal to the canonical form of the corresponding polytope $P$.

Let $\Psi'_n$ denote the right hand side of \eqref{eq:scatform}.  We first note that both $\Psi_n$ and $\Psi'_n$ are projective, meaning that they are pullbacks of forms from $\P(K_n)$.  For $\Psi_n$ this follows from the fact that the scattering equations, and thus the scattering correspondence as well, are invariant under the dilation action of $\C^\times$ on $K_n$.  For $\Psi'_n$ this follows from the definition of the signs $\sign(\D)$; see \cref{ex:pullbackAn}.  Thus it suffices to show that $\Psi_n$ and $\Psi'_n$ have the same poles on $K_n$, and the same residues at these poles.  The poles of $\Psi'_n$ are along the hyperplanes $K_{ij} = \{X_{ij} = 0\}$.  The poles of $\Psi_n$ come from the image of the poles of $\Omega_{0,n}$.  By \cref{lem:M0npos} and \cref{thm:M0npos}, the poles of $\Omega_{0,n}$ are along the divisors $\tM_{ij} = \{u_{ij} = 0\}$ of $\tM_{0,n}$.  

As explained in \cref{sec:toricstringy} and \cite{AHLcluster}, the affine variety $\tM_{0,n}$ is the complement of a hypersurface $H$ in the toric variety $X_P$, where $P$ is an associahedron.  Under the identification $\tM_{0,n} \cong X_P\setminus H$, the canonical form $\Omega_{0,n}$ is identified with the canonical form of the torus of $X_P$.  Furthermore, the stratification of $X_P$ into torus orbit closures induces the stratification of $\tM_{0,n}$ into $\tM_{\D}$.  Let us consider the inclusions 
$$K_n \times \M_{0,n} \subset K_n \times \tM_{0,n} \subset K_n \times X_P$$
and taking the closure of the scattering correspondence we obtain inclusions $\I \subset \I' \subset \I''$.  The projection $\I'' \to K_n$ is a proper map of degree $(n-3)!$.  There is an open subset $a: V \hookrightarrow K_n$ such that the pullbacks $\I'' \times_{K_n} V$ and $\I \times_{K_n} V$ are isomorphic.  Thus $\Psi_n = q_* \Omega$, where $\Omega$ is the canonical form of $X_P$, pulled back to $\I''$.

By \cite[Section 7.1]{AHLstringy}, the restriction of $q$ to (the inverse image of) a subspace $H(\uc)$ can be identified with the algebraic moment map of $X_P$.  The divisors $\tM_{ij} = \{u_{ij} = 0\}$ are sent by $q$ to the facets of $P$, which are exactly the hyperplanes $K_{ij} = \{X_{ij} = 0\}$.  By \cite[Proposition 2.5]{KR}, taking residues commutes with pushforward, so we have
$$
\Res_{K_{ij}} \Psi_n = \Res_{K_{ij}} q_* \Omega = q_* \Res_{\tM_{ij}} \Omega = q_* (\Omega_{0,n_1} \wedge \Omega_{0,n_2})
$$
where the last equality is by \cref{lem:M0npos}.  On $\tM_{ij}$, all characters $u_{kl}$ where $(kl)$ crosses $(ij)$, are equal to 1.  So over $K_{ij}$, the potential restricts to
$$
\phi_X|_{K_{ij} \times \tM_{ij}} = \prod_{(ab)} u_{ab}^{X_{ab}}
$$
where the product is over all diagonals that do not cross $(ij)$.  This is exactly the potential of $\tM_{0,n_1} \times \tM_{0,n_2}$ in the notation of \cref{lem:M0npos}.  Thus the map $q: q^{-1}(K_{ij}) \to K_{ij}$ can be identified with the map
$$
q_1 \times q_2\times {\rm id}: (K_{n_1} \times X_{P,n_1}) \times (K_{n_2} \times X_{P,n_2}) \times W \to K_{n_1} \times K_{n_2} \times W
$$
where the direct sum decomposition $H_{ij} = K_{n_1} \times K_{n_2} \times W$ corresponds to diagonals contained in the polygon on one side of $(ij)$, the diagonals on the other side of $(ij)$, and the diagonals that cross $(ij)$.  We deduce that $\Res_{K_{ij}} \Psi_n = \Psi_{n_1} \times \Psi_{n_2}$.  On the other hand, it is immediate from the definition of $\Psi'_n$ that $\Res_{K_{ij}} \Psi'_n = \Psi'_{n_1} \times \Psi'_{n_2}$.  By induction on $n$, with the base case $n=4$ checked in \cref{ex:Psi4}, we conclude that $\Res_{K_{ij}} \Psi_n = \Res_{K_{ij}} \Psi'_n$ for any of the hyperplanes $K_{ij}$.  It follows that $\Psi_n = \Psi'_n$.

\subsection{Exercises and Problems}
\begin{exercise}\label{ex:Mobius}
Prove \cref{prop:CHYPGL2}, as follows.  The group $\PGL(2)$ acts on $\P^1$ as M\"obius transformations. For each $i=1,2,\ldots,n$, let $\sigma'_i = \frac{\alpha \sigma_i + \beta}{\gamma \sigma_i + \delta}$.  Show that $(\sigma_1,\ldots,\sigma_n)$ is a solution to the scattering equations if and only if $(\sigma'_1,\ldots,\sigma'_n)$ is.
\end{exercise}

\begin{exercise}\

\begin{enumerate}[label=(\alph*)]
\item
Fill out the details in the proof of \cref{prop:nullSE}.
\item
Write down conditions that guarantee that $r(z)$ in \eqref{eq:rz} is a regular morphism $r: \P^1 \to \P^{D-1}$ of degree $n-2$.
\end{enumerate}
\end{exercise}

\begin{exercise}
Solve the scattering equations for $n = 5$ by hand or by computer.
\end{exercise}

\begin{exercise}\label{ex:CHYrel}\
\begin{enumerate}[label=(\alph*)]
\item 
Show that $\sum_{i=1}^n \sigma_i^{\ell} Q_i = 0$ for $\ell = 0,1,2$.  Thus a solution to $n-3$ of the scattering equations $Q_i$ is automatically a solution to the remaining 3.
\item
Prove \cref{lem:CHYred} by using (a) and the following general statement.  Let $M$ be a $n \times n$ matrix and suppose that the nullspace $K$ of $M$ has dimension $k$.  Let $m^{a_1,\ldots,a_k}_{b_1,\ldots,b_k}$ denote the minor obtained from $M$ by removing rows $a_1,\ldots,a_k$ and columns $b_1,\ldots,b_k$.  Then we have
$$
m^{a_1,\ldots,a_k}_{b_1,\ldots,b_k} \Delta_{a'_1,\ldots,a'_k}(K) = m^{a'_1,\ldots,a'_k}_{b_1,\ldots,b_k}  \Delta_{a_1,\ldots,a_k}(K)
$$
where $\Delta_I(K)$ denote the Pl\"ucker coordinates of $K$ viewed as a point in the Grassmannian $\Gr(k,n)$.  A similar statement holds for changing $b_1,\ldots,b_k$.
\end{enumerate}
\end{exercise}

\begin{exercise}
Compute the Euler characteristics of $\tM_{0,n}$ and $\bM_{0,n}$ using the fact that Euler characteristic is additive under open-closed decompositions into algebraic sets.
\end{exercise}
%

\begin{exercise}\label{ex:pullbackAn}\
\begin{enumerate}[label=(\alph*)]
\item
Show that the formula \eqref{eq:scatform} for the scattering form $\Psi_n$ is cyclic invariant up to sign.  
\item 
Prove that $\Psi_n$ can be written as a polynomial in the forms $\dlog(X_{ij}/X_{kl})$.  This says that $\Psi_n$ is \emph{projective} in the sense of \cite{ABHY}.
\item
Prove \cref{prop:pullbackAn}.
\end{enumerate}
\end{exercise}

\begin{exercise}\label{ex:partial}
Let $\alpha$ be a cyclic ordering of $[n]$.  A (usual, cubic, $n$-leaf) planar tree $T$ is \emph{compatible} with $\alpha$ if there exists an embedding of $T$ into the plane so that the leaves are cyclically arranged in the order given by $\alpha$.  The partial amplitude $A_n(12\cdots n|\alpha)$ is given by
$$
A_n^{\phi^3}(123\cdots n|\alpha) := \sum_{T} \prod_{e \in E(T)} \frac{1}{X_{ij}}
$$
where the sum is over planar trees compatible with $\alpha$.
\begin{enumerate}
\item
Check for $n = 4$ and $n=5$, that the partial amplitudes can be obtained by using the Parke-Taylor factor $\PT(12\cdots n)\PT(\alpha)$ in \eqref{eq:CHYdef} instead of $\frac{1}{(\sigma_{12}\sigma_{23} \cdots \sigma_{n1})^2}$.
\item Check that $A_4(1234|2134)$ can be obtained by pulling back the planar scattering form $\Psi_4$ to the subspace $X_{24}=c$, for a constant $c$.  Which partial amplitude do you get by pulling back to $X_{13} = c$?
\item For each partial amplitude $A_5(12345|\alpha)$ find a subspace $\iota: H \hookrightarrow K_5$ such that the amplitude can be obtained by the pullback $\iota^* \Psi_5$.
\end{enumerate}
\end{exercise}

\begin{problem}
Express the reduced determinant of \cref{lem:CHYred} in terms of the dihedral coordinates $u_{ij}$.
\end{problem}

\begin{problem}
Explore the scattering equations and \eqref{eq:qz} within the setting of the moduli space of stable, rational, $n$-pointed maps to $\P^1$.
\end{problem}

\subsubsection{Tropical $u$-equations}
Recall the $u$-equations \eqref{eq:Rij} satisfied by the dihedral coordinates $u_{ij}$.  This equation has no minus signs (if we place the $1$ on the other side of the equality), and we can formally take its ``positive tropicalization", given by 
$$
u_{ij} \mapsto U_{ij}, \qquad + \mapsto \min, \qquad \times \mapsto +, \qquad {\rm constant} \mapsto 0 
$$
obtaining the \emph{tropical $u$-equation}
$$
\Trop(R_{ij}) := \min(U_{ij}, \sum_{(kl) \text{ crosses } (ij)} U_{kl}) = 0.
$$

We now take $\{U_{ij} \in K_n\}$ to be the basis dual to $\{X_{ij} \in K_n^*\}$.  Thus the tropical $u$-equations define a subset of $K_n^*$.  
\begin{lemm}\label{lem:tropu}
The intersection of the $\binom{n}{2}-n$ tropical $u$-equations is a pure polyhedral fan $\Trop_{\geq 0} \M_{0,n}$ in $K_n^*$ of dimension $n-3$, with maximal cones given by 
$$
C(\D) = \sp_{\geq 0}(X_{ij} \mid (ij) \in \D) \subset K_n^*
$$
for $\D$ a triangulation of the $n$-gon.
\end{lemm}
\begin{proof}
Let $X \in K_n^*$ belong to the intersection of $\Trop(R_{ij}) = 0$.  Clearly, $U_{ij}(X) \geq 0$ for all $(ij)$.  Let
$$
\D_X = \{(ij) \mid U_{ij}(X) >0\}.
$$
If $(ij)$ and $(kl)$ are diagonals that cross, then $\Trop(R_{ij}) = 0$ implies that they cannot both be elements of $\D$.  It follows that the possible choices of $\D_X$ are exactly the subdivisions of the $n$-gon, and conversely, if $\D_X$ is a subdivision of the $n$-gon then $X$ satisfies the tropical $u$-equations.  For a fixed subdivision $\D'$ of the $n$-gon we obtain a cone $C(\D')$ with dimension equal to the number of diagonals used, and the set we seek is the union of these cones.  The maximal cones correspond to triangulations $\D$.
\end{proof}

The fan $\Trop_{\geq 0} \M_{0,n}$ is the \emph{positive tropical pre-variety} associated to the equations $R_{ij} = 0$.  It turns out that $\Trop_{\geq 0} \M_{0,n}$ is equal to the positive tropicalization of $\M_{0,n}$ in the sense of \cite{SW}.  Note that the rays (that is, the one-dimensional cones) of the fan $\Trop_{\geq 0} U$ are exactly $\{\R_{\geq 0} \cdot X_{ij}\}$.  The monoid of bounded characters (\cref{ex:intrinsic}) on $\M_{0,n}$ are exactly the characters taking nonnegative values on these rays.  

Let $\Delta = \Delta_{0,n}$ be the simplicial complex of subdivisions of the $n$-gon, as in \cref{sec:MD}.  Then the tropical $u$-equations cut out the \emph{cone over $\Delta$}.  This suggests that for binary geometries the following is a reasonable definition of positive tropicalization.

\def\pre{{\rm pre}}
\begin{definition}\label{def:tropu}
Let $U$ be a binary geometry for a flag simplicial complex $\Delta$ on $[n]$.  The \emph{positive tropical pre-variety} $\Trop^\pre_{\geq 0} U$ of $U$ is the intersection of the tropical $u$-equations $\Trop(R_i) = 0$ for $i=1,2,\ldots,n$.
\end{definition}

It follows from the same calculation as in \cref{lem:tropu} that $\Trop^\pre_{\geq 0}U$ is a pure polyhedral fan whose cones are exactly
$$
C(F) :=  \sp_{\geq 0}(X_{i} \mid i \in F).
$$
We do not know (\cref{prob:binarytrop}) the relation between $\Trop^\pre_{\geq 0} U$ and the usual notion of (positive) tropicalization (\cref{sec:postrop}).  The intersection of the tropical $u$-equations is analogous to the definition of the \emph{(positive) Dressian} \cite{ALS,SWDress}.

\begin{problem}\label{prob:binarytrop} 
How is the positive tropical pre-variety $\Trop^\pre_{\geq 0} U$ in \cref{def:tropu} related to the positive tropicalization of the very affine variety $U$?  
\end{problem}

The study of the tropicalization of binary geometries has recently been initiated by Cox and Makhlin \cite{CM}.

\section{Positive geometry of four dimensional space-time}
\label{sec:4dim}

In \cref{sec:SE}, we looked at the the kinematics of $n$-particle scattering in $D$-dimensional space-time.  Space-time of the real world is four-dimensional, so the case $D = 4$ is of particular physical importance, so we specialize to $D =4$ in this section.  For physics background to the material of this section, see \cite{Dix, Bri, EH}.

\subsection{Spinors}
  In \cref{def:Kn}, kinematic space $K_n$ for $n$-particle scattering is defined to be a complex vector space of dimension $\binom{n}{2} - n$.  This is viewed as the space of symmetric $n \times n$ matrices of the form $M = P P^T$, where the row vectors of the matrix $P$ are the momentum vectors $p_1,\ldots,p_n \in \C^D$.  For low $D$, the rank of the matrix $M$ is bounded above by $D$, and the ``correct" kinematic space is the subvariety $K_n^D \subset K_n$ cut out by the conditions that all $(D+1) \times (D+1)$ minors vanish.

Setting space-time dimension to $D = 4$ leads to the remarkable algebra and geometry of spinors.  We view $\C^4 \cong \Mat_{2 \times 2}$ as the space of $2 \times 2$ matrices, with isomorphism chosen so that the symmetric bilinear form is given by $p^2 = p \cdot p = \det(p)$.  Equivalently, the inner product is
$$
\begin{bmatrix} a & b \\ c & d \end{bmatrix} \cdot \begin{bmatrix} a' & b' \\ c' & d' \end{bmatrix} = \frac{1}{2}(ad' + a'd - bc' - b'c).
$$
The massless condition that $p^2 = \det(p) = 0 $ implies that $p$ has rank $\leq 1$ and can be written as $p = \lambda \tlambda$ where $\lambda$ (resp. $\tlambda$) is a column (resp. row) vector in $\C^2$.  The vectors $\lambda,\tlambda \in \C^2$ are called \emph{spinors}.  

There is a redundancy in this description: for a non-zero scalar $t \in \C^\times$, the spinors $(\lambda,\tlambda)$ and $(t\lambda,t^{-1}\tlambda)$ represent the same momentum vector $p$.  This feature is (perhaps?) a slight annoyance when we consider scattering of scalar particles, but an advantage when we consider scattering of gluons.  Spinor coordinates are also preferred when considering scattering of fermions (spin $1/2$ particles), though we will only consider gluon scattering.  For more on \emph{spinor-helicity formalism}, see the textbook \cite{EH}.

\begin{lemm}\label{lem:spinordet}
Suppose the spinors of $p$ and $q$ are $(\lambda,\tlambda)$ and $(\mu, \tmu)$ respectively.  Then
$$
2p \cdot q = \det(\lambda,\mu) \det(\tlambda,\tmu)
$$
where the determinants are taken by forming $2 \times 2$ matrices, using two column (resp. row) vectors.
\end{lemm}
\begin{proof}
\cref{ex:spinordet}.
\end{proof}

Now, let $\lambda_1,\ldots,\lambda_n$ and $\tlambda_1,\ldots,\tlambda_n$ be spinors for the $n$ particles.  We use the standard shorthand:
$$
\langle i j \rangle := \det(\lambda_i,\lambda_j), \qquad [i j] := \det(\tlambda_i,\tlambda_j)
$$
satisfying the identities
$$
\langle i j \rangle = - \langle j i \rangle, \qquad [ij] = -[ji], \qquad s_{ij} = \langle i j \rangle [i j].
$$
Momentum conservation $p_1+p_2+ \cdots+ p_n = 0$ becomes the identity $\sum_i \lambda_i \tlambda_i = \begin{bmatrix} 0 & 0 \\ 0 & 0 \end{bmatrix}$, which combining with \cref{lem:spinordet} gives
\begin{equation}\label{eq:spinorMC}
\sum_{i} \langle j i \rangle [i k] = 0, \text{ for all } j, k.
\end{equation}

Let $\SO(D)$ denote the complex special orthogonal group in $D \geq 3$ dimensions.  The group $\SO(D)$ is not simply-connected and its simply-connected cover is the spin group $\Spin(D)$, which is a double cover of $\SO(D)$.  When $D=4$, we have an exceptional isomorphism $\Spin(4) \cong \SL(2) \times \SL(2)$ arising from the Dynkin diagram isomorphism $D_2 \cong A_1 \times A_1$.  The covering map $\SL(2) \times \SL(2) \to \SO(4)$ has kernel given by the two element group $\{(I,I), (-I,-I)\}$.  The $\SO(4)$ action on $\C^4$ is replaced by the $\SL(2) \times \SL(2)$ action on the pair $(\lambda, \tlambda)$.

We now view $\lambda = (\lambda_1,\ldots,\lambda_n)$ as a $2 \times n$ matrix and $\tlambda = (\tlambda_1,\ldots,\tlambda_n)$ as a $n \times 2$ matrix.  By the first fundamental theorem of invariant theory, the ring of $\SL(2)$-invariant polynomial functions on $2 \times n$ (resp. $n \times 2$) matrices is generated by the $2 \times 2$ determinants $\langle i j \rangle$ (resp. $[ij]$).  This ring, also called the \emph{Pl\"ucker ring}, is the coordinate ring of the cone over the Grassmannian $\Gr(2,n)$.  Thus the kinematics of $n$ massless particle scattering in $D=4$-dimensional space-time involves two 2-planes $\lambda,\tlambda$ in $n$-dimensional space, and \eqref{eq:spinorMC} says that $\lambda,\tlambda$ are orthogonal 2-planes (\cref{ex:orthogrel}).

Let $\hGr(2,n)$ denote the cone over the Grassmannian $\Gr(2,n)$.  Thus the coordinate ring of $\hGr(2,n)$ is the ring generated by $\langle i j \rangle$.
\begin{definition}
\emph{Spinor kinematic space} is the space $\tK_n^4$ of pairs $(\lambda,\tlambda) \in \hGr(2,n) \times \hGr(2,n)$ satisfying $\lambda \perp \tlambda$.\end{definition}

For more on the algebraic geometry of spinor kinematic space, we refer the reader to the recent work \cite{EPS}.

\begin{remark}
In $D=3$ dimensions, we have the exceptional isomorphism $\Spin(3) \cong \SL(2)$, arising from the Dynkin diagram isomorphism $A_1 \cong B_1$.  In $D=6$ dimensions, we have the exceptional isomorphism $\Spin(6) \cong \SL(4)$, arising from the Dynkin diagram isomorphism $A_3 \cong D_3$.  Though we will focus on $D=4$, there is special physics appearing for all three choices of space-time dimension $D = 3,4,6$. 
\end{remark}

\subsection{Gluon scattering}
Elementary particles come with a discrete parameter, \emph{spin}, which takes values in $0,\frac{1}{2},1,\frac{3}{2},2,\ldots$.  Spin 0 particles are also called scalar particles, and the amplitudes that we looked at in \cref{sec:CHYscalar} are scattering amplitudes for scalars.  A particle with nonzero spin has intrinsic angular momentum -- it is allowed to be rotating, and the axis of rotation gives additional degrees of freedom.  For scattering of particles with nonzero spin, the amplitudes depend on not just the momentum vector of the particle, but also additional information.  In the case of spin 1 particles, the additional information is called a polarization vector.  Thus scattering amplitudes are functions of momentum vectors $p_1,\ldots,p_n$ and polarization vectors $\xi_1,\ldots,\xi_n$.

Gluons are elementary particles that mediate the strong force, that is, the exchange of gluons leads to an interaction between elementary particles called the strong force.  Gluons appearing in the strong force are massless spin 1 particles with ``color", and we commonly call any such particle in a QFT a gluon.  Color refers to the symmetries of a compact Lie group called a gauge group.  In the case of the strong force, this gauge group is ${\rm SU}(3)$.  The dependence of the scattering amplitude on the gauge group can often be separated out in the sense that the amplitude can be factored as the product of a ``group theory factor" and a ``kinematic factor".  We ignore the gauge group in our presentation and focus on the kinematics.  That is, we consider ``color-stripped amplitudes".  However, let us remark that some of the beautiful combinatorics and geometry of amplitudes arises exactly because of the gauge group.  For example, taking the rank of the gauge group to infinity (that is, ${\rm SU}(N)$ with $N \to \infty$) leads to the consideration of ``planar amplitudes".

For gluon scattering, polarization vectors have two degrees of freedom.  We may consider the polarization vector as a linear combination of a positive helicity polarization vector and a negative helicity one.  So to compute amplitudes, it suffices to compute the $2^n$ helicity amplitudes, where each of the $n$ particles is given either positive or negative helicity.  These amplitudes are typically denoted $A_n(1^+,2^-,3^-,\ldots,n^+)$ (particle 1 has helicity $+$, particle 2 has helicity $-$, ...) and so on.

It is a remarkable and beautiful fact that the additional information of the polarization vector, or equivalently of helicity, is stored in the choice of spinors $(\lambda,\tlambda)$ (which previously appeared redundant).   Let $T \cong (\C^\times)^n$ act on $\tK_n^4$ by letting $t = (t_1,\ldots,t_n) \in T$ send $(\lambda_i,\tlambda_i)$ to $(t_i\lambda_i,t_i^{-1} \tlambda_i)$.  Let $h = (h_1,\ldots,h_n) \in \{\pm 1\}^n$ denote a helicity vector.  Then we have the following assertion:
\begin{equation}\label{eq:helicity}
\parbox[c]{\textwidth}{Helicity $h$ amplitudes are weight vectors for the $T$-action with weight $(2h_1,\ldots,2h_n)$.}
\end{equation}
In other words, helicity $h$ amplitudes are rational functions on $\tK_n^4$ that scale by $\prod_i t_i^{2h_i}$ under the $T$-action on $\tK_n^4$.  We can also state this projectively, using the projective variety $\Proj(\C[\tK_n^4]) \subset \Gr(2,n) \times \Gr(2,n)$.

\begin{definition}
For an appropriate line bundle $L$ on $\Proj(\C[\tK_n^4])$, helicity $h$ amplitudes are rational sections of $L$ with $T$-weight equal to $(2h_1,\ldots,2h_n)$.
\end{definition}

The coordinate ring $\C[\tK_n^4]$ of the variety $\tK_n^4$ is the multi-homogeneous coordinate ring of (or multi-cone over) the flag variety $\Fl(2,n-2;n)$ of partial flags $\{(0 \subset \lambda \subset \tlambda^\perp \subset \C^n)\}$ of dimensions $2,n-2$ in $\C^n$.  (Here, we have used the isomorphism $\Gr(2,n) \cong \Gr(n-2,n)$ sending $\tlambda$ to $\tlambda^\perp$.)  The action of $T$ on $\Fl(2,n-2;n)$ is the usual one, though to preserve the conventions on weights we should take only the subtorus $\{(t_1,\ldots,t_n) \mid \prod_i t_i = 1\}$.  We could also think of helicity amplitudes as rational sections of a line bundle on an appropriate GIT quotient $\Fl(2,n-2;n)/\!\!/T$, though I do not know how useful this point of view is.

\subsection{Scattering equations from marked rational maps}
We briefly review the origin of the scattering equations, and in this subsection we work in general space-time dimension $D$.  Let us consider the following rational map from a $n$-punctured Riemann sphere $\P^1 \setminus \{\sigma_1,\ldots,\sigma_n\}$ to $\C^D$:
\begin{equation}\label{eq:qz}
q(z) := \sum_{i=1}^n \frac{p_i}{z - \sigma_i}, \qquad z \in \P^1 \setminus \{\sigma_1,\ldots,\sigma_n\}.
\end{equation}
The map \eqref{eq:qz} depends on a choice of $\usigma$ and on momentum vectors $p_1,\ldots,p_n \in \C^D$, which we assume are massless and satisfy momentum conservation.  Let 
\begin{equation}\label{eq:rz}
r(z) = q(z) \prod_{i}(z-\sigma_i) =  \sum_{i=1}^n p_i \prod_{j \neq i}(z - \sigma_i),
\end{equation}
which is now a polynomial in $z$.  Due to the momentum conservation assumption, we have that $r(z)$ has degree $ \leq n-2$.  We may view $r(z)$ as a rational map $r: \P^1 \to \P^{D-1}$.  The specific form \eqref{eq:rz} says that $r(\sigma_i) = p_i$, now viewed as a point in $\P^{D-1}$.  Let $(K_n \times \M_{0,n})^\circ \subset K_n \times \M_{0,n}$ denote the dense open subset where $r(z)$ is a regular morphism of degree $n-2$.  We may then consider the formula for $r(z)$ (or $q(z)$) as defining a morphism
$$
(K_n \times \M_{0,n})^\circ \longrightarrow \M_{0,n}(\P^{D-1}, (n-2)),
$$
where $\M_{0,n}(\P^{D-1}, (n-2))$ is the moduli space of $n$-pointed rational curves of degree $n-2$ in $\P^{D-1}$.  It would be interesting to study the scattering equations within the context of the moduli space $\overline{\M}_{0,n}(\P^{D-1}, (n-2))$ of stable rational curves.

The null-cone, or light-cone, of $\C^D$ is the subset of $\C^D$ consisting of vectors satisfying $p^2 = p \cdot p = 0$ (the massless condition).  It is a quadric in $\C^D$, singular only at the origin.

\begin{proposition}\label{prop:nullSE}
The scattering equations are equivalent to the condition that $r(z)$ has image in the null-cone of $\C^D$.
\end{proposition}

\begin{proof}[Sketch Proof]
Differentiating $r(z)^2 = 0$, we obtain the condition $r'(z) \cdot r(z) = 0$.  Evaluating this equation at $z= \sigma_i$ gives equations equivalent to the scattering equations.
\end{proof}

\subsection{Scattering equations in four dimensions}
When $D=4$, we identify $\C^4 \cong \Mat_{2\times 2}$ and ask that $r(z)$ has image in rank one matrices.  Alternatively, we view $r(z)$ as a degree-$(n-2)$ map 
$$
r: \P^1 \setminus \{\sigma_1,\ldots,\sigma_n\} \to \P^3
$$
and ask that the image of $r$ lies in the quadric $\P^1 \times \P^1 \subset \P^3$.  

In this case, we can write the $2 \times 2$ matrix $r(z)$ as $r(z) = \tau(z) \ttau(z)$ where $\tau(z)$ is a column vector of polynomials, and $\ttau(z)$ is a row vector of polynomials.  This representation is unique (up to rescaling $(\tau,\ttau) \mapsto (t \tau, t^{-1} \ttau)$) when the four matrix entries of $r(z)$ do not have a common root, or equivalently, when $r(z) = 0$ has no solutions.  Additionally imposing that $r(z)$ has degree equal to $n-2$, this is exactly the situation when $r(z)$ defines a genuine map $\P^1 \to \P^1 \times \P^1 \subset \P^3$ of degree $n-2$.  The polynomials $\tau(z),\ttau(z)$ are rational maps $\tau, \ttau: \P^1 \to \P^1$ obtained by composing $r$ with the projections to one of the factors. 

Recall the scattering correspondence $\I \subset K_n \times \M_{0,n}$.  We may restrict the family $\I \to K_n$ to the subvariety $K_n^4 \subset K_n$ of rank $\leq 4$ matrices.  Let $P_n^4$ denote the space of $n \times 4$ matrices $P$ whose rows are massless momentum vectors $p_1,\ldots,p_n$ satisfying momentum conservation.  Then we have a map $P_n^4 \to K_n^4$ (sending $P \mapsto PP^T$) and we can pull the family $\I$ back to a family $\tI' \to P_n^4$.  We let $\tI \subset P_n^4 \times \M_{0,n}$ denote those irreducible components of $\tI'$ that map dominantly to $P_n^4$.
Thus the map $\tI \to P_n^4$ is a morphism of degree $(n-3)!$; the preimages of a point in $P_n^4$ are the solutions to the scattering equations.

\begin{proposition}\label{prop:Jd}
The correspondence $\tI \subset P_n^4 \times \M_{0,n}$ has $n-3$ irreducible components $\tJ_1,\tJ_2,\ldots,\tJ_{n-3}$, all of the same dimension $3n-4 = \dim P_n^4$.  For $d = 1,2,\ldots,n-3$, the $d$-th irreducible component $\tJ_d$ corresponds to factorizations $r(z) = \tau(z) \ttau(z)$ where $\tau(z)$ has degree $d$ and $\ttau(z)$ has degree $\td=n-2-d$.
\end{proposition}
\begin{proof}
Given a point in $\tI$, sitting over a generic point $P \in P_n^4$, we have a corresponding map $r(z) = \tau(z)\ttau(z)$ defined by \eqref{eq:rz}, mapping $ \P^1 \setminus \{\sigma_1,\ldots,\sigma_n\}$ to the quadric $\P^1 \times \P^1 \subset \P^3$.  Since $r(z)$ has degree $n-2$, if $\tau(z)$ has degree $d$ then $\ttau(z)$ has degree $\td =n-2-d$.  First note that for a generic point $p_1,\ldots,p_n \in P_n^4$, neither $\tau(z)$ nor $\ttau(z)$ can be constant.  This is because by \eqref{eq:rz} we have $r(\sigma_i) = p_i$, and this would force all the $\lambda_i$ to be parallel, or all the $\tlambda_i$ to be parallel, which is generically not the case.  Thus $1 \leq d \leq n-3$.

The possible maps $\tau: \P^1 \to \P^1$ (resp. $\ttau$) is determined by a pair of polynomials of degree $\leq d $ (resp. $\leq \td = n-d-2$), belonging to a $2(d+1)$ (resp. $2(\td+1)$) dimensional affine space.  Thus $\tau(z),\ttau(z)$ is represented by a point in this affine space, and an open subset of this affine space arises in this way.  The possible maps $r(z)$ obtained in this way is thus an irreducible $(2n-1)$-dimensional space, since ($\alpha \tau(z))(\alpha^{-1} \ttau(z))= \tau(z) \ttau(z)$ for a scalar $\alpha$.  Also $\dim \M_{0,n} = n-3$ so we have described an irreducible subvariety of $\tI$ of dimension $3n-4$; let us call this family $\tJ_d$.

Now, $\dim P_n^4 = 4n - 4 - n = 3n-4$, with 4 constraints from momentum conservation and $n$ from the massless condition.  Since $\tI \to P_n^4$ is generically finite of degree $(n-3)!$, we have $\dim \tI = \dim P_n^4 = 3n-4$.  It follows that each of the subvarieties $\tJ_d$ described above is full-dimensional in $\tI$.  Over the locus $(P_n^4)^\circ$ where $\tI \to (P_n^4)^\circ$ has exactly degree $(n-3)!$, every solution belongs to one of the $\tJ_d$.
Thus $\tJ_1,\ldots,\tJ_{n-3}$ are exactly the irreducible components of $\tI$.
\end{proof}

The irreducible components $\tJ_d$ of $\tI$ separate the solutions of the scattering equations into $n-3$ sectors.  Let $E_{n,k}$ denote the Eulerian number, counting the number of permutations in $S_n$ with $k$ ascents.  Thus $\sum_k E_{n,k} = n!$.  For example, $E_{3,1} = 1$, $E_{3,2}  = 4$, and $E_{3,3} = 1$.  In \cite{CHYthree}, Cachazo--He--Yuan explain the following beautiful formula (earlier observed in a slightly different setting in \cite{RSV}).

\medskip

\noindent
{\bf CHY Formula.} 
\begin{equation}
\label{eq:Euler}
\#\{ \text{solutions in the $k= d+1$-th sector} \} =  E_{n-3,k-2}.
\end{equation}
Equivalently, the degree of the map $\tJ_d \to P_n^4$ is equal to the Eulerian number $E_{n-3,d-1}$.

\medskip

I do not know of an explanation of this CHY formula similar to the proof of \cref{thm:n!} that depends on the calculation of Euler characteristic.

\subsection{Twistor strings}\label{sec:twistor}
According to a proposal of Witten \cite{Wit}, and expanded upon by Roiban--Spradlin--Volovich \cite{RSV}, gluon scattering amplitudes in Yang-Mills theory can be obtained from \emph{twistor string theory}.  Indeed, this proposal was one of the key motivations for the CHY scattering equations.  

Substituting $z = \sigma_i$ into \eqref{eq:rz}, we have the equality of $2 \times 2$ matrices
$$
\tau(\sigma_i) \ttau(\sigma_i) = r(\sigma_i) = \prod_{j\neq i}(\sigma_j-\sigma_i) p_i = \prod_{j\neq i}(\sigma_j-\sigma_i)  \lambda_i \tlambda_i 
$$
and it follows that for some $t_i,\ttt_i \neq 0$, we have the equations
\begin{equation}\label{eq:ttt}
t_i \tau(\sigma_i) = t_i \sum_{m=0}^d \tau_m \sigma_i^m = \lambda_i, \qquad \ttt_i \ttau(\sigma_i) = \ttt_i \sum_{m=0}^{\td} \ttau_m \sigma_i^m =\tlambda_i
\end{equation}
where $t_i \ttt_i = \prod_{j\neq i}(\sigma_j-\sigma_i)^{-1}$, and $\tau_m$ (resp. $\ttau_m$) are the coefficients of the $2$-component polynomial $\tau(z)$ (resp. $\ttau(z)$).  When $\lambda_i,\tlambda_i$ have been fixed, we view \eqref{eq:ttt} as constraints on the $k,\tilde k$-planes (with $k = d+1$ and $\tilde k = \td + 1$) $C \in \Gr(k,n)$, $\tC \in \Gr(\tilde k,n)$:
\begin{equation}\label{eq:CtC}
C = \begin{bmatrix}
t_1 & t_2 & \cdots & t_n \\
t_1 \sigma_1 & t_2 \sigma_2 & \cdots & t_n \sigma_n \\
\vdots & \vdots & \ddots &\vdots \\
t_1 \sigma_1^{d} & t_2 \sigma_2^{d} & \cdots & t_n \sigma_n^d 
\end{bmatrix}, \qquad \tC = \begin{bmatrix}
\ttt_1 & \ttt_2 & \cdots & \ttt_n \\
\ttt_1 \sigma_1 & \ttt_2 \sigma_2 & \cdots & \ttt_n \sigma_n \\
\vdots & \vdots & \ddots &\vdots \\
\ttt_1 \sigma_1^{\td} & \ttt_2 \sigma_2^{\td} & \cdots & \ttt_n \sigma_n^{\td}
\end{bmatrix} 
\end{equation}
stating that $\lambda \subset C$ and $\tlambda \subset \tC$.  In other words, the existence of $\tau(z),\ttau(z)$, gives the 2-vectors $\tau_0,\ldots,\tau_d,\ttau_0,\ldots,\ttau_{\td}$, which witness $\lambda \subset C$ and $\tlambda \subset \tC$.
\begin{lemm}\label{lem:ttt}
With $t_i \ttt_i = \prod_{j\neq i}(\sigma_j-\sigma_i)^{-1}$, we automatically have $C \cdot \tC^T = 0$.  
\end{lemm}
Since $k + \tilde k= n$, we have that $C$ and $\tC$ are orthogonal complements of each other.  In other words, $C = \tC^\perp$ as points in $\Gr(k,n)$.  The scattering equations are now the constraint that $\lambda \subset C \subset \tlambda^\perp$.

%

Following \cite{ABCT}, we view \eqref{eq:CtC} using the Veronese map $\theta: \P^1 \to \P^{k-1}$ (with $k = d+1$) induced by the map 
$$
\theta: \C^2 \to \Sym^{k-1} \C^2 \cong \C^{k}, \qquad v \mapsto v \otimes v \otimes \cdots \otimes v.
$$
The representation $\Sym^{k-1} \C^2$ of $\GL(2)$ gives a homomorphism $\theta: \GL(2) \to \GL(k)$, and the Veronese map is equivariant with respect to the actions of $\GL(2)$:
$$
\theta(g \cdot v) = \theta(g) \cdot \theta(v), \qquad g \in \GL(2), v \in \C^2.
$$
The map $\theta \times \theta \cdots \theta: (\C^2)^n \to (\C^{k+1})^n$ descends to a rational \emph{Veronese map} $\theta: \Gr(2,n) \longdashrightarrow \Gr(k,n)$ given by the formula
$$
\begin{bmatrix} a_1 & a_2 & \cdots & a_n \\
b_1 & b_2 & \cdots & b_n
\end{bmatrix} \longmapsto \begin{bmatrix} a^{k-1}_1 & a^{k-1}_2 & \cdots & a^{k-1}_n \\
a_1^{k-2}b_1 & a_2^{k-2}b_2 & \cdots & a_n^{k-2}b_n \\
\vdots & \vdots & \ddots & \vdots \\
b^{k-1}_1 & b^{k-1}_2 & \cdots & b^{k-1}_n
\end{bmatrix}
$$
Note that $\theta$ is only a rational map.  It is not, for example, defined on the $\binom{n}{2}$ points $\sp(e_i,e_j) \in \Gr(2,n)$.  

The matrices $C, \tC$ of \eqref{eq:CtC} are, generically, those of the form $\theta(V)$ for $V \in \Gr(2,n)$.
The scattering correspondence can now be viewed as a subvariety of $\Gr(2,n) \times \tK_n^4$:
\begin{equation}\label{eq:Jd}
\J_d = \J_{k-1} := \{(V,\lambda,\tlambda) \mid \lambda \subset \theta(V) \subset \tlambda^\perp\} \subset \Gr(2,n) \times \tK_n^4.
\end{equation}
The fiber of $\J_d$ over generic $(\lambda,\tlambda)$ can be identified with the fiber of $\tJ_d$ over the corresponding $P \in P_n^4$, where $p_i = \lambda_i \tlambda_i$.

\begin{definition}\label{def:ABCT}
Define the \emph{ABCT variety} $V(k,n) := \overline{\theta(\Gr(2,n))} \subset \Gr(k,n)$ as the closure of the image of $\theta$, a $(2n-4)$-dimensional irreducible subvariety of the Grassmannian.
\end{definition}

\begin{remark}
Let us mention the following natural appearance of more positive geometry.
\begin{conjecture}
The closure $V(k,n)_{\geq 0}:= \overline{\theta(\Gr(2,n)_{>0})}$ is a positive geometry.
\end{conjecture}
Note that the face structure of $V(k,n)_{\geq 0}$ and $\Gr(2,n)_{\geq 0}$ differ, even though their interiors are naturally diffeomorphic.
I expect the face structure of $V(k,n)_{\geq 0}$ to be closely related to that of the momentum amplituhedron \cite{DFLP} (see \cref{sec:SH}) in the same way that the face structure of $(\M_{0,n})_{\geq 0}$ matches that of the associahedron.
\end{remark}


%

\subsection{Extracting spinor-helicity gluon amplitudes} 
%

In \cref{sec:scatmap}, we saw that the scalar amplitude could be obtained by interpreting the scattering equations as a scattering map from $\M_{0,n}$ to kinematic space $K_n$.  To obtain spinor-helicity amplitudes, we similarly take the scattering correspondence $\J_d$ \eqref{eq:Jd} and view it as a map from $\Gr(2,n)$ to a subvariety of $\tK_n^4$.   Since $\dim \Gr(2,n) = 2n-4$, we need a $(2n-4)$-dimensional subvariety of $\tK_n^4$.

Let $[n] = N \sqcup P$, where the particles in $N$ have negative helicity and the particles in $P$ have positive helicity.  Let $k = |N|$ and $n-k = |P|$.
We define an ``ambitwistor" subvariety $\Lambda_{N,P}$ of $\tK_n^4$ as follows.  Let $\pi_N: \C^n \to \C^{N}$ and $\pi_P: \C^n \to \C^P$ denote the orthogonal projections.  Viewing $(\lambda_i, i \in N)$ as a $2 \times k$ matrix defining a two-plane $\lambda_N \subset \C^N$ and similarly $\tlambda_P \subset \C^P$.  Define
$$
\Lambda_{N,P} := \{(\lambda, \tlambda) \in \tK_n^4 \mid (\lambda,\tlambda) \in \Gr(2,\pi_N^{-1}(\lambda_N)) \times \Gr(2, \pi_P^{-1}(\tlambda_P))\} \subset \tK_n^4.
$$
Since $\dim \pi_N^{-1}(\lambda_N) = n-k+2$ and $\dim \pi_P^{-1}(\lambda_P) = k+2$, the product of Grassmannians has dimension $2(n-k) + 2k = 2n$, and the subvariety of the product satisfying $\lambda \perp \tlambda$ has dimension $2n-4$.  This is our desired subvariety, and the correspondence $\J_d$ defines a dominant pushforward map
$$
\Phi(\lambda_N,\tlambda_P): \Gr(2,n) \to \Lambda_{N,P}.
$$
Note that in $\Lambda_{N,P}$, the spinors $\lambda_i, i\in N$ and $\tlambda_i,i \in P$ are considered fixed, while the remaining spinors $\lambda_i, i \in P$ and $\tlambda_i, i \in N$ are considered varying.

By \cite{MaSk,HZ,DFLP,HKZ}, the spinor-helicity amplitude $A_{N,P}$ can be obtained as the coefficient of 
$$
\prod_{i \in N} (d \tlambda_i^1 \wedge d \tlambda_i^2) \prod_{j \in P} (d \lambda_j^1 \wedge d \lambda_j^2)
$$
in 
$$
\Phi(\lambda_N,\tlambda_P)_* \Omega(\Gr(2,n)_{\geq 0}) \wedge d^4(\lambda \tlambda).
$$
To this computation, we pull the $(2n-4)$-form $\Theta = \Phi(\lambda_N,\tlambda_P)_* \Omega(\Gr(2,n)_{\geq 0})$ back to $\hLambda_n^4$.  Wedging with $d^4(\lambda \tlambda)$ gives us a well-defined $2n$-form on $\Mat_{2,n} \times \Mat_{2,n}$.
We explain in the following example.

\begin{example}
Let us consider the case $|N|=k=2$ and $|P|=n-2$, called ``MHV amplitudes".  Then $\dim \pi_N^{-1}(\lambda_N) = n$ so $\pi_N^{-1}(\lambda_N) = \C^n$ while $\dim \pi_P^{-1}(\lambda_P) = 4$.  If $(\lambda,\tlambda) \in \Lambda_{N,P}$ then $\dim \lambda^\perp = n-2$ so generically we must have $\tlambda = \lambda^\perp \cap \pi_P^{-1}(\lambda_P)$.  In other words, the projection 
$$
\Lambda_{N,P} \to \Gr(2,\pi_N^{-1}(\lambda_N)) \cong \Gr(2,n), \qquad (\lambda,\tlambda) \mapsto \lambda
$$
is birational.  The composition of $\Phi(\lambda_N,\tlambda_P)$ with this projection is an isomorphism, and thus the pushfoward form $\Omega =\Phi(\lambda_N,\tlambda_P)_* \Omega(\Gr(2,n)_{\geq 0})$ is simply the canonical form on $\Gr(2,n)$.  As in \cref{ex:forms}, we pull the form back to a $(2n-4)$-form $\Theta$ on $\Mat_{2,n}$.  We have, with $N = \{a,b\}$, up to a constant,
$$
\Theta = \frac{\langle a b \rangle^2}{\langle 12 \rangle \langle 23 \rangle \cdots \langle n1 \rangle} \prod_{j \in P} (d \lambda_j^1 \wedge d \lambda_j^2) + \text{ other terms},
$$
where the other terms involve $d \lambda_i^\alpha$ for $i \in N$, and vanish if we assume that $\lambda_i$ is constant for $i \in N$.  (Note that this form is invariant under scaling each of the spinors $\lambda_i$.)
Now, the momentum conservation constraint $\lambda \tlambda = 0$ consists of four equations, and $d^4(\lambda \tlambda)$ is the wedge of them.  We have, up to a constant, 
$$
d^4(\lambda \tlambda) = \langle a b \rangle^2 d\tlambda_a^1 \wedge d\tlambda_a^2 \wedge d\tlambda_b^1 \wedge d\tlambda_b^2 + \text{ other terms}.
$$
The other terms either involve $d\tlambda_j$ for $j \in P$, or they vanish in the product $\Theta \wedge d^4(\lambda \tlambda)$.  Thus, 
$$
\Theta \wedge d^4(\lambda \tlambda) =  \frac{\langle a b \rangle^4}{\langle 12 \rangle \langle 23 \rangle \cdots \langle n1 \rangle} \prod_{i \in N} (d \tlambda_i^1 \wedge d \tlambda_i^2) \prod_{j \in P} (d \lambda_j^1 \wedge d \lambda_j^2)
$$
and
$$
A_{N,P} = A^{{\rm gluon}}_{N,P} = \frac{\langle a b \rangle^4}{\langle 12 \rangle \langle 23 \rangle \cdots \langle n1 \rangle}.
$$
This is the famed \emph{Parke-Taylor formula} for MHV amplitudes.
\end{example}

We can also use the scattering correspondence diagram
\[ \begin{tikzcd}
&\J_d \arrow{rd}{q} \arrow[swap]{ld}{p} & \\%
\Gr(2,n) && \Lambda_n^4
\end{tikzcd}
\]
to define a scattering form on spinor kinematic space, analogous to \cref{def:scatform}.

\begin{definition}\label{def:Up}
The Yang-Mills scattering form is the $(2n-4)$-form
$$
\Upsilon_n := q_* p^* \Omega(\Gr(2,n)_{\geq 0})
$$
on spinor kinematic space $\Lambda_n^4$.
\end{definition}

We expect that $\Upsilon_n$ is equal to the scattering form defined in \cite{HZ}.  We formulate this as a conjecture.
\begin{conjecture}
The $(2n-4)$-form $\Upsilon_n$ coincides with the $\mathcal{N}=4$ SYM scattering form defined in \cite{HZ}.
\end{conjecture}

\subsection{Exercises and Problems}
\begin{exercise}\label{ex:spinordet} \
\begin{enumerate}[label=(\alph*)]
\item
 Prove \cref{lem:spinordet}.
 \item
 Let $n = 4$.  Prove that 
 $$
 \frac{\langle i j \rangle^4}{\langle 12 \rangle \langle 23 \rangle \langle 34 \rangle \langle 14 \rangle} =  \frac{[ kl ]^4}{[12][23][34][14]}
  $$
  where $\{i,j,k,l\} = \{1,2,3,4\}$.
   \item
 Prove \cref{lem:ttt}.
 \item Figure out the meaning of, and then prove, the following identity of four-point amplitudes:
 $$
 A_4(1^-,2^-,3^+,4^+) +  A_4(2^-,1^-,3^+,4^+) +  A_4(2^-,3^+,1^-,4^+) =0.
 $$
 \item Prove the identity
 $$
 A_5(1,2,5,3,4) = A_5(1,2,4,3,5) + A_5(1,4,2,3,5) + A_5(1,4,3,2,5)
 $$
 for any choice of helicities for $1,2,3,4,5$.

 \end{enumerate}
\end{exercise}

\begin{exercise}\label{ex:orthogrel} Let $X \subset \Gr(k,n) \times \Gr(\ell,n)$ denote the subvariety of pairs $(V,W) \in  \Gr(k,n) \times \Gr(\ell,n)$ such that $V \perp W$.  Figure out or look up the generators for the ideal that cuts out $X$.  In the case $k=2=\ell$, check that one recovers the equations of \eqref{eq:spinorMC}.
\end{exercise}

\begin{exercise}
Calculate the helicity amplitude $A_{N,P}$ when $|N|=|P|=3$ and $n = 6$.
\end{exercise}

\begin{problem}
Use the geometry and topology of $\tI$ to prove the CHY formula \eqref{eq:Euler}.
\end{problem}

\begin{problem}\footnote{This problem is addressed by Agostini, Ramesh, and Shen \cite{ARS}.  In particular, (b), (c) have been solved for $k=3$.} Recall the subvariety $V(k,n) \subset \Gr(k,n)$ from \cref{def:ABCT}.
\begin{enumerate}[label=(\alph*)]
\item
The variety $V(3,6)$ is a codimension one irreducible subvariety of $\Gr(3,6)$.  It is cut out by the condition that $6$ points on $\P^2$ lie on a conic.  Find, in Pl\"ucker coordinates, the relation that cuts out this variety.
\item
Calculate the ideal of $V(k,n)$.
\item
What is the degree of the subvariety $V(k,n) \subset G(k,n)$?  What is the cohomology class $[V(k,n)] \in H^*(\Gr(k,n))$?  (We expect this to be closely related to, or identical to, the sum of the cohomology classes of the $(2n-4)$-dimensional positroid varieties which appear in the formula for the amplitude in \cite{Grassbook}.)
\end{enumerate}
\end{problem}

\subsubsection{Spinor-helicity polytopes} \label{sec:SH} 
As mentioned in \cref{rem:Ass}, the pullback $\iota^*_{\uc} \Psi_n$ of the scattering form to $H(\uc)$ is the canonical form of the associahedron polytope, a positive geometry.  Analogously, in momentum-twistor coordinates, the (super-)Yang-Mills amplitude can be obtained from the canonical form of the \emph{amplituhedron} \cite{AT}, another positive geometry.  In recent years, interest has turned to the momentum amplituhedron $\M_{n,k} \subset \Lambda_n^4$ \cite{DFLP, HZ, HKZ}.

\def\tL{{\tilde L}}
\def\tY{{\tilde Y}}
Let $L$ be a full-rank $(n-k+2) \times n$ matrix and $\tL$ be a full-rank $(k+2) \times n$ matrix, viewed as linear maps.  They induce (rational) maps on Grassmannians
$$
\tL: \Gr(k,n) \longdashrightarrow \Gr(k,k+2), \qquad L: \Gr(k,n) \longdashrightarrow \Gr(n-k,n-k+2)
$$
where the first map is just $V \mapsto \tL(V)$ for $V \in \Gr(k,n)$ and the second map is $V \mapsto L(V^\perp)$.  These combine to a rational map
$$
(\tL,L): \Gr(k,n) \longdashrightarrow \Gr(k,k+2) \times \Gr(n-k,n-k+2).
$$
Let $\sp(L), \sp(\tL) \subset \C^n$ denote the $(n-k+2)$-dimensional (resp. $(k+2)$-dimensional) subspaces spanned by the rows of $L$ (resp. $\tL$).  We have a map  
\begin{align}\label{eq:tLL}
\begin{split}
 \Gr(k,k+2) \times \Gr(n-k,n-k+2) &\to  \Gr(2,k+2) \times \Gr(2,n-k+2) \\ &\cong \Gr(2,\sp(\tL)) \times \Gr(2,\sp(L)) \\ &\hookrightarrow \Gr(2,n) \times \Gr(2,n),
 \end{split}
\end{align}
where the first map takes orthogonal complements.  In the following, we abuse notation by letting $\Lambda_n^4$ refer to a subvariety of $\Gr(2,n) \times \Gr(2,n)$ (instead of the cones over such).  Let $\Phi_{(\tL,L)}$ denote the composition of $(\tL,L)$ with \eqref{eq:tLL}.

\begin{lemm}[\cite{DFLP}]
The map $\Phi_{(\tL,L)}$ lands in $\Lambda_n^4$.
\end{lemm} 
\begin{proof}
Equip $\C^n$ and $\C^{n-k+2}$ with their natural inner products, and view $L: \C^n \to \C^{n-k+2}$ and $L^T: \C^{n-k+2} \to \C^n$ as linear maps so that $\langle a, L(b) \rangle_{\C^{n-k+2}} = \langle L^T(a), b \rangle_{\C^{n}}$.  Let $Y = L(V^\perp)$.  Let $v \in V^\perp$ and $u \in Y^\perp \subset \C^{n-k+2}$.  Then we have 
$ \langle v, L^T u \rangle_{\C^n} = \langle L(v), u \rangle_{\C^{n-k+2}} = 0 $
since $L(v) \in Y$.  Thus $V^\perp \subset (L^T(Y^\perp))^\perp$ or $L^T(Y^\perp) \subset V$.  Here, $L^T(Y^\perp) \subset \C^n$ is a 2-dimensional subspace.  Similarly, with $\tY = \tL(V)$, we have $\tL^T(\tY^\perp) \subset V^\perp$.  It follows that $L^T(Y^\perp)$ and $\tL^T(\tY^\perp)$ are orthogonal.
\end{proof}

The image of $(\tL,L)$ in $\Gr(k,k+2) \times \Gr(n-k,n-k+2)$ has codimension 4, and thus dimension $2k+2(n-k)-4 = 2n-4$.  The subvariety $\Phi_{(\tL,L)}(\Gr(k,n)) \subset \Lambda_n^4$ also has dimension $2n-4$, and it is exactly these subvarieties that the scattering form $\Upsilon_n$ of \cref{def:Up} should be pulled back to, to obtain Yang-Mills amplitudes.

We now specialize to the case where all Grassmannians and matrices are defined over $\R$.  The \emph{momentum amplituhedron} $\M_{n,k}$ \cite{DFLP} is defined to be the image of $\Gr(k,n)_{\geq 0}$ under the map $(\tL,L)$, or under the map $\Phi_{(\tL,L)}$, when $(\tL,L)$ satisfies the appropriate positivity condition.  Positivity conditions are discussed in \cite[(2.16) and below (2.32)]{DFLP}, which we do not review here.  An important motivation of our presentation is the work of He-Kuo-Zhang \cite{HKZ}, who conjecture that the scattering map of \cref{sec:twistor} sends $\Gr(2,n)_{>0}$ diffeomorphically onto the interior of the momentum amplituhedron.  

\begin{remark}
The amplituhedron, defined as a subspace of the Grassmannian $\Gr(k,k+m)$ is a direct generalization of a polytope, with the Grassmannian replacing projective space.  In contrast, the momentum amplituhedron lives in the rather more complicated $\Lambda_n^4$ which can be identified with a partial flag variety.  There are nevertheless advantages to the latter setup.  The analogue of the scattering equations, and the twistor string variety $V(k,n) \subset \Gr(k,n)$ is not known in the setting of amplituhedron, nor is the analogue of the HKZ conjecture \cite{HKZ}.  In momentum space, the pairs $(\lambda,\tlambda) \in \Lambda_n^4$ exhibits the natural parity duality $\lambda \leftrightarrow \tlambda$, whereas for the amplituhedron parity duality is more involved \cite{GLparity}.
\end{remark} 

We suggest the following definition, an analogue of the Grassmann polytopes of \cite{LamCDM}.

\begin{definition}
A \emph{spinor-helicity polytope} is the image $$P :=\Phi_{(\tL,L)}(\Gr(k,n)_{\geq 0}) \cong (\tL,L)(\Gr(k,n)_{\geq 0}),$$ when the map is well-defined on the totally-nonnegative Grassmannian.
\end{definition}

\begin{problem}
Find conditions on $(\tL,L)$ to guarantee that $\Phi_{(\tL,L)}(\Gr(k,n)_{\geq 0})$ is well-defined.
\end{problem}
In the case of Grassmann polytopes, this problem is solved by Karp \cite{Kar}.  A sufficient condition is discussed earlier in \cite{LamCDM} for ``tame Grassmann polytopes".

Let us also suggest the following Schubert calculus problem which is the spinor-helicity version of \cite{LamAmpl}.  Recall the positroid varieties $\Pi_f \subset \Gr(k,n)$ \cite{KLS}.

\begin{problem}
What is the cohomology class of $[\Phi_{(\tL,L)}(\Pi_f)]$?  When do we have $\dim \Phi_{(\tL,L)}(\Pi_f) = \dim \Pi_f$, and in that case, what is the degree of the map $\Pi_f \to \Phi_{(\tL,L)}(\Pi_f)$?
\end{problem}

\section{String amplitudes}\label{sec:string}
\def\Fcal{{\mathcal{F}}}
\def\Re{{\rm Re}}
In this section, we introduce string theory amplitudes, following the presentation of \cite{AHLstringy}.  For other perspectives, see for example \cite{SS,MizThesis}.

\subsection{Open string amplitudes at tree level} Recall that we defined the Koba-Nielsen potential $\phi_X$ on $\M_{0,n}$ in \cref{def:KB}.
\begin{definition}
The \emph{tree level $n$-point open superstring amplitude} is the function on kinematic space $K_n$ defined as the following integral:
\begin{equation}\label{eq:string}
I_n(\s) := (\alpha')^{n-3} \int_{(\M_{0,n})_{>0}} \phi_{\alpha' X} \Omega((\M_{0,n})_{>0}) =  (\alpha')^{n-3} \int_{(\M_{0,n})_{>0}} \prod_{a<b} (ab)^{\alpha' s_{ab}} \Omega_{0,n}.
\end{equation}
This function should be thought of as the analytic continuation of the integral in the domain where the integral converges.
\end{definition}
In string theory, particles are replaced by either circles $S^1$ (in closed string theory) or intervals $[0,1]$ (in open string theory).  The inverse string tension $\alpha'$ is a parameter that measures the size of the string.  In the usual formula for vibrating strings, we have that string tension and string length are inversely related.  As $\alpha' \to 0$, the strings shrink in size and one obtains point particle scattering in the limit.  Whereas the field theory scattering amplitude is a sum over $L$-loop Feynman diagrams with $n$ distinguished leaves, the string theory amplitude is an integral over $\M_{L,n}$, the moduli space of genus $L$ curves with $n$ marked points.  

We consider only genus zero string amplitude.  In closed string theory, the integral is over all complex points of $\M_{0,n}$.  In open string theory, the integral is over the set of real points of $\M_{0,n}$.  The choice of a component $(\M_{0,n})_{>0}$ in \eqref{eq:string} is precisely analogous to the sum over \emph{planar} trees in \cref{thm:phi3}.  

Let us consider \eqref{eq:string} for $n = 4$.  In this case, setting $s = s_{12}$ and $t = s_{23}$, we have
\begin{equation}\label{eq:I4}
I_4(s,t) = \alpha' \int_{0}^1 \frac{du}{u(1-u)} u^{\alpha's} (1-u)^{\alpha' t} = \alpha' B(\alpha's, \alpha' t),
\end{equation}
where $B(s,t)$ is the Euler Beta function 
\begin{equation}\label{eq:Beta}
B( s,t) =  \int_{0}^1 \frac{du}{u(1-u)} u^{s} (1-u)^{t}  = \frac{\Gamma(s) \Gamma(t)}{\Gamma(s+t)}.
\end{equation}
The integral converges when $\Re(s), \Re(t) > 0$ and $B(s,t)$ extends to a meromorphic function in the two complex variables $s,t$.  Since the gamma function has no zeroes, the poles of $B(s,t)$ are the poles of $\Gamma(s) \Gamma(t)$ which occur when $s$ or $t$ is a nonpositive integer.  Thus $B(s,t)$ has a pole at $s = 0$, at $t = 0$, and a zero at $s+t = 0$.  As a function of $\alpha'$, the integral $I_4(s,t)$ thus has a a simple pole at $\alpha' = 0$; the prefactor $\alpha'$ in \eqref{eq:I4} ensures a well-defined limit as $\alpha' \to 0$.  Indeed, field theory amplitudes are recovered from string theory amplitudes in the limit as inverse string tension goes to 0: 
\begin{equation}
\label{eq:I4A4}
\lim_{\alpha' \to 0} I_4(s,t) = \frac{s+t}{st} = \frac{1}{s} + \frac{1}{t} = A_4^{\phi^3}(s,t),
\end{equation}
where we have used the approximation $\alpha' \Gamma(\alpha' s) \sim \frac{1}{s}$ when $\alpha' \to 0$.
This limit is also called a ``low-energy" limit of string theory.

The Beta function was studied by Euler and Legendre in the 18th and 19th centuries.  In the 20th century, Veneziano \cite{veneziano1968construction} noticed the interpretation of the Beta function as an amplitude, and this was one of the key ingredients that started string theory. 

The connected components of $\M_{0,n}(\R)$ are in bijection with dihedral orderings of $1,2,\ldots,n$.  For a dihedral ordering $\beta$, we denote by $\M_{0,n}(\beta)$ the corresponding connected component, and by $\Omega(\beta)$ its canonical form.  We have the following larger family of partial string amplitudes 
$$
I_n(\beta|\gamma) := (\alpha')^{n-3} \int_{\M_{0,n}(\gamma)} \phi_{\alpha' X} \Omega(\beta),
$$
where $\beta,\gamma$ are two dihedral orderings.  These are string theory versions of partial $\phi^3$-amplitudes; see \cref{rem:partial} and \cite{MizThesis} for further discussion.  When $\beta$ is fixed, say $\beta = +$ is the positive ordering, and $\gamma$ varies we have a family of integral functions $I_n(+|\gamma)$.  These integral functions are not linearly independent -- they span a vector space of dimension $(n-3)!$.

\subsection{Toric stringy integrals} \label{sec:toricstringy} The string integral \eqref{eq:string} belongs to a class of integral functions called \emph{stringy integrals} in \cite{ALS} and \emph{Euler integrals} in \cite{MHMT}, which have the following form:
$$
\int \frac{y_1^{\tau_1}\cdots y_n^{\tau_n}}{p_1^{c_1} \cdots p_r^{c_r}} \frac{dy_1}{y_1} \wedge \cdots \wedge \frac{dy_n}{y_n},
$$
where $p_1,p_2,\ldots,p_r$ are polynomials in $y_1,\ldots,y_n$, and $\tau_1,\ldots,\tau_n,c_1,\ldots,c_r$ are complex parameters.  Besides string amplitudes, Euler integrals also include hypergeometric integrals and various generalizations, and also Feynman integrals.  We refer the reader to the survey \cite{MHMT} for a discussion of the many beautiful properties of these integrals.

Here we focus on an algebro-geometric perspective.  Let $p(y) \in \C[y]$ be a polynomial in variables $y_1,\ldots,y_d$ with non-vanishing constant term, and let $p_1(y),p_2(y), \ldots, p_r(y)$ be the irreducible factors of $p$.  To fix conventions, we assume that the constant terms in $p_1(y),\ldots,p_r(y)$ are all equal to $1$.  Let $T_y \cong (\C^\times)^d = \Spec(\C[y_1^{\pm 1},\ldots,y_d^{\pm_1}])$ denote the torus with coordinates $y_1,\ldots,y_d$.  Let $\oU \subset T_y$ be the open subset where $p(y) \neq 0$.  The variety $\oU$ is a very affine variety.  

\begin{lemm}
The very affine variety $\oU$ has character lattice $\lat$ equal to the lattice of Laurent monomials in the elements $y_1,y_2,\ldots,y_d, p_1(y),\ldots,p_r(y)$. 
\end{lemm}
\begin{proof}
Suppose that $r(y) = f(y)/g(y) \in \C(y_1,\ldots,y_d)$ is a rational function that is a unit in $\C[\oU]$, and $f(y), g(y)$ have no common factors.  Any irreducible factor $h(y)$ of $f(y)$ or $p(y)$ which is not a monomial in $y_1,\ldots,y_d$ defines a non-empty hypersurface $h(y) = 0$ in the torus $T_y$.  It follows that $h(y)$ must equal one of the $p_1(y),\ldots,p_r(y)$.  Thus $r(y)$ is a Laurent monomial in the functions $y_1,y_2,\ldots,y_d, p_1(y),\ldots,p_r(y)$.  Furthermore, irreducibility implies that there are no monomial relations between these Laurent monomials (as regular functions on $\oU$).  The result follows.
\end{proof}

Thus $\oU$ has an intrinsic scattering potential (see \cref{sec:verypot})
$$
\phi_{\tau, c} := \prod_{i=1}^d y_i^{\tau_i} \prod_{j=1}^{r} p_j(y)^{-c_i}.
$$
We have chosen the signs of the exponents for reasons that will be clear later.  

The variety $\oU$ can naturally be identified with an open subset of a toric variety.  Let $P \subset \Z^d$ be the Newton polytope of $p(y)$; this is the convex hull of the exponent vectors of the monomials in $p(y)$.  We assume that $P$ is full-dimensional and let $A = P \cap \Z^d = \{\a_1,\a_2,\ldots,\a_p\}$ be the set of lattice points in $P$.  Consider the toric variety $X_A$ that is the closure of the image of the map
\begin{equation}\label{eq:XA}
T_y \ni (y_1,\ldots,y_d) \mapsto (y^{\a_1}: \cdots : y^{\a_p}) \in \P^{p-1}.
\end{equation}
This map is an inclusion as long as the $\Z$-affine span of $\{\a_1,\a_2,\ldots,\a_p\}$ is equal to $\Z^d$.  We will assume this by replacing $p(y)$ with a power of it, not changing the variety $\oU$.  

The polynomial $p(y)$ is a linear combination of the monomials $y^{\a_i}$.  Thus the hypersurface $\{p(y) = 0\} \subset T_y$ is the intersection of $T_y$ with a hyperplane $H \subset \P^{r-1}$ cut out by the corresponding linear equation.  Then $\oU$ can be identified with $T_y \setminus H$.

If $p(y)$ has positive coefficients, then one of the connected components of $\oU$ is simply the positive part $(T_y)_{>0} = \R_{>0}^d$ of the torus $T_y$.  We denote this positive component by $\oU_{>0}$.  Henceforth, we make the
\begin{assumption}
The polynomial $p(y)$ has positive coefficients.
\end{assumption}

\begin{definition}
The string integral of $\oU$ is the integral function
$$
I(\tau,c):=  (\alpha')^d\int_{\oU_{>0}} \phi_{\alpha'\tau,\alpha'c} \, \Omega
$$
where $\Omega =  \frac{dy_1}{y_1} \wedge \cdots \wedge \frac{dy_n}{y_n}$ is the canonical form of $\oU_{>0}$, and $(\tau,c) \in \lat_\C$.
\end{definition}

We investigate the convergence of the string integral, assuming that $(\tau,c)\in \lat_\R$.

\begin{example}\label{ex:Beta}
We consider the Beta function \eqref{eq:Beta}.  Setting $ y= u/(1-u)$, we have
$$
B( s,t) =  \int_{0}^1 \frac{du}{u(1-u)} u^{s} (1-u)^{t}  = \int_0^\infty y^s (1+y)^{-t-s} \frac{dy}{y}.
$$
So $B(s,t)$ is a string integral where we have a single polynomial $p(y) = 1+y$.  We may break up the integral as $\int_0^\infty = \int_0^1 + \int_1^\infty$.  When $y \to 0$, the integrand behaves like $y^{s-1} dy$ and $\int_0^1 y^{s-1} dy$ converges when $s > 0$.  When $y \to \infty$, both the integrand behaves like $y^{s+(-t-s)} = y^{-t}$, and $\int_1^\infty y^{-t-1} dy$ converges when $t > 0$.   Thus the integral converges (absolutely) if $s>0$ and $t >0$.
\end{example}

Let $\lat^\vee_{y}$ be the cocharacter lattice of $T_y$.  The vector space $\lat^\vee_{y,\R}$ can be viewed as the Lie algebra of the real Lie group $T_y(\R)$.  The exponential map $\exp: \lat^\vee_{y,\R} \to \oU_{>0} = (T_y)_{>0}$ is an isomorphism, and we may rewrite 
\begin{equation}\label{eq:logI}
I(\tau,c) =   (\alpha')^{d} \int_{\lat^\vee_{y,\R}} \prod_{i=1}^d \exp(\alpha' \tau_i Y_i)  \prod_{j=1}^r p_j(\exp(Y))^{-\alpha' c_j}  \, dY_1 \cdots dY_d.
\end{equation}
where $Y_i = \log y_i$ are coordinates on $\lat^\vee_{y,\R}$.  Define the (positive) tropicalization of rational functions in $y_1,\ldots,y_d$ by 
$$
y_i \mapsto Y_i, \qquad + \mapsto \min, \qquad \times \mapsto +, \qquad \div \mapsto -. 
$$
This assignment makes sense for the potential $\phi_{\alpha'\tau,\alpha'c}$ even when $\alpha',\tau,c$ are not integers, and we obtain
$$
\Trop(\phi) =  \Trop(\phi_{\alpha'\tau,\alpha'c}):= \alpha' (\sum_{i=1}^d \tau_i  Y_i - \sum_{j=1}^d c_j \Trop(p_j)(Y)).
$$
We view this as a piecewise linear function on $\lat^\vee_{y,\R}$.  We call a piecewise linear function \emph{positive} if it is positive on $\lat^\vee_{y,\R} \setminus \{0\}$.  The analysis of \cref{ex:Beta} can be generalized to give the following result.  It is often formulated in terms of Minkowski sums of Newton polytopes \cite{AHLstringy,NP,MHMT}.

\begin{theorem}[{\cite[Claim 3]{AHLstringy}}]\label{thm:stringy}
Suppose $(\tau,c) \in \lat_\R$.  The integral \eqref{eq:logI} converges if and only if the integral
\begin{equation}\label{eq:tropint}
\int_{\lat^\vee_{y,\R}} \exp(-\Trop(\phi_{\alpha'\tau,\alpha'c})(Y)) dY_1 \cdots dY_d
\end{equation}
converges.  This is the case if and only if the piecewise linear function $\Trop(\phi)(Y)$ is positive on $\lat^\vee_{y,\R}$.
\end{theorem}
\begin{proof}[Idea of Proof.]
Bound the integrand $\prod_{i=1}^d \exp(\tau_i Y_i)  \prod_{j=1}^r p_j(\exp(Y))^{-c_j} $ above and below by multiples of $ \exp(-\Trop(\phi_{\tau,c})(Y)) $.
\end{proof}

\begin{example}\label{ex:Beta2}
We continue \cref{ex:Beta}. We have
$$
\Trop(y^s (1+y)^{-t-s}) = sY -(s+t) \min(0,Y) = \begin{cases} s Y & \mbox{ if $Y >0$,} \\
-tY & \mbox{if $Y < 0$.}
\end{cases}
$$
This function is positive if and only if $s>0$ and $t >0$.
\end{example}

The domains of linearity of the piecewise linear function $\Trop(\phi)$ give a complete fan $\Fcal$ in $\lat^\vee_{y,\R}$, where $\Trop(\phi)$ restricts to a linear function on each maximal cone of $\Fcal$.  The integral \eqref{eq:tropint} becomes a sum of integrals over the maximal cones $C$ of $\Fcal$,
$$
\int_{C} \exp(-(a_1 Y_1+ \cdots + a_d Y_d)) dY_1 \cdots dY_d,
$$
where $a_i$ are linear functions in $\alpha',\tau,c$.  These integrals are Laplace transforms of cones, and are rational functions in the $a_i$.
\begin{proposition}\label{prop:cone}
Let $C$ be a full-dimensional cone and $\ell$ a linear function that is positive on $C \setminus \{0\}$.  Then
$$
\int_{C} \exp(-\ell(Y)) dY_1 \cdots dY_d = \Vol(C \cap \{\ell(Y) \leq 1\})
$$
where $\Vol$ denotes the normalized volume where the standard simplex has volume $1$.
\end{proposition}

\begin{example}\label{ex:coneLT}
We have
$$
\int_{\R_{>0}^d} \exp(-(a_1Y_1 + \cdots + a_d Y_d)) dY_1 \cdots dY_d 
= \prod_{i=1}^d \int_{\R_{>0}} \exp(-a_iY_i)  dY_i 
= \prod_{i=1}^d \frac{1}{a_i}
$$
which is the normalized volume of a simplex with vertices $0, \frac{1}{a_1} e_1, \ldots, \frac{1}{a_d} e_d$, where $e_i$ are coordinate vectors dual to the functions $Y_i$.
\end{example}

\begin{corollary}[{\cite[Claim 3]{AHLstringy}}]
Suppose that $(\tau,c)$ are such that $\Trop(\phi)$ is positive.  Then we have $\lim_{\alpha' \to 0} I(\tau,c) = \Vol(\Trop(\phi_{\tau,c})) \leq 1)$.
\end{corollary}
\begin{proof}
Summing over maximal cones of $\Fcal$, we deduce from \cref{prop:cone} that
$$
\int_{\lat^\vee_{y,\R}} \exp(-\Trop(\phi_{\tau,c})(Y)) dY_1 \cdots dY_d =  \Vol(\Trop(\phi)) \leq 1.
$$
Now, 
$$
\Vol(\Trop(\phi_{\alpha'\tau,\alpha'c}) \leq 1)  = \Vol(\alpha'\Trop(\phi) \leq 1) = \frac{1}{(\alpha')^d} \Vol(\Trop(\phi)) \leq 1).
$$
The factor $\frac{1}{(\alpha')^d} $ cancels out the prefactor of $(\alpha')^d$ in the definition of $I(\tau,c)$.  Finally, by the proof of \cref{thm:stringy}, the integral $I(\tau,c)$ is bounded above and below by $(\alpha')^d$ times the integral \eqref{eq:tropint}, up to a factor of $\gamma^{\alpha'}$ for positive constants $\gamma$.  Since $\lim_{\alpha'\to 0} \gamma^{\alpha'} = 1$, the result follows.
\end{proof}

\begin{example}
We continue \cref{ex:Beta2}.  When $s, t> 0$, the inequality $\Trop(y^s (1+y)^{-t-s}) \leq 1$ cuts out the region $[-1/t,1/s]$ which has volume $1/s+1/t$.  This agrees with \eqref{eq:I4A4}.
\end{example}

%

We return to the string amplitude.
\begin{theorem}\label{thm:limitstring}
The field theory limit of the open string amplitude is given by the planar $\phi^3$-amplitude:
$$
\lim_{\alpha' \to 0} I_n(\s) = A_n^{\phi^3}(\s).
$$
\end{theorem}
\begin{proof}
Let $y_1,\ldots,y_d$, where $ d= n-3$ be the positive parameters of \cref{sec:poscoord}.  By \cref{prop:inverty}, the lattice $\lat$ of Laurent monomials in $u_{ij}$ is isomorphic to the lattice of Laurent monomials in $\{y_1,\ldots,y_d\} \cup \{p_{ij}(y)\}$.  Thus $\M_{0,n} \cong \oU$ where $\oU \subset T_y$ is defined by the non-vanishing of $p(y) = \prod p_{ij}(y)$. 

Let $\lat^\vee$ be the intrinsic cocharacter lattice of $\M_{0,n}$.  The functions $U_{ij}$ and the functions $Y_i$ are linear functions on $\lat^\vee_\R$.  Define a linear map $\pi: \lat^\vee_\R \to \lat^\vee_{y,\R}$ by projection, that is, by the equality $Y_i = Y_i \circ \pi$.  The key statement is that $\pi$ sends maximal cones of $\Trop_{\geq 0} U$ to maximal cones of $\Fcal$ (\cref{ex:posparam}).  The summation of \cref{thm:phi3} can be written as
$$
A_n^{\phi^3} = \sum_{C \subset \Trop_{\geq 0} U} \prod_{(ij) \in \D} \frac{1}{X_{ij}}
$$
summed over maximal cones of $ \Trop_{\geq 0} U$ spanned by $X_{ij}, (ij) \in \D$.  This in turn matches the summation 
$$
\int_{\lat^\vee_{y,\R}} \exp(-\Trop(\phi_{\tau,c})(Y)) dY_1 \cdots dY_d = \Vol(C \cap (\Trop(\phi) \leq 1))
$$
over maximal cones of $\Fcal$.
\end{proof}

\subsection{Scattering map revisited}\label{sec:scatmap2}
We now interpret the scattering equations in terms of the integrals $I(\tau,c)$.  The ``method of steepest descent" or "saddle-point method" estimates the asymptotics of integrals
$$
\int f(x) \exp(\alpha' g(x)) dx
$$
when $\alpha' \to \infty$, by considering the behavior of the integral near the critical points of $g(x)$.  In our setting, $g = \log \phi_{\tau, c}$, and the critical point equations are those for $\log \phi_{\tau, c}$, which are exactly the scattering equations in the case of the string integral.  This is also the setting of the positive models of \cite[Section 6]{ST}.

\begin{definition}\label{def:stringSE}
The scattering equations for the string integral $I(\tau,c)$ are the critical point equations
$$
\dlog \phi_{\tau,c} = 0
$$
on $\oU$.
\end{definition}   
In the string integral, positive coordinates $y_1,\ldots,y_d$ on $\oU$ have been distinguished.  This choice gives a particularly elegant way to formulate the scattering map of \cref{sec:scatmap} and an explanation of the subspaces $H(\uc) \subset K_n$.  Indeed, we may write
$$
\prod_{(ij)} (ij)^{s_{ij}} = \prod_{(ij)} u_{ij}^{X_{ij}} = \prod_{i=1}^{n-3} y^{\tau_i}_i \prod_{ij} p^{-c_{ij}}_{ij}(y).
$$
Under the positive parametrization of \cref{sec:poscoord}, when $i,j$ satisfy $1 \leq i < j-1 < j \leq n-1$ we have 
$$
(ij) = p_{ij}(y) \times \mbox{ monomial in } y_1,\ldots,y_{n-3}.
$$
It follows that for such $i,j$ we have $c_{ij} = - s_{ij}$.  Thus the subspace $H(\uc)$ is obtained by fixing the exponents $c_{ij}$ and allowing $\tau_i$ to vary.  More abstractly, the coordinates $(y_1,\ldots,y_n)$ provide a sublattice $\lat' \subset \lat$.  The orthogonal sublattice $(\lat')^\perp \subset \lat^\vee$ is a space of linear functions that are fixed on each subspace $H(\uc)$.  The constants $\uc$ are the values that elements of $(\lat')^\perp$ take.

We now generalize the scattering map $\Phi(\uc)$ of \eqref{eq:scatPhi} to the general setting of a stringy integral.  The scattering equations (\cref{def:stringSE}) can be written as a collection of $d$ equations 
$$
\frac{1}{\phi} \frac{\partial \phi}{\partial y_i} = 0,\qquad i = 1,2,\ldots,d
$$
which we can write as 
$$
\frac{\tau_i}{y_i} = \sum_{j=1}^r \frac{c_j}{p_j(y)} \frac{\partial p_j(y)}{\partial y_i} ,\qquad i = 1,2,\ldots,d.
$$
or 
$$
\tau_i= y_i \sum_{j=1}^r \frac{c_j}{p_j(y)} \frac{\partial p_j(y)}{\partial y_i} ,\qquad i = 1,2,\ldots,d.
$$
Recall that $(\tau,c)$ are coordinates on $\lat_\R$.  For fixed $c$, we view this as a (rational) map 
$$
\Phi(c): U(\R) \to H(c) \cong \Lambda_{y,\R}
$$
where $H(c) \subset \lat_\R$ is an affine subspace where the $c$-coordinates are held fixed and the $\tau$-coordinates vary.

\begin{theorem}[\cite{AHLstringy}] \label{thm:scatmapgen}
Suppose that $I(\tau,c)$ converges, or equivalently by \cref{thm:stringy}, that $\Trop(\phi_{\tau,c})$ is positive.  Then we have 
$$
\Phi(c)_* \Omega_y = \lim_{\alpha' \to 0} I(\tau,c)  d\tau_1 \wedge \cdots \wedge d\tau_d
$$ 
as forms on $H(c)$.  Furthermore, $\Phi(c)$ sends $\R_{>0}^d$ to the interior of the polytope $R^\vee$ polar dual to
$$
R = \{\Trop(\phi) \leq 1\}
$$
and $\Phi(c)_* \Omega_y$ is the canonical form of $R$.
\end{theorem}

\begin{remark}
The scattering map $\Phi(c): U(\R) \to H(c) \cong \Lambda_{y,\R}$ can also be identified with the algebraic moment map of the toric variety $X_A$.  See \cite[Section 7.3]{ABL} and \cite{AHLstringy}.
\end{remark}

\subsection{$u$-coordinates}\label{sec:ucoord}
Our setup distinguishes a subsemigroup inside the character lattice $\lat$.  Call a character $\chi \in \lat$ \emph{bounded} if it takes bounded values on $\oU_{>0}$.  Write $\chi = \frac{a(y)}{b(y)}$ in reduced form, where $a(y)$ and $b(y)$ are polynomials in $y_1,\ldots,y_d$.

\begin{lemm}\label{lem:bounded}
For a character $\chi = \frac{a(y)}{b(y)} \in \lat$, the following are equivalent:
\begin{enumerate}
\item  $\chi$ is bounded,
\item  the Newton polytope of $a(y)$ is contained inside the Newton polytope of $b(y)$,
\item the function $\Trop(\chi)$ is positive.
\end{enumerate}
\end{lemm}

For example, all the five rational functions in \eqref{eq:u5y} satisfy this criterion.  Denote by $\Gamma \subset \lat$ the subsemigroup of bounded characters.

\begin{lemm}
With our assumption that the Newton polytope $P$ of $p(y)$ is full-dimensional, we have that 
\begin{enumerate}
\item
$\Gamma$ generates the rational function field $\C(y) = \C(\oU)$, and
\item
$\Gamma \cup \Gamma^{-1}$ generate the coordinate ring $\C[\oU]$. 
\end{enumerate}
\end{lemm}
\begin{proof}
We may, and will, assume that $P$ is sufficiently large by replacing $p(y)$ by a large power of it.  By \cref{lem:bounded}, the rational functions
$$
\frac{y^\a}{p(y)}, \qquad \a \in P \cap \Z^d
$$
belong to $\Gamma$.  Taking lattice points $\a'$ and $\a' + e_i$ shows that $y_i \in \C(\Gamma)$.  It follows that $\C(\Gamma) = \C(y) = \C(\oU)$.  

The condition that $\Trop(\chi)$ is positive from \cref{lem:bounded} is an open condition in $\lat_\R$, and defines an open polyhedral cone $B \subset \lat_\R$ so that $B \cap \lat = \Gamma$.  Since $\Gamma \neq \emptyset$ we have that $B$ is a non-empty full-dimensional cone.  It follows that for $B$ contains adjacent lattice points in all coordinate directions.  In particular $y^{\pm 1}_1,\ldots,y^{\pm 1}_d,p^{\pm 1}_1(y),\ldots,p^{\pm 1}_r(y)$ all belong to $\Gamma \Gamma^{-1}$.  Thus $\C[\Gamma \cup \Gamma^{-1}] = \C[\oU]$.
\end{proof}

\begin{definition}
The \emph{$u$-coordinates} on $\oU$ are the generators of $\Gamma$.  The affine closure of $\oU$ is $\tU = \Spec(\C[\Gamma])$.
\end{definition}

\begin{proposition}\label{prop:tU}
The partial compactification of $\tU$ is the affine open $X_A \setminus H$ in the projective toric variety $X_A$ of \eqref{eq:XA}.  In particular, $\tU$ has a natural stratification indexed by the faces of the Newton polytope $P$.
\end{proposition}

In the case that $U = \M_{0,n}$, we have $\tU = \tM_{0,n}$.  In this case the polytope $P$ is an associahedron.

\subsection{CEGM amplitudes and Grassmannian string integrals}\label{sec:Xkn}
An important generalization of $\M_{0,n}$ is the configuration space $X(k,n)$ of $n$ points $\sigma_1,\ldots,\sigma_n \in \P^{k-1}$ in \emph{general linear position}, that is, no $k$ of the points belong to the same hyperplane in $\P^{k-1}$.  Let $\Gr(k,n)^\circ \subset \Gr(k,n)$ be the subset of the Grassmannian where all Pl\"ucker coordinates are non-vanishing.  The space $X(k,n)$ is isomorphic to the torus quotient $\Gr(k,n)^\circ/T'$, and is a very affine variety.  

The space $X(k,n)$ has a positive part $X(k,n)_{>0}$, called \emph{positive configuration space} \cite{ALS,ALSnon}, the image of the component $\Gr(k,n)_{>0} \subset \Gr(k,n)(\R)$ where all Pl\"ucker coordinates are positive.  Positive configuration space can be positively parametrized \cite{PosTP,SW,ALS} with coordinates $y_1,y_2,\ldots,y_d$, where $d = \dim X(k,n) = k(n-k) - (n-1)$.  For example, the following coordinates $y_1,y_2,y_3,y_4$ positively parametrize the four-dimensional space $X(3,6)$:
$$
\begin{bmatrix}
 0 & 0 & -1 & -1 & -1 & -1 \\
 0 & 1 & 0 & -1 & -1-y_1 & -1-y_1-y_1 y_3 \\
 1 & 0 & 0 & 1 & 1+y_1+y_1 y_2 & 1+y_1+y_1y_2+y_1y_2 y_3 +y_1y_3 +y_1y_2 y_3 y_4
\end{bmatrix}.
$$
The Pl\"ucker coordinates $\Delta_I = p_I(y)$ are nonnegative polynomials in $y_1,\ldots,y_d$.  We let $p_1(y),p_2(y), \ldots$ be the distinct irreducible factors of these polynomials.  This puts $X(k,n)$ in the setting of \cref{sec:toricstringy}.  The corresponding partial compactification $\tU$ from \cref{prop:tU} is studied in \cite{ALS}.

Cachazo-Early-Guevara-Mizera \cite{CEGM} generalized the scattering equations to $X(k,n)$ and defined corresponding scalar amplitudes, called \emph{CEGM amplitudes} or \emph{generalized biadjoint scalar amplitudes}; see \cite{CE,CEZ} for some recent work in this direction.  These amplitudes are the field theory limits of the \emph{Grassmannian string integrals} \cite{AHLstringy} which have the form
$$
I(\s) =(\alpha')^d \int_{X(k,n)_{>0}} \prod_{I \in \binom{[n]}{k}} p^{\alpha' s_I}_I  \;\Omega(X(k,n)_{>0}) = \int_{\R_{>0}^d} \prod_{I \in \binom{[n]}{k}} p^{\alpha' s_I}_I(y)  \; \frac{dy_1}{y_1} \wedge \cdots \wedge  \frac{dy_d}{y_d}
$$
where the $s_I$ satisfy $\sum_{a_2,\ldots,a_k} s_{a_1,a_2,\ldots,a_k} = 0$ for any $a_1$, the analogue of momentum conservation, which also guarantees the integrand is torus-invariant and descends to $X(k,n)$.

\subsection{Exercises and Problems}
\begin{exercise} \
\begin{enumerate}[label=(\alph*)]
\item
Verify that the integral \eqref{eq:Beta} converges when $\Re(s) >0$ and $\Re(t) > 0$.
\item
Work out the details in the proof of \cref{thm:stringy}.  To do so, decompose the integration domain $\Lambda^\vee_{y,\R} \cong \R^d$ into the maximal cones of the domains of linearity of $\trop(\phi)$ and analyze the convergence on each cone separately.
\end{enumerate}
\end{exercise}

\begin{exercise}\label{ex:posparam}
Show that the map $\pi: \lat^\vee_\R \to \lat^\vee_{y,\R}$ in the proof of \cref{thm:limitstring} sends maximal cones of $\Trop_{\geq 0} U$ to maximal common domains of linearity of $\Trop(p_{ij})(Y)$.  Then check that the contributions of these cones to $A_n^{\phi^3}$ and to $\lim_{\alpha' \to 0} I_n$ match.
\end{exercise}

\begin{exercise}
Prove \cref{lem:bounded}.
\end{exercise}

\begin{exercise}
Prove \cref{prop:tU}.  How is the stratification of \cref{prop:tU} related to the vanishing and non-vanishing of the generators of $\Gamma$?
\end{exercise}

\begin{exercise}
Recall the positive parametrization of $\M_{0,5}$ by parameters $y_1,y_2$.  The five variables $u_{13},u_{14},u_{24},u_{25},u_{35}$ are subtraction-free rational functions in $y_1,y_2$.  
\begin{enumerate}[label=(\alph*)]
\item Consider the (positive) tropicalization $\trop(\phi_X)$ as a piecewise-linear function on $\Lambda^\vee_{y,\R}$.  Compute the domains of linearity of this piecewise-linear function.
\item The domains of linearity of $\trop(\phi_X)$ give the structure of a complete fan in $\Lambda^\vee_{y,\R}$.  This complete fan has five one-dimensional cones.  Compute the minimal lattice generators $r_1,r_2,\ldots,r_5 \in \Lambda^\vee_y$ of these cones.
\item Verify that the generators $r_i$ can be labeled as $r_{13},r_{14},r_{24},r_{25},r_{35}$ so that
$$
\trop(u_{ij})(r_{ab}) = \delta_{(ij),(ab)}.
$$
\end{enumerate}

\end{exercise}

\subsubsection{Simplicial $\Gamma$ and $u$-equations}
Consider the setting of \cref{sec:ucoord}.  The nicest situation is when $\Gamma$ is generated by $u_1,\ldots,u_d$ and these generators also form a basis for the character lattice $\lat$.  This is the case for $\M_{0,n}$, and the generators are exactly the dihedral coordinates $u_{ij}$.  We say in this case that $\Gamma$ is \emph{simplicial}.

\begin{problem}
Find examples of simplicial $\Gamma$.
\end{problem}
See \cite[Section 9]{AHLstringy} for related discussion.  Even when $\Gamma$ is simplicial, and we have natural coordinates $u_1,\ldots,u_d$ on $\lat$, it is not clear when we obtain a binary geometry in the sense of \cref{sec:binary}.  Note that the $u$-equation $R_i$ for $u_i$ exists exactly because the function $1-u_i$ is a monomial in $u_1,\ldots,u_d$.

\begin{problem}
Let $u$ be a generator of $\Gamma$.  When do we have $1-u \in \lat$, and when do we have $1-u \in \Gamma$?  More generally, for which pairs $(c, u^X) \in \C^\times \times \lat$ do we have $c - u^X \in \lat$?
\end{problem}

\subsubsection{Pell's space}\label{sec:Pells}
This problem comes from \cite[Section 4]{HLRZ}.  Consider the polynomials $p_i(y) = (1+ y_i)$ for $i= 1,2\ldots,d$ and $p_{d+i}(y) = (1+y_i + y_iy_{i+1})$ for $i=1,2,\ldots,d-1$.  The following is my interpretation of statements in \cite{HLRZ}, which we may consider conjectures.  It extends \cref{prob:Pell2} which corresponds to the case $d = 3$.

\begin{problem}\footnote{This problem has been addressed by Bossinger, Telek, and Tillmann-Morris \cite{BTT}.}\
\begin{enumerate}[label=(\alph*)]
\item
Show that the Newton polytope $P$ of $p(y) = \prod_i p_i(y)$ is simple, has $3d-1$ facets, and $P_d$ vertices, where $P_d$ is Pell's number recursively defined by $P_1 = 1, P_2 = 2, P_d = 2P_{d-1} + P_{d-2}$.
\item
We have $\dim \lat = (2d-1) + d$.  Show that the minimal generators of $\Gamma$ are a basis of $\lat$.   Call theses generators $u_1,\ldots,u_{3d-1}$.
\item
Show that $\tU$, ``Pell's space", is a binary geometry for the simplicial complex $\Delta$, defined to be the face complex of $P$.  Find and prove the $u$-equations in terms of $u_1,\ldots,u_{3d-1}$.
\item 
Study the geometry of $\tU$.  Is there a moduli interpretation of this space?
\end{enumerate}
\end{problem}

\section{Very affine amplitudes?}\label{sec:veryaffine}
\def\RR{{\mathcal R}}
\def\trop{{\Trop}}
\def\CC{{\mathcal C}}
\def\TT{{\mathcal T}}
\def\oX{{\mathring{X}}}
\def\ray{{\rm rays}}

\begin{definition}
A \emph{very affine variety} $\oU$ is an irreducible closed subvariety of a torus.  
\end{definition}

Examples of very affine varieties include: $\M_{0,n}$, complements of essential hyperplane arrangements, configuration spaces of points in $\P^k$, and so on.  In this section, we speculate on amplitudes for very affine varieties.  In \cite{Lamfuture}, we study these amplitudes in more detail in the situation of hyperplane arrangement complements.

\subsection{Intrinsic torus}
\label{sec:intrinsic}
The group $\lat = \lat(U) := \C[\oU]^\times/\C^\times$ of units of the coordinate ring of $\oU$, modulo scalars, is a free abelian group of finite rank by a theorem of Rosenlicht \cite{Ros} and Samuel \cite{Sam}.  Let $T = \Hom(\lat, \C^\times)$ be the torus with character group $\lat \cong \Z^n$, and we denote by $\lat^\vee$ the dual group of cocharacters of $T$.  There is an inclusion 
$$
\iota: \oU \hookrightarrow T, \qquad p \mapsto (f \mapsto f(p)),
$$
and any closed embedding of $\oU$ into a torus $T'$ is a composition of $\iota$ with a homomorphism $T \to T'$.

\subsection{Scattering potential}\label{sec:verypot}
Let $\oU \subset T \cong (\C^\times)^n$ be a $d$-dimensional very affine variety, and let $u_1,u_2,\ldots,u_n$ be the coordinates of $T$.  Then $\oU$ has a (multi-valued) potential function 
$$
\phi_X = \prod_{i=1}^n u_i^{X_i}
$$
for $X = (X_1,\ldots,X_n) \in \lat_\C \cong \C^n$.  When $T$ is the intrinsic torus, we call $\phi_X$ the intrinsic potential on $\oU$.
%

%
%
%
%

\begin{definition}
The \emph{scattering equations} of $\oU$ are the critical point equations of the log potential on $\oU$:
$$
\dlog \phi_X = \sum_{i=1}^n X_i \,\dlog\, u_i =  \sum_{i=1}^n  \frac{X_i}{u_i} du_i = 0.
$$
While $\phi_X$ is in general multi-valued, the equations $\dlog \phi_X = 0$ give a well-defined subvariety of $\oU$.
\end{definition}

The following result of Huh generalizes earlier work of Orlik and Terao \cite{OT} in the case of hyperplane arrangement complements.

\begin{theorem}[\cite{Huh}]\label{thm:Huh}
Suppose that $\oU$ is smooth.  For generic $X$, the critical point set consists of $|\chi(\oU)|$ reduced points, where $\chi(\oU)$ denotes the Euler characteristic.
\end{theorem}

\subsection{Very affine positive geometries}
Let $U$ be a very affine variety.  We consider positive geometries inside compactifications $X$ of $U$.   For example, one can take $X$ to be a tropical compactification in the sense of Tevelev \cite{Tev}.  


\begin{definition}
A very affine positive geometry is a triple $(U,X,X_{\geq 0})$ where $(X,X_{\geq 0})$ is a positive geometry, $X$ is a compactification of $U$, and $X_{\geq 0}$ is the analytic closure in $X$ of a connected component of $U(\R)$.
\end{definition}

For a very affine positive geometry, there is a natural definition of string amplitude.  Let $\phi_Z$ denote the intrinsic potential of $\oU$.

\begin{definition}\label{def:vastring}
The string amplitude of a very affine positive geometry $(U,X,X_{\geq 0})$ is the function on $\Lambda_\C$ defined by analytic continuation of the integral function
$$
I(Z) := (\alpha')^d \int_{X_{>0}} \phi_{\alpha' Z} \; \Omega(X_{\geq 0}),
$$
where $Z \in \Lambda_\C$.
\end{definition}

\begin{definition}\label{def:vascalar}
The scalar amplitude of a very affine positive geometry $(U,X,X_{\geq 0})$ is the rational function on $\Lambda_\C$ defined by 
$$
A(Z) := \int \delta^d(\phi_{Z}; \Omega(X_{\geq 0})) \; \Omega(X_{\geq 0}),
$$
where $Z \in \Lambda_\C$, and the delta function $\delta^d(\phi_{Z}, \Omega(X_{\geq 0}))$ is defined in \cref{sec:delta}.
\end{definition}

We expect that the scalar amplitude should be the field-theory limit of the string amplitude, that is, $A(Z) = \lim_{\alpha' \to 0} I(Z)$.  More generally, if we are given two very affine positive geometries $(U,X, X_{\geq 0})$ and $(U,X, Y_{\geq 0})$ with the same very affine variety $U$, we can define the partial string amplitude
$$
I_{X_{>0},Y_{>0}}(Z) := (\alpha')^d \int_{Y_{>0}} \phi_{\alpha' Z} \; \Omega(X_{\geq 0}), \qquad I_{Y_{>0},X_{>0}}(Z) := (\alpha')^d \int_{X_{>0}} \phi_{\alpha' Z} \; \Omega(Y_{\geq 0})
$$
and the partial scalar amplitude (cf. \cref{def:ampl})
$$
A_{X_{>0},Y_{>0}}(Z) = A_{Y_{>0},X_{>0}}(Z) :=A(\Omega(X_{\geq 0})|\Omega(Y_{\geq 0}))=  \int \delta^d(\phi_{Z}; \Omega(Y_{\geq 0})) \; \Omega(X_{\geq 0}). 
$$

We may also define a scattering form using the scattering correspondence:
\[ \begin{tikzcd}
&\I \arrow{rd}{q} \arrow[swap]{ld}{p} & \\%
U && \lat_\C
\end{tikzcd}
\]

\begin{definition}\label{def:genscatform}
For a very affine positive geometry $(U,X,X_{\geq 0})$, the scattering form $\Upsilon(X_{\geq 0})$ on $\lat_\C$ is given by
$$
\Upsilon:= q_* p^* \Omega(X_{\geq 0}).
$$
\end{definition}

\subsection{Positive parametrizations of very affine varieties}
In some cases, we can positively parametrize the very affine $U$, putting us in the situation of \cref{sec:string}, and \cref{def:vastring} and \cref{def:vascalar} will hold.

Let $\oU \subset T$ be a $d$-dimensional very affine variety in the intrinsic torus $T$, and let $u_i$ be the coordinates of the $T$.  Let $\lat= \{u_1^{a_1} \cdots u_n^{a_n} \mid (a_1,\ldots,a_n) \in \Z^n \}$ be the character lattice.  Suppose that we have $y_1,\ldots,y_d \in \lat$ such that $\C(y_1,\ldots,y_d) = \C(\oU)$.  In particular, each $u_i \in \C(y)$ is a rational function $u_i(y) \in \C(y)$ in $y_1,\ldots,y_d$.    Let $p_1(y),p_2(y), \ldots,p_r(y) \in \C[y]$ be the list of all distinct irreducible factors (assumed to have constant term 1) appearing in the numerator or denominator of $u_i(y)$, excluding $y_1,\ldots,y_d$ themselves.   

\begin{definition}
In this situation, suppose that $p_i(y) \in \lat$ for $i=1,2,\ldots,r$.  Then we say that $y_1,\ldots,y_d$ give a parametrization of $\oU$.  We call this a \emph{positive parametrization} if the coefficients of $p_i(y)$ are all positive.
\end{definition}

\begin{lemm}
Suppose $y_1,\ldots,y_d \in \lat$ give a parametrization of $\oU$.  Then there is an invertible monomial transformation between $\{y_1,y_2,\ldots,y_d, p_1(y),\ldots,p_r(y)\}$ and $\{u_1,\ldots,u_n\}$, and $\{y_1,y_2,\ldots,y_d, p_1(y),\ldots,p_r(y)\}$ also form a basis of $\lat$.
\end{lemm}
\begin{proof}
There are no monomial relations between $y_1,\ldots,y_d, p_1(y),\ldots, p_r(y)$, since we have chosen all the polynomials to be irreducible with constant term 1.  It follows that they form a basis of $\lat$.
\end{proof}

Thus in the case of a (positive) parametrization, we are in the situation \cref{sec:toricstringy}.


\subsection{Tropicalizations of very affine varieties}
Now suppose that $\oU$ is a very affine variety, but we do not have (or do not know) any natural top-form $\Omega$ on $\oU$ (arising from a positive geometry), nor do we have a positive parametrization of $\oU$.  In this case, we can still try to define a scalar amplitude by viewing the formula \cref{thm:phi3} as a Laplace transform of $\trop_{\geq 0} U$, cf. \cref{prop:cone}.  We begin by defining the tropicalization of a very affine variety, and refer the reader to \cite{MS,Tev} for more on tropicalization. 

Let $\CC = \bigcup_{n \geq 1} \C((t^{1/n}))$ denote the field of Puiseux series, equipped with a valuation 
$$
\val: \CC \to \R \cup\{\infty\}, \qquad \val(f(t)) = \begin{cases} \text{minimal degree appearing in }f(t) & \mbox{if $f \neq 0$} \\
\infty & \mbox{if $f = 0$.}
\end{cases}
$$

\begin{definition}
The tropicalization $\trop(\oU)$ is the subset of $\lat^\vee_\R:= \lat^\vee \otimes_\Z \R \cong \R^n$ obtained by taking the closure of the valuations of all $\CC$ points of $\oU$:
$$
\trop(\oU):= \overline{\{ \val(p) \mid p \in \oU(\CC)\}} \subset \lat^\vee_\R.
$$
\end{definition}

The tropicalization $\trop(\oU)$ can be given the structure of a polyhedral fan, pure of dimension $d = \dim(\oU)$.  There are a number of different definitions of the fan structure on $\trop(\oU)$.  We assume that one such fan structure has been fixed.

\begin{example}\label{ex:tropM04}
Suppose that $\oU = \M_{0,4}$.  Then the intrinsic torus has coordinates $(u_{13},u_{24})$, and $\oU$ is given by the single equation $u_{13}+u_{24}=1$.  The tropicalization $\trop(\oU)$ consists of the three 1-dimensional cones (rays): $\R_{\geq 0} \cdot (1,0),  \R_{\geq 0} \cdot (0,1), \R_{\geq 0} \cdot (-1,-1)
$.  The positive tropicalization $\trop_{>0}(\oU)$ consists of the two 1-dimensional cones: $\R_{\geq 0} \cdot (1,0),  \R_{\geq 0} \cdot (0,1)$.
\end{example}
Let $C \subset \lat^\vee_\R$ be a maximal cone of $\trop(\oU)$.  The affine span of $C$ is a $d$-dimensional subspace $W \subset \lat^\vee_\R$.  We define a $d$-form on $W$ that will serve as an integration measure.  Since $W$ is rationally defined, the intersection $\lat^\vee \cap W$ is a sublattice of $\lat^\vee$ that spans $W$.  We define an integration measure $d^dW$ by declaring that a unit cube in $\lat^\vee \cap W$ has unit volume with respect to $d^dW$.  The standard simplex in $\lat^\vee \cap W$ has normalized volume equal to one.


We propose the following general notion of (partial) amplitude for a very affine variety.  
\begin{definition}
Let $\TT$ denote a subset of maximal cones of $\trop(\oU)$.  Then the $\TT$-partial amplitude is the rational function on $\lat_\R$
\begin{align}\label{eq:ampldef}
\begin{split}
A_{\TT}(X) &:= \sum_{C \in \TT} \int_{C} e^{-X} d^dW = \Vol((\bigcup_\TT C) \cap \{ X \leq 1\}) = \sum_{C \in \TT} \Vol( C \cap \{ X \leq 1\}),
\end{split}
\end{align}
where $X \in \lat_\R$ is viewed as a linear function on $\lat^\vee_\R$.  The second equality is by \cref{prop:cone}.
\end{definition}
We view $A_{\TT}(X)$ as the Laplace transform of $\TT \subset \lat^\vee_\R$.  In the case the cones $C$ is simplicial with appropriately chosen generators $r_1,r_2,\ldots,r_d$ that generate $W$, we have
$$
 \Vol(C \cap \{X \leq 1\}) =\prod_{i=1}^d \frac{1}{(r_i,X)}
$$
is a product of inverses of linear functions $(r_i,X)$ on $\lat^\vee_\R$; see \cref{ex:coneLT}.

\begin{remark}
In the case of the configuration space $X(k,n)$ of \cref{sec:Xkn}, Cachazo-Early-Zhang \cite{CEZ} have studied subfans $\TT$ called \emph{chirotopal tropical Grassmannians}.  See also the recent work of Antolini and Early \cite{AE}.
\end{remark}

\subsection{Positive components}
To proceed to the definition of \emph{planar} amplitudes, we need to specify a set of cones $\TT$.   This roughly amounts to picking a connected component of $\oU(\R)$.  We henceforth assume that $\oU$ is defined over $\R$.  

In the case $\oU = \M_{0,n}$, the following are roughly equivalent:
\begin{enumerate}
\item 
the choice of a connected component of $\M_{0,n}(\R)$;
\item
the choice of a positive parametrization of $\M_{0,n}$;
\item
the choice a canonical form $\Omega_{0,n}$ on $\M_{0,n}$;
\item
the choice of ``planar trees" among all $n$-leaf trees;
\item
the choice of an affine compactification $\M_{0,n} \subset \tM_{0,n}$;
\item
the choice of a subset of ``planar" boundary divisors of $\bM_{0,n}$;
\item
the choice of a subsemigroup $\Gamma \subset \lat$ of bounded characters;
\item
the choice of the Parke-Taylor factor ${\rm PT}(\alpha)$;
\item
the choice of a subfan $\Trop_{\geq 0} \M_{0,n} \subset \Trop \M_{0,n}$.
\end{enumerate}

Let $\kappa: \lat \to \Z/2\Z$ be a group homomorphism.  If a basis $u_1,\ldots,u_n$ of $\lat$ is chosen, we write $\kappa$ as a sequence in $\{+,-\}^n$ with a $+$ indicating $\kappa(u_i) = 0 \in \Z/2\Z$ and a $-$ indicating $\kappa(u_i) = 1 \in \Z/2\Z$.  The torus $T$, as an abstract algebraic variety, has an alternative group structure denoted $T_\kappa$, given by sending the identity $(1,1,\ldots,1)$ to element $((-1)^{\kappa_1},\ldots,(-1)^{\kappa_n})$, where $\kappa_i = \kappa(u_i)$.  Equivalently, $T_\kappa$ has characters given by $(-1)^{\kappa(u)} u$, for $u \in \lat$.  

The set of real points $T(\R)$ of the torus has $2^n$ components.  We let $T_{>0}$ be the connected component containing the identity, and let $T_\kappa(\R) :=(T_\kappa)_{>0}$.  Thus $T(\R) = \bigsqcup_\kappa T_\kappa(\R)$.   

Define a \emph{component} of $\oU(\R)$ to be a non-empty intersection
$$
\oU_\kappa:= \oU \cap T_\kappa(\R)
$$ 
of the very affine variety with $T_\kappa(\R)$.  If non-empty, this is a union of connected components of $\oU(\R)$, and in many interesting cases $\oU_\kappa$ is a single connected component of $\oU(\R)$.  This is the case for $\M_{0,n}$ (\cref{ex:signpatternM0n}).  While $\oU_\kappa$ can consist of multiple connected components, we will use $\kappa \in\pi_0(T(\R))$ as an input to our definition of amplitudes.

\subsection{Positive part of tropicalization}\label{sec:postrop}
Given a choice of component $T_\kappa(\R)$, we may define $\trop_{\kappa}U$ in a number of ways.  The first one is the positive tropicalization of \cite{SW}.  Let  $\RR_{>0} \subset \CC$ be the semifield consisting of nonzero Puiseux series such that the coefficient of the lowest degree term is real and positive.   Let $\oU(\CC)$ denote the $\CC$-points of $\oU$.  Let $\oU(\RR_{>0}) \subset \oU(\CC)$ denote the subset of points $p \in \oU(\CC)$ where $u(p) \in \RR_{>0}$ for every character $u \in \lat$ of $T$.

\begin{definition}
The positive tropicalization $\trop_{> 0}(\oU)$ is the subset of $\lat^\vee_\R$ obtained by taking the closure of the valuations of all $\RR_{>0}$ points of $\oU$:
$$
\trop_{>0}(\oU):= \overline{\{ \val(p) \mid p \in \oU(\RR_{>0})\}} \subset  \lat^\vee_\R.
$$
\end{definition}

The positive tropicalization $\trop_{>0}(\oU)$ is a subfan of $\trop(\oU)$, also pure of dimension $d$.  Let $\oU(\RR_{\kappa}) \subset \oU(\CC)$ denote the subset of points $p \in \oU(\CC)$ where $ (-1)^{\kappa(u)}u(p) \in \RR_{>0}$ for every character $u \in \lat$ of $T$.

\begin{definition}
The $\kappa$-tropicalization $\trop_\kappa(\oU)$ is the set
$$
\trop_{\kappa}(\oU):= \overline{\{ \val(p) \mid p \in \oU(\RR_{\kappa}) \text{ for all } u \in L\}} \subset  \lat^\vee_\R.
$$
\end{definition}
We view $\trop_\kappa(\oU)$ as the tropicalization of the component $\oU_\kappa$.

\begin{example}
Continuing \cref{ex:tropM04}, 
\begin{itemize}
\item $\trop_{++}(\oU) = \trop_{>0}(\oU)$  consists of the two 1-dimensional cones $\R_{\geq 0} \cdot (1,0)$ and $\R_{\geq 0} \cdot (0,1)$.
\item $\trop_{+-}(\oU)$ consists of the two 1-dimensional cones $\R_{\geq 0} \cdot (0,1)$ and $\R_{\geq 0} \cdot (-1,-1)$.
\item $\trop_{-+}(\oU)$ consists of the two 1-dimensional cones $\R_{\geq 0} \cdot (1,0)$ and $\R_{\geq 0} \cdot (-1,-1)$.
\item $\trop_{--}(\oU)$ is empty.
\end{itemize}
This agrees with $|\pi_0(\M_{0,4}(\RR))|= (4-1)!/2 = 3$.
\end{example}

\subsection{Bounded invertible functions}
We suggest an alternative to positive tropicalization, using bounded invertible functions, as in \cref{ex:intrinsic}.
\begin{definition}
Let $\oU_\kappa$ be a non-empty component.  Define the submonoid $B_\kappa \subset \lat$ by
$$
B_\kappa = \{u \in \lat \mid u|_{\oU_\kappa} \text{ takes bounded values}\}.
$$
Define the $\kappa$-positive part of $\trop(U)$ by
$$
\trop^\kappa(U) := \{q \in \trop(U) \mid u(q) \geq 0 \text{ for all } u \in B_\kappa\}.
$$
Here, $u \in \lat$ is viewed as a linear function on $\lat^\vee_\R$.
\end{definition}
We expect that under good conditions, $u$ belongs to $B_\kappa$ if and only if $u(\Trop_\kappa(\oU)) \subset \R_{\geq 0}$.

\begin{example}
For $\oU= \M_{0,4}$, 
\begin{itemize}
\item $B_{++}$ is generated by $u_{13},u_{24}$.  
\item $B_{+-}$ is generated by $u_{13}^{-1}, u_{24}u_{13}^{-1}$.
\item $B_{-+}$ is generated by $u_{24}^{-1}, u_{13}u_{24}^{-1}$.
\end{itemize}
\end{example}

Finally, we can take $\TT = \trop_\kappa(\oU)$ or $\TT = \trop^\kappa(\oU)$ in \eqref{eq:ampldef} as the definition of the planar scalar amplitude of $\oU$.  
\newcommand{\etalchar}[1]{$^{#1}$}


\begin{thebibliography}{AHCD{\etalchar{+}}24}

\bibitem[ABFK{\etalchar{+}}23]{ABFKST}
Daniele Agostini, Taylor Brysiewicz, Claudia Fevola, Lukas K\"uhne, Bernd Sturmfels, and Simon Telen.
\newblock {Likelihood degenerations}. With an appendix by Thomas Lam.
\newblock {\em Adv. Math.}  414 (2023), Paper No. 108863, 39 pp.

\bibitem[ARS24]{ARS}
Daniele Agostini, Lakshmi Ramesh, and Dawei Shen.
\newblock Points on Rational Normal Curves and the ABCT Variety.
\newblock {\em \arxiv{2412.12514}}.

\bibitem[AE24]{AE}
Dario Antolini and Nick Early.
\newblock The Chirotropical Grassmannian.
\newblock {\em \arxiv{2411.07293}}.

\bibitem[AHBC{\etalchar{+}}16]{Grassbook}
Nima Arkani-Hamed, Jacob~L. Bourjaily, Freddy Cachazo, Alexander~B. Goncharov,
  Alexander Postnikov, and Jaroslav Trnka.
\newblock {\em {Grassmannian Geometry of Scattering Amplitudes}}.
\newblock Cambridge University Press, 2016.

\bibitem[AHBCT11]{ABCT}
Nima Arkani-Hamed, Jacob Bourjaily, Freddy Cachazo, and Jaroslav Trnka.
\newblock {Unification of Residues and Grassmannian Dualities}.
\newblock {\em JHEP}, 01:049, 2011.

\bibitem[AHBHY18]{ABHY}
Nima Arkani-Hamed, Yuntao Bai, Song He, and Gongwang Yan.
\newblock {Scattering Forms and the Positive Geometry of Kinematics, Color and
  the Worldsheet}.
\newblock {\em JHEP}, 05:096, 2018.

\bibitem[AHBL17]{ABL}
Nima Arkani-Hamed, Yuntao Bai, and Thomas Lam.
\newblock {Positive Geometries and Canonical Forms}.
\newblock {\em JHEP}, 11:039, 2017.

\bibitem[AHCD{\etalchar{+}}24]{ACDFS}
Nima Arkani-Hamed, Qu~Cao, Jin Dong, Carolina Figueiredo, and Song He.
\newblock Scalar-scaffolded gluons and the combinatorial origins of yang-mills
  theory.
\newblock {\em \arxiv{2401.00041}}.

\bibitem[AHFS{\etalchar{+}}23]{AFSPT}
Nima Arkani-Hamed, Hadleigh Frost, Giulio Salvatori, Pierre-Guy Plamondon, and
  Hugh Thomas.
\newblock All loop scattering as a counting problem.
\newblock {\em \arxiv{2309.15913}}.

\bibitem[AHHL21a]{AHLcluster}
Nima Arkani-Hamed, Song He, and Thomas Lam.
\newblock Cluster configuration spaces of finite type.
\newblock {\em SIGMA Symmetry Integrability Geom. Methods Appl.}, 17:Paper No.
  092, 41, 2021.

\bibitem[AHHL21b]{AHLstringy}
Nima Arkani-Hamed, Song He, and Thomas Lam.
\newblock Stringy canonical forms.
\newblock {\em J. High Energy Phys.}, (2):Paper No. 069, 59, 2021.

\bibitem[AHHLT23]{AHLT}
Nima Arkani-Hamed, Song He, Thomas Lam, and Hugh Thomas.
\newblock Binary geometries, generalized particles and strings, and cluster
  algebras.
\newblock {\em Phys. Rev. D}, 107(6):Paper No. 066015, 8, 2023.

\bibitem[AHLS21a]{ALSnon}
Nima Arkani-Hamed, Thomas Lam, and Marcus Spradlin.
\newblock Non-perturbative geometries for planar {$\mathcal{N} = 4$} {SYM}
  amplitudes.
\newblock {\em J. High Energy Phys.}, (3):Paper No. 065, 14, 2021.

\bibitem[AHLS21b]{ALS}
Nima Arkani-Hamed, Thomas Lam, and Marcus Spradlin.
\newblock Positive configuration space.
\newblock {\em Comm. Math. Phys.}, 384(2):909--954, 2021.

\bibitem[AHT14]{AT}
Nima Arkani-Hamed and Jaroslav Trnka.
\newblock {The Amplituhedron}.
\newblock {\em JHEP}, 10:030, 2014.

\bibitem[BBBB{\etalchar{+}}15]{BBBDF}
Christian Baadsgaard, N.~E.~J. Bjerrum-Bohr, Jacob~L. Bourjaily, Poul~H.
  Damgaard, and Bo~Feng.
\newblock Integration rules for loop scattering equations.
\newblock {\em J. High Energy Phys.}, (11):080, front matter+19, 2015.

\bibitem[BTTM24]{BTT}
Lara Bossinger, M\'at\'e L. Telek, and Hannah Tillmann-Morris
\newblock Binary geometries from pellytopes.
\newblock {\em \arxiv{2410.08002}}.

\bibitem[BEPV24]{BEPV}
Sarah Brauner, Christopher Eur, Elizabeth Pratt, and Raluca Vlad.
\newblock Wondertopes.
\newblock {\em \arxiv{2403.04610}}.

\bibitem[Bri]{Bri}
Ruth Britto.
\newblock Constructing scattering amplitudes.
\newblock {\em \url{https://www.maths.tcd.ie/~britto/lecture-notes.pdf}}.

\bibitem[Bro09]{Bro}
Francis C.~S. Brown.
\newblock {Multiple zeta values and periods of moduli spaces
  $\mathfrak{M}_{0,n}$}.
\newblock {\em Annales Sci. Ecole Norm. Sup.}, 42:371, 2009.

\bibitem[CE20]{CE}
Freddy Cachazo and Nick Early.
\newblock Planar kinematics: Cyclic fixed points, mirror superpotential,
  k-dimensional catalan numbers, and root polytopes.
\newblock {\em \arxiv{2010.09708}}, 2020.

\bibitem[CEGM19]{CEGM}
Freddy Cachazo, Nick Early, Alfredo Guevara, and Sebastian Mizera.
\newblock {Scattering Equations: From Projective Spaces to Tropical
  Grassmannians}.
\newblock {\em JHEP}, 06:039, 2019.

\bibitem[CEZ22]{CEZ}
Freddy Cachazo, Nick Early, and Yong Zhang.
\newblock Color-dressed generalized biadjoint scalar amplitudes: Local
  planarity.
\newblock {\em \arxiv{2212.11243}}, 2022.

\bibitem[CHY13]{CHYthree}
Freddy Cachazo, Song He, and Ellis~Ye Yuan.
\newblock {Scattering in Three Dimensions from Rational Maps}.
\newblock {\em JHEP}, 10:141, 2013.

\bibitem[CHY14a]{CHYKLT}
Freddy Cachazo, Song He, and Ellis~Ye Yuan.
\newblock {Scattering equations and Kawai-Lewellen-Tye orthogonality}.
\newblock {\em Phys. Rev.}, D90(6):065001, 2014.

\bibitem[CHY14b]{CHYarbitrary}
Freddy Cachazo, Song He, and Ellis~Ye Yuan.
\newblock {Scattering of Massless Particles in Arbitrary Dimensions}.
\newblock {\em Phys. Rev. Lett.}, 113(17):171601, 2014.

\bibitem[CHY14c]{CHYmassless}
Freddy Cachazo, Song He, and Ellis~Ye Yuan.
\newblock {Scattering of Massless Particles: Scalars, Gluons and Gravitons}.
\newblock {\em JHEP}, 07:033, 2014.

\bibitem[CM24]{CM}
Shelby Cox and Igor Makhlin.
\newblock Tropicalizing binary geometries.
\newblock {\em \arxiv{2410.13652}}.

\bibitem[DFLP19]{DFLP}
David Damgaard, Livia Ferro, Tomasz Lukowski, and Matteo Parisi.
\newblock {The Momentum Amplituhedron}.
\newblock {\em JHEP}, 08:042, 2019.

\bibitem[DG14]{DG}
Louise Dolan and Peter Goddard.
\newblock Proof of the formula of {C}achazo, {H}e and {Y}uan for {Y}ang-{M}ills
  tree amplitudes in arbitrary dimension.
\newblock {\em J. High Energy Phys.}, (5):010, front matter+23, 2014.

\bibitem[Dix96]{Dix}
Lance~J. Dixon.
\newblock Calculating scattering amplitudes efficiently.
\newblock {\em \arxiv{hep-ph/9601359}}.

\bibitem[EPS24]{EPS}
Yassine El Maazouz, Ana\"elle Pfister, and Bernd Sturmfels.
\newblock Spinor-Helicity Varieties.
\newblock {\em \arxiv{2406.17331}}.

\bibitem[EH15]{EH}
Henriette Elvang and Yu-tin Huang.
\newblock {\em Scattering amplitudes in gauge theory and gravity}.
\newblock Cambridge University Press, Cambridge, 2015.

\bibitem[FM21]{FM}
Hadleigh Frost and Lionel Mason.
\newblock Lie polynomials and a twistorial correspondence for amplitudes. 
\newblock {\em Lett Math Phys.}, 111, (2021), article 147.

\bibitem[GKL22a]{GKL3}
Pavel Galashin, Steven~N. Karp, and Thomas Lam.
\newblock Regularity theorem for totally nonnegative flag varieties.
\newblock {\em J. Amer. Math. Soc.}, 35(2):513--579, 2022.

\bibitem[GKL22b]{GKL1}
Pavel Galashin, Steven~N. Karp, and Thomas Lam.
\newblock The totally nonnegative {G}rassmannian is a ball.
\newblock {\em Adv. Math.}, 397:Paper No. 108123, 23, 2022.

\bibitem[GL20]{GLparity}
Pavel Galashin and Thomas Lam.
\newblock Parity duality for the amplituhedron.
\newblock {\em Compos. Math.}, 156(11):2207--2262, 2020.

\bibitem[HKZ22]{HKZ}
Song He, Chia-Kai Kuo, and Yao-Qi Zhang.
\newblock The momentum amplituhedron of {SYM} and {ABJM} from twistor-string
  maps.
\newblock {\em J. High Energy Phys.}, (2):Paper No. 148, 32, 2022.

\bibitem[HLRZ20]{HLRZ}
Song He, Zhenjie Li, Prashanth Raman, and Chi Zhang.
\newblock Stringy canonical forms and binary geometries from associahedra,
  cyclohedra and generalized permutohedra.
\newblock {\em J. High Energy Phys.}, (10):054, 35, 2020.

\bibitem[HS14]{HS}
June Huh and Bernd Sturmfels.
\newblock Likelihood geometry.
\newblock In {\em Combinatorial algebraic geometry}, volume 2108 of {\em
  Lecture Notes in Math.}, pages 63--117. Springer, Cham, 2014.

\bibitem[Huh13]{Huh}
June Huh.
\newblock The maximum likelihood degree of a very affine variety.
\newblock {\em Compos. Math.}, 149(8):1245--1266, 2013.

\bibitem[HZ18]{HZ}
Song He and Chi Zhang.
\newblock {Notes on Scattering Amplitudes as Differential Forms}.
\newblock {\em JHEP}, 10:054, 2018.

\bibitem[Kar17]{Kar}
Steven~N. Karp.
\newblock Sign variation, the {G}rassmannian, and total positivity.
\newblock {\em J. Combin. Theory Ser. A}, 145:308--339, 2017.

\bibitem[KKT24]{KKT}
Leonie Kayser, Andreas Kretschmer, and Simon Telen.
\newblock Logarithmic Discriminants of Hyperplane Arrangements.
\newblock {\em \arxiv{2410.11675}}.

\bibitem[Kee92]{Keel}
Sean Keel.
\newblock Intersection theory of moduli space of stable {$n$}-pointed curves of
  genus zero.
\newblock {\em Trans. Amer. Math. Soc.}, 330(2):545--574, 1992.

\bibitem[KLS13]{KLS}
Allen Knutson, Thomas Lam, and David~E. Speyer.
\newblock Positroid varieties: juggling and geometry.
\newblock {\em Compos. Math.}, 149(10):1710--1752, 2013.

\bibitem[KN69]{KN}
Ziro Koba and Holger~Bech Nielsen.
\newblock {Manifestly crossing invariant parametrization of n meson amplitude}.
\newblock {\em Nucl. Phys.}, B12:517--536, 1969.

\bibitem[KR01]{KR}
Boris Khesin and Alexei Rosly.
\newblock Polar homology and holomorphic bundles.
\newblock Topological methods in the physical sciences (London, 2000).
\newblock R. Soc. Lond. Philos. Trans. Ser. A Math. Phys. Eng. Sci. 359 (2001), no.1784, 1413--1427.

\bibitem[Lam16a]{LamAmpl}
Thomas Lam.
\newblock Amplituhedron cells and {S}tanley symmetric functions.
\newblock {\em Comm. Math. Phys.}, 343(3):1025--1037, 2016.

\bibitem[Lam16b]{LamCDM}
Thomas Lam.
\newblock Totally nonnegative Grassmannian and Grassmann polytopes.
\newblock Current developments in mathematics 2014, 51--152.
International Press, Somerville, MA, 2016.

\bibitem[Lam22]{LamPosGeom}
Thomas Lam.
\newblock An invitation to positive geometries. 
\newblock In {\em Open Problems in Algebraic Combinatorics}, volume 110 of {\em Proceedings of Symposia in Pure Mathematics}. 
\newblock AMS, 2024.


\bibitem[Lam24]{Lamfuture}
Thomas Lam.
\newblock Matroids and amplitudes.
\newblock {\em \arxiv{2412.06705}}.


\bibitem[MHMT23]{MHMT}
Saiei-Jaeyeong Matsubara-Heo, Sebastian~Mizera, and Simon Telen.
\newblock Four lectures on Euler integrals.
\newblock {\em {\arxiv{2306.13578}}}, 2023.

\bibitem[Miz18]{Miz}
Sebastian Mizera.
\newblock Scattering amplitudes from intersection theory.
\newblock {\em Phys. Rev. Lett.}, 120(14):141602, 6, 2018.

\bibitem[Miz20]{MizThesis}
Sebastian Mizera.
\newblock {\em Aspects of scattering amplitudes and moduli space localization}.
\newblock Springer Theses. Springer, Cham, [2020] \copyright 2020.
\newblock Doctoral thesis accepted by the University of Waterloo, Canada.

\bibitem[MS14]{MaSk}
Lionel Mason and David Skinner.
\newblock Ambitwistor strings and the scattering equations.
\newblock {\em Journal of High Energy Physics}, 2014(7):48, 2014.

\bibitem[MS15]{MS}
Diane Maclagan and Bernd Sturmfels.
\newblock {\em Introduction to tropical geometry}, volume 161 of {\em Graduate
  Studies in Mathematics}.
\newblock American Mathematical Society, Providence, RI, 2015.

\bibitem[NP13]{NP}
Lisa Nilsson and Mikael Passare.
\newblock Mellin transforms of multivariate rational functions.
\newblock {\em J. Geom. Anal.}, 23(1):24--46, 2013.

\bibitem[OT95]{OT}
Peter Orlik and Hiroaki Terao.
\newblock The number of critical points of a product of powers of linear
  functions.
\newblock {\em Invent. Math.}, 120(1):1--14, 1995.

\bibitem[Pan97]{Pan}
Rahul Pandharipande.
\newblock The canonical class of {$\overline{M}_{0,n}(\bold P^r,d)$} and
  enumerative geometry.
\newblock {\em Internat. Math. Res. Notices}, (4):173--186, 1997.

\bibitem[Pos06]{PosTP}
Alexander Postnikov.
\newblock {Total positivity, Grassmannians, and networks}.
\newblock 2006.

\bibitem[Pro07]{Pro}
Claudio Procesi.
\newblock {\em Lie groups}.
\newblock Universitext. Springer, New York, 2007.
\newblock An approach through invariants and representations.

\bibitem[Ros57]{Ros}
Maxwell Rosenlicht.
\newblock Some rationality questions on algebraic groups.
\newblock {\em Ann. Mat. Pura Appl. (4)}, 43:25--50, 1957.

\bibitem[RSV04]{RSV}
Radu Roiban, Marcus Spradlin, and Anastasia Volovich.
\newblock {On the tree level S matrix of Yang-Mills theory}.
\newblock {\em Phys. Rev.}, D70:026009, 2004.

\bibitem[Sam66]{Sam}
Pierre Samuel.
\newblock \`{A} propos du th\'{e}or\`eme des unit\'{e}s.
\newblock {\em Bull. Sci. Math. (2)}, 90:89--96, 1966.

\bibitem[Sil23]{Sil}
Rob Silversmith.
\newblock Cross-ratio degrees and triangulations.
\newblock {\em \arxiv{2310.07377}}, 2023.

\bibitem[SS13]{SS}
O.~Schlotterer and S.~Stieberger.
\newblock Motivic multiple zeta values and superstring amplitudes.
\newblock {\em J. Phys. A}, 46(47):475401, 37, 2013.

\bibitem[ST21]{ST}
Bernd Sturmfels and Simon Telen.
\newblock Likelihood equations and scattering amplitudes.
\newblock {\em Algebr. Stat.}, 12(2):167--186, 2021.

\bibitem[SW05]{SW}
David Speyer and Lauren Williams.
\newblock The tropical totally positive Grassmannian.
\newblock {\em Journal of Algebraic Combinatorics}, 22(2):189--210, 2005.

\bibitem[SW21]{SWDress}
David Speyer and Lauren~K. Williams.
\newblock The positive {D}ressian equals the positive tropical {G}rassmannian.
\newblock {\em Trans. Amer. Math. Soc. Ser. B}, 8:330--353, 2021.

\bibitem[Tev07]{Tev}
Jenia Tevelev.
\newblock Compactifications of subvarieties of tori.
\newblock {\em Amer. J. Math.}, 129(4):1087--1104, 2007.

\bibitem[Var95]{Var}
A.~Varchenko.
\newblock Critical points of the product of powers of linear functions and
  families of bases of singular vectors.
\newblock {\em Compositio Math.}, 97(3):385--401, 1995.

\bibitem[Ven68]{veneziano1968construction}
Gabriele Veneziano.
\newblock Construction of a crossing-symmetric, {R}egge-behaved amplitude for
  linearly rising trajectories.
\newblock {\em Il Nuovo Cimento A (1965-1970)}, 57(1):190--197, 1968.

\bibitem[Wit04]{Wit}
Edward Witten.
\newblock {Perturbative gauge theory as a string theory in twistor space}.
\newblock {\em Commun. Math. Phys.}, 252:189--258, 2004.

\end{thebibliography}
\end{document}